\documentclass[12pt]{article}

\usepackage{amsmath,amsthm,amsfonts,amssymb,bbm,dsfont,mathrsfs,yfonts,mathalfa,bm}
\usepackage[authoryear]{natbib}
\usepackage[colorlinks,citecolor=blue,urlcolor=blue]{hyperref}
\usepackage[pdftex]{graphicx}
\usepackage{subcaption}
\usepackage{xr}
\usepackage{color}
\usepackage{enumerate,mathtools}
\usepackage{multirow,tabularx, booktabs}\graphicspath{{figures/}}
\usepackage[title]{appendix}

\usepackage[
left=1.2in,
right=1.2in,
top=1.2in,
bottom=1.2in,]{geometry}
\theoremstyle{plain}
\newtheorem{theorem}{Theorem}
\newtheorem{lemma}{Lemma}
\newtheorem{proposition}{Proposition}

\newtheorem{assumption}{Assumptions}

\newtheorem*{test*}{Hypothesis Test}
\newtheorem{innercustomthm}{Proposition}
\newenvironment{myprop}[1]
  {\renewcommand\theinnercustomthm{#1}\innercustomthm}
  {\endinnercustomthm}
  
\theoremstyle{definition}

\newtheorem{remark}{Remark}

\numberwithin{equation}{section}

\newcommand\numberthis{\addtocounter{equation}{1}\tag{\theequation}}

\newcommand{\rt}[1]{\sqrt #1}
\newcommand{\tp}{^{\top}}
\newcommand{\td}{\widetilde}

\newcommand{\norm}[1]{\lVert {#1} \rVert}
\newcommand{\normm}[1]{\left\lVert {#1} \right\rVert}

\newcommand{\weakly}{\Rightarrow}

\newcommand{\tr}{\mathrm{tr} }

\newcommand{\Var}{\mathrm{Var}}

\newcommand{\pf}[1]{\noindent{\text{#1})}}
\newcommand{\pfspace}{\vspace{2mm}}

\renewcommand{\norm}[1]{\lVert {#1} \rVert}

\newcommand{\0}{{\mathbf 0}}

\renewcommand{\aa}{{\mathbf a}}

\newcommand{\ee}{{\mathbf e}}
\newcommand{\ff}{{\mathbf f}}
\newcommand{\sigmab}{{\boldsymbol \sigma}}

\newcommand{\uu}{{\mathbf u}}
\newcommand{\vvv}{{\mathbf v}}

\newcommand{\xx}{{\mathbf x}}
\newcommand{\yy}{{\mathbf y}}
\newcommand{\zz}{{\mathbf z}}
\newcommand{\epsilonb}{\boldsymbol{\epsilon}}
\newcommand{\phib}{\boldsymbol{\phi}}

\newcommand{\cB}{{\mathcal B}}
\newcommand{\cF}{{\mathcal F}}

\newcommand{\N}{\mathbb N}
\newcommand{\C}{\mathbb C}

\newcommand{\Em}{{\mathbb E}}
\newcommand{\Pm}{{\mathbb P}}
\newcommand{\Rm}{{\mathbb R}}
\newcommand{\Z}{{\mathbb Z}}

\DeclareMathOperator*{\argmin}{argmin}
\newcommand{\zero}{\boldsymbol{0}}

\newcommand{\tGa}{\widetilde{\boldsymbol{\Sigma}}}

\newcommand{\pW}{\boldsymbol{W}}

\newcommand{\tu}{\widetilde{\boldsymbol{\epsilon}}}

\newcommand{\hu}{\widehat{\boldsymbol{\epsilon}}}

\newcommand{\pM}{\boldsymbol{M}}

\begin{document}
	\def\spacingset#1{\renewcommand{\baselinestretch}%
		{#1}\small\normalsize} \spacingset{1}
	
	\title{Spiked eigenvalues of high-dimensional sample autocovariance matrices: CLT and applications}
	\author{Daning Bi\\
		Hunan University, China\\
		\\
		Xiao Han\footnote{Corrsponding author: Professor Xiao Han, International Institute of Finance, School of Management, University of Science and Technology of China. Email Address: xhan011@ustc.edu.cn}\\
		University of Science and Technology of China, China\\
		\\
		Adam Nie\\
		The Australian National University, Australia \\
		\\
		Yanrong Yang\footnote{Four coauthors are co-first authors.}\\
		The Australian National University, Australia }
	\maketitle

	\begin{abstract}
		We establish the central limit theorem (CLT) for spiked eigenvalues of high-dimensional sample autocovariance matrices under general conditions: (1) the spiked eigenvalues are allowed to go to infinity without restrictions in divergence order; (2) the number of spiked eigenvalues and the time lag of the autocovariance matrix could be either fixed or tending to infinity. As a further statistical application, a novel autocovariance test is proposed to detect the equivalence of spiked eigenvalues for two high-dimensional time series. Simulation studies are illustrated to justify the theoretical findings. Lastly, a hierarchical clustering approach is constructed to clustering mortality data from multiple countries.  
	\end{abstract}
	\noindent%
	\vfill
	
	\newpage

	\section{Introduction}
Advances in modern technology have facilitated the collection and analysis of high-dimensional data. A major challenge of statistical inference on high-dimensional data is the well-known ``curse of dimensionality'' phenomenon \citep{Donoho2000}. Dimension reduction, which projects high-dimensional data into a low-dimensional subspace, is a natural idea to overcome the large-dimensional disaster. 
Principal component analysis (PCA) is a commonly-used dimension-reduction technique for high-dimensional independent and identically distributed (i.i.d) data, which pursuits the low-dimensional subspace that keeps the most variation of the original data. A significant and intrinsic difference between i.i.d data and time series lies in the perception that time series have temporal dependence along the sample observations. As informed in earlier literature, \citet{Box_Tiao1977, Pena_Box1987, Tiao_Tsay1989}, identifying the low-dimensional representation or common factors that drive the temporal dependence of original time series is the major purpose of dimension reduction for high-dimensional time series. \citet{LamYao2012} conduct the eigen-decomposition of autocovariance matrices and justify that, the eigenvectors corresponding to the largest eigenvalues span a subspace where the projection of the original time series reserves the most temporal covariance. In view of such a close connection, there is a need to explore the spectral properties for high-dimensional autocovariance matrices.

The major contribution of this paper is to establish the asymptotic distribution of spiked eigenvalues for high-dimensional sample autocovariance matrices. Similar to the spiked covariance model raised in \citet{Johnstone2001}, we consider a spiked autocovariance model in which the population autocovariance matrix has a few large eigenvalues, called spiked eigenvalues, that are detached from the bulk spectrum. The spiked autocovariance model could be expressed via a factor model, like \citet{LamYao2012}, in which all temporal dependence is absorbed in low-dimensional common-factor time series. Intuitively speaking, spiked eigenvalues are equal to autocovariances of common-factor time series. In view of this point, spiked eigenvalues from high-dimensional autocovariance matrices could quantify the temporal dependence reserved in low-dimensional projected time series or common-factor time series. We will work under this factor model and investigate spiked eigenvalues from a symmetrized sample autocovariance matrix, which is the product of the autocovariance matrix and its transpose.

In the context of high-dimensional sample autocovariance matrix analysis, fundamental asymptotic properties for spiked eigenvalues under moderately high-dimensional settings are available in the literature. Although \citet{LamYao2012} and \citet{LiWangYao2017} both focus on a ratio-based selection criterion for the number of factors, they essentially contribute to asymptotic properties of spiked and nonspiked eigenvalues of high-dimensional sample autocovariance matrices. In the case of strong spikiness, that is the spiked eigenvalues tending to infinity, \citet{LamYao2012} provide the rate of convergence for spiked and nonspiked eigenvalues. To be more sophisticated, \citet{LiWangYao2017} investigate the exact phase transition that distinguishes the factor part and the noise part. So their proposed ratio-based estimator is also applicable in weak spiked cases when the spiked eigenvalues are of constant order. 
More result for unspiked eigenvalues are derived in \citet{LiPanYao2015}, \citet{WangYao2015} and \citet{BoseBhattacharjee2018}. Recently, \citet{yao2021eigenvalue}, \citet{bose2021spectral} develop asymptotic properties of the smallest eigenvalues for large dimensional autocovariance matrices and their variants.

As the first main contribution of this paper, under general conditions, we establish the asymptotic normality of $\lambda_i$, the $i$-th largest spiked eigenvalue of the matrix $\widehat \Sigma_\yy(\tau) \widehat \Sigma_\yy\tp(\tau)$, where $\widehat{\Sigma}_{\yy}(\tau)$ is the lag $\tau$ sample autocovariance matrix of high-dimensional time series $\{\yy_t, t=1, 2, \ldots, T\}$ under study. 
We assume that the spiked population eigenvalues $\{\mu_i\}_{i\le K}$ diverge as $T\to\infty$ without restrictions on the diverging rates, which relaxes the specific rates used in \citet{LamYaoBathia2011} and \citet{LamYao2012}. Additionally, we also allow the number of factor $K$ to be either fixed or diverging as $T\to\infty$. This type of assumption has been made in the literature for large covariance matrices such as \citet{CaiHanPan2017}, but has not yet been incorporated into the factor model for high-dimensional time series. Furthermore, the lag $\tau$ in the autocovariance matrix $\widehat \Sigma_\yy(\tau)$ is allowed to be either fixed or diverging. Our results show that the scalings for CLTs are not of the same order. In particular, if one is interested in the eigenvalues of $\widehat \Sigma_\yy(\tau) \widehat \Sigma_\yy(\tau)\tp$ for a moderately large $\tau$, the CLT in the regime where $K \to\infty$ might provide a more accurate result than the case for a fixed $\tau$. Since we are working in a regime with less restricted assumptions on the number of factors, the lag of autocovariance, and the spikiness, as a natural trade-off, some difficulties arising in our work are worth to be noted. 

A major source of difficulty in our setting comes from less restrictions on the rate of divergence of spiked population eigenvalues $\{\mu_{i}, i=1, 2, \ldots, K\}$. We argue that the specification of the diverging speed of $\mu_i$ such as ones used in \citet{LamYaoBathia2011} and \citet{LamYao2012} entirely reduce the analysis of a high-dimensional factor model to the study of low-dimensional common-factor time series (see the remarks below Theorem \ref{theorem - 2.1}). 
While this aligns with the goals of dimension reductions in \citet{LamYaoBathia2011} and \citet{LamYao2012}, it obfuscates some interesting features otherwise seen in high-dimensional models. 
Without such restriction, the idiosyncratic noise is no longer negligible and we obtain a clearer picture of how the high-dimensional noise part accumulates and affects the location of spiked eigenvalues. 
More specifically, even though $\lambda_i$ is close to $\mu_i$ asymptotically, the convergence rate of $\lambda_i- \mu_i$ (after appropriate scaling) is in general slower than $T^{-1/2}$. In other words, we will not be able to obtain a CLT using $\mu_i$ as the centering term. What happens here is that the bias of $\lambda_i$ decays too slowly to obtain a CLT and a more accurate centering is needed. In our work, this centering term will be defined implicitly as the solution to an established equation. The phenomenon described above is common in large random matrix literature where, however, there is less emphasis on reducing high-dimensional models into low-dimensional ones \citep[see,][for example]{CaiHanPan2017}.

Besides, instead of working with the autocovariance matrix $\widehat \Sigma_\yy(\tau)$, we are dealing with the symmetrized version $\widehat \Sigma_\yy(\tau) \widehat \Sigma_\yy^{\tp}(\tau)$ in our analysis.
From the technical aspect, the matrix $\widehat \Sigma_\yy(\tau) \widehat \Sigma_\yy^{\tp}(\tau)$ could not be decomposed into a matrix with independent entries like the covariance matrix $\widehat \Sigma$ does. Therefore, the common ideas and regular techniques of some existing works in large random matrix theory such as \citet{BaiYao2008} and \citet{CaiHanPan2017} are not applicable directly in our work. Consequently, as the first work on CLT for large autocovariance matrices under less restricted assumptions, we need a new approach to establish the asymptotic normality for the empirical eigenvalues $\{\lambda_i\}$. Moreover, the approach we develop here could potentially be applied to other types of products of covariance-type matrices. 

Another important contribution of this paper is a novel autocovariance test which is built on the developed CLT for $\{\lambda_i\}$. It is well known that when the data dimension $p$ increases with sample size $T$, directly comparing and testing the equivalence of two autocovariance matrices is infeasible due to the ``curse of dimensionality''. The major idea of the proposed so-called autocovariance test is to compare the autocovariance of the low-dimensional common-factor time series. It is equivalent to testing whether spiked population eigenvalues of two high-dimensional autocovariance matrices are the same. It is worth mentioning that, as the CLT involves some unknown parameters, we propose an AR-sieve bootstrap to derive a feasible test statistic. Furthermore, the proposed test statistic is powerful under some local alternative hypotheses, which are demonstrated via theoretical results and various simulation designs. This autocovariance test is not only in its own interest 
but also motivates other statistical inferences such as statistical clustering analysis on multi-population high-dimensional time series. In this paper, we construct a new hierarchical clustering approach based on the autocovariance test. It is applied to multi-country mortality data, for which we group those countries with similar low-dimensional autocovariances. The clustering results are consistent with findings in common literature on mortality studies. 

The rest of this article is organized as follows. Section~\ref{section - setting} introduces the setting and assumptions of our work, sets up the relevant notations, and presents some preliminary results. The theoretical results of our work are given in Section~\ref{section - main clt}. 
In Section~\ref{section - location of spikes in main results} we investigate the asymptotic location of empirical eigenvalues and construct an accurate centering for these eigenvalues where technical results are collected in Appendix~\ref{section - proofs} of the Supplement Material. The CLT for empirical eigenvalues, which is the main result of our work, is given in Section~\ref{section - subsection CLT for }. The proof of the CLT is quite involved and is thus divided into a series of intermediate results collected in Appendix~\ref{CLT - proofs}, and technical lemmas are collected in Appendixes~\ref{section - clt lemmas}~and~\ref{section - resolvents} of the Supplement Material.
We give a summary of the strategy of the proof in Section~\ref{section - subsection CLT for } and explain how the intermediate results are used to obtain the CLT. Lastly, a novel autocovariance test is illustrated in Section~\ref{section - applications} as a statistical application of the proposed CLT where numerical results including simulation studies and real applications on mortality data 
are also provided in Sections~\ref{3:sec:4}~and~\ref{3:sec:5}, respectively. Technical proofs, additional numerical results of this autocovariance test including its application on the multi-country mortality data are left to the Supplement Material. 

\section{Model Setting}\label{section - setting}
We consider a high-dimensional time series with a factor model structure that appeared previously in \cite{LamYao2012,LamYaoBathia2011,LiWangYao2017}.
Let $(\yy_t)_{t=1,\ldots, T}$ be a $K+p$ dimensional stationary time series containing $K$ independent factors, observed over a period of length $T$. 
Formally we may write the time series as 
\begin{align*}
	\yy_t = L \ff_t + \epsilonb_t,\quad t=1,\ldots, T,\numberthis\label{equation - original model equation}
\end{align*}
where $(\ff_t)_{t=1,\ldots, T}$ is the $K\times T$ matrix of factors, each assumed to be a stationary times series, $L$ is the $(p+K)\times K$ factor loading matrix and $(\epsilonb_t)_t$ is a $K+p$ dimensional idiosyncratic noise series to be specified below.

It is well-known (see for example \cite{BaiLi2012}) that the factor model \eqref{equation - original model equation} is not identifiable without additional constraints on $L$ and $\ff_t$, we refer to Table 1 of \cite{BaiLi2012} for a discussion of the different setups found in the existing literature.
In our work we will assume, without any loss of generality, that $L\tp L$ is equal to a diagonal matrix and all factors are standardized, i.e. $\Em[f_{it}]=0$ and $\Em[f_{it}^2]=1$ for all $i=1,\ldots, K$ and $t=1,\ldots, T$.

We will work in the so-called high-dimensional setting where the dimension of the model $p$ and the sample size $T$ diverge at the same time while the ratio $p/T$ tends to a constant $c>0$ as $T\to\infty$. We allow the number of factors $K$ to diverge as $T\to\infty$, but impose conditions on the speed of its divergence so that the number of factors remains small in comparison to the dimension of the entire observation (see Assumption \ref{assumptions - tau fixed} and Assumption \ref{assumptions - tau div}).

Each factor $(f_{it})_t$ is assumed to be a stationary time series of the form
\begin{align*}
	f_{it} = \sum_{l=0}^{\infty} \phi_{il} z_{i, t-l}, \quad i = 1,\ldots, K, \quad t = 1, \ldots, T,
	\numberthis\label{equation - fit = timeseries}
\end{align*}
where the random variables $(z_{it})$ are i.i.d. with zero mean, unit variance and finite $(4+\epsilon)$-th moment for some small $\epsilon>0$.
Under this setup, the constraint $\Var(f_{it}) = 1$ mentioned above directly translates to the constraint $\norm{\phib_i}_{\ell_2}=1$ where $\phib_i := (\phi_{il})_l$ is the vector of coefficients for the $i$-th factor and $\norm{\cdot}_{\ell_2}$ is the usual $\ell_2$ norm.
Write $\gamma_i(\tau):= \Em[f_{i,1} f_{i, \tau+1}]$ for the lag-$\tau$ population auto-covariance of the  $i$-th factor  time series $\ff_i$.
In terms of the representation (\ref{equation - fit = timeseries}), clearly $\gamma_i(\tau)$ can be written as
\begin{align*}
	\numberthis\label{equation - definition of gammitau}
	\gamma_i(\tau) :=\Em[f_{i,1} f_{i, \tau+1}]
	= \sum_{l=0}^\infty \phi_{i,l} \phi_{i,l+ \tau}.
\end{align*}

Although the loading matrix $L$ appears in the (population) covariance and auto-covariance matrices of $\yy_t$, it does not affect the eigenvalues of the covariance and auto-covariance matrix after normalization.
Furthermore, as observed in \cite{LiWangYao2017}, under additional Gaussian assumptions on the error time series $\epsilonb_t$, the factor model can be reduced to a canonical form where
$L = \begin{pmatrix} I_K & \boldsymbol{0}_{K\times p}\end{pmatrix}\tp$. 
Under this assumption, \cite{LiWangYao2017} is able to obtain explicit results on the phase transition of the asymptotic locations of the spiked eigenvalues.
As already mentioned, we will adopt a slightly different normalization for the matrix $L$, mainly for notational convenience. Nevertheless, we can invoke similar arguments as in \cite{LiWangYao2017} to obtain a canonical form of the factor model where $L$ takes the form
\begin{align}
	L = \begin{pmatrix}
		\mathrm{diag}(\sigma_1, \ldots, \sigma_K) \\ \boldsymbol{0}_{p\times K}
	\end{pmatrix},
	\numberthis\label{canonical form}
\end{align}
where $(\sigma_1, \ldots, \sigma_K)$ is a sequence of positive real numbers.

For the completeness of our exposition, we give a detailed explanation of the simplification~\eqref{canonical form}. 
By assumption, the $(p+K)\times K$ matrix $\overline L:= L\ \mathrm{diag}(\sigma_1^{-1}, \ldots, \sigma_K^{-1})$ satisfies $\overline L \tp \overline L = I_K$ under our choice of normalization, thus there exists a $(p+K)\times p$ matrix $\underline L$ with orthonormal columns such that $\td L:= (\overline L, \underline L)$ is orthonormal. Recall from \eqref{equation - original model equation} that $\yy_t = L \ff_t + \epsilonb_t$. Define
\begin{align*}
	\zz_t:= \td L\tp \yy_t
	 & =  \begin{pmatrix}
		\overline L\tp \\ \underline L\tp
	\end{pmatrix}
	L\ff_t
	+\td L\tp \epsilonb_t
	=  \begin{pmatrix}
		\overline L\tp \\ \underline L\tp
	\end{pmatrix}
	\overline L\ \mathrm{diag}(\sigma_1, \ldots, \sigma_K)\ \ff_t
	+\td L\tp \epsilonb_t.
\end{align*}
By definition we have $\overline L\tp \overline L = I_K$ and $\underline L\tp \overline L=\0_{p}$,
therefore
\begin{align*}
	\zz_t=\td L\tp \yy_t
	=
	\begin{pmatrix}
		\mathrm{diag}(\sigma_1, \ldots, \sigma_K) \\ \boldsymbol{0}_{p\times K}
	\end{pmatrix}
	\ff_t
	+\td L\tp \epsilonb_t.
	\numberthis\label{equation - alternative model}
\end{align*}
Note that $\zz_t$ is simply the original data $\yy_{t}$ subjected to an orthonormal transformation, so intuitively the sample auto-covariance matrix of $(\zz_t)$ contains the same temporal information as that of $(\yy_t)$.
More precisely, define the sample auto-covariance matrices
\begin{align*}
	\Sigma_{\yy}(\tau):= \frac{1}{T}\sum_{t=1}^{T- \tau} \yy_{t+ \tau} \yy_t\tp, \quad
	\Sigma_{\zz}(\tau):= \frac{1}{T}\sum_{t=1}^{T- \tau} \zz_{t+ \tau} \zz_t\tp =
	\td L\tp \Sigma_\yy(\tau) \td L.
\end{align*}
It is easy to see that the spectrum of $\Sigma_{\yy}(\tau) \Sigma_{\yy}\tp(\tau)$ coincides with that of  $\Sigma_{\zz}(\tau) \Sigma_{\zz}\tp(\tau)$. Indeed, we have
\begin{align*}
	\Sigma_{\zz}(\tau)\Sigma_{\zz}\tp(\tau)
	= \td L\tp \Sigma_\yy(\tau) \td L\ \td L\tp \Sigma_\yy(\tau)\tp \td L = \td L\tp \Sigma_\yy(\tau) \Sigma_\yy\tp(\tau) \td L,
\end{align*}
where $\td L$ is orthonormal so a conjugation by $\td L$ does not affect the spectrum $\Sigma_\yy(\tau)\Sigma_\yy\tp(\tau)$.

The main goal of our work is to establish the asymptotic distribution of the spiked eigenvalues of $\Sigma_{\yy}(\tau)\Sigma_{\yy}\tp(\tau)$. Since the eigenvalues of $\Sigma_{\zz}(\tau)\Sigma_{\zz}^{\top}(\tau)$ are the same as those of the matrix $\Sigma_{\yy}(\tau)\Sigma_{\yy}^{\top}(\tau)$, it suffices to consider $\Sigma_{\zz}(\tau)\Sigma_{\zz}\tp(\tau)$ instead of $\Sigma_{\yy}(\tau)\Sigma_{\yy}\tp(\tau)$. That is, we may without any loss of generality assume that
\begin{align*}
	\yy_t
	=
	\begin{pmatrix}
		\mathrm{diag}(\sigma_1, \ldots, \sigma_K) \\ \boldsymbol{0}_{p\times K}
	\end{pmatrix}
	\ff_t + \td L\tp \epsilonb_t.
\end{align*}
Finally, under the assumptions  $\epsilonb_t$ is orthogonally invariant, the transformed error $\td L\tp \epsilonb_t$ is equal in distribution to $\epsilonb_t$. Under this assumption, we have
\begin{align*}
	\yy_t
	\stackrel{\mathrm{dist.}}{=}
	\begin{pmatrix}
		\mathrm{diag}(\sigma_1, \ldots, \sigma_K) \\ \boldsymbol{0}_{p\times K}
	\end{pmatrix}
	\ff_t + \epsilonb_t
	\numberthis\label{equation - canonical form of y}
\end{align*}
and we may take this as the canonical form of the factor model (\ref{equation - original model equation}). Motivated by these observations, we will work under the canonical form \eqref{equation - canonical form of y}.

\subsection{Assumptions}
The asymptotic properties of empirical spiked eigenvalues $\lambda_i, i=1, 2, \ldots, K$ mainly depend on the following five types of parameters:
\begin{enumerate}
\item[(1).]\textbf{Factor strength} $(\sigma_i^2)_{i=1}^K$, which are the variances of factors (before normalization).  
\item[(2).]\textbf{Spikiness} $(\mu_i^2)_{i=1}^K$, which are spiked eigenvalues of the population matrix $\Sigma_{\yy}\Sigma_{\yy}^{\top}$. 
\item[(3).]\textbf{Time lag} $\tau$, which is allowed to be fixed or tend to infinity. 
\item [(4).]\textbf{The number of factors} $K$, which could be fixed or tending to infinity. 
\item[(5).]\textbf{The dimension and sample size}: $p$ and $T$, which tend to infinity simultaneously. 
\end{enumerate}
Depending on the choices made on these parameter, the model could exhibit a wide range of different behaviours. 
Below we will detail the assumptions we made relating to these five types of parameters, together with some discussion and justifications on these choices. 
  
First, under the canonical representation \eqref{equation - canonical form of y}, $\sigma_1, \ldots, \sigma_K$ are in fact the standard deviations of the original factors.
We will assume that the factors are ``strong'', i.e. $\sigma_i\to\infty, i=1, 2, \ldots, K$ as $p\to\infty$, which is a common assumption in the factor modelling literature. We note that \cite{LamYao2012} and many subsequent results explicitly specified the rate at which $\sigma_i$ diverges, for instance \cite{LamYao2012} assumed $\sigma_i^2\sim p^{1- \delta}$ where $\delta\in[0,1]$ is a fixed constant. This type of assumption can be understood as essentially reducing the problem to a low dimensional setting, in the sense that the accumulated effects of the high dimensional noise $\epsilonb$ is still small in comparison to the signal strength $\sigma_i$. We will lift this restriction and consider a very general setting where $\sigma_i$ can diverge at any rate instead of specific functions of $T$. This relaxation brings numerous technical difficulties and required us to carry out detailed analysis of high dimensional random matrices.

Next, we note that the spiked eigenvalues $\mu_{i, \tau}, i=1, 2, \ldots, K$ of the population matrix $\Sigma_{\yy}(\tau)\Sigma_{\yy}^{\top}(\tau)$ are closely related to the factor strength as well as the temporal dependence of the common factors. Recall $\gamma_i(\tau) :=\Em[f_{i,1} f_{i, \tau+1}]$ from \eqref{equation - definition of gammitau}. Under the canonical form \eqref{equation - canonical form of y}, the (population) lag-$\tau$ auto-covariance function for each time series $(y_{it})_t$ can be written as
\begin{align*}
	\mu_{i,\tau} := \left(\Em[y_{i,t}y_{i,t+\tau}]\right)^2
	= \sigma_i^4 \gamma_i(\tau)^2, \quad i=1,\ldots, K,\quad \tau\ge 0.
	\numberthis\label{equation - definition of mu n tau}
\end{align*}

To the extent of our knowledge, in the existing literature the parameter $\tau$ is assumed to be a given constant. This classical setting is included in our work under Assumption \ref{assumptions - tau fixed}. 
We will also study a novel setup where we allow $\tau$ diverge as $T\to\infty$. This new setup is motivated by the observation that the theoretical results under the classical fixed $\tau$ settings can be inaccurate when $\tau$ is large. As will be shown, the quantity $\gamma_n(\tau)$ will appear in the characterisation of the asymptotic location of the eigenvalues as well as the variance term in the central limit theorem. For $\tau$ large, $\gamma_n(\tau)$ can be very small, in which case our second setting provides a more accurate description on the effect of large $\tau$.

In the case where $\tau$ is fixed, we will assume without any loss of generality that the sequence $(\mu_{i,\tau})_i$ is arranged in decreasing order, at least for large $T$. Furthermore, we assume impose a spectral gap condition that guarantees $\{\mu_{i, \tau}\}$ is well separated, i.e. there exists $\epsilon>0$ such that $\mu_{i, \tau}/\mu_{i+1, \tau}> 1+ \epsilon$ for all $i$ and $\tau$. This assumption is standard (see e.g. \cite{CaiHanPan2017}) and ensures that the empirical eigenvalues are separated asymptotically.

In the case where $\tau$ is allowed to vary with $T$,
it is too restrictive to assume that such an ordering on $\mu_{i, \tau}$ exists for all $\tau\ge 0$. For example, suppose the $(y_{1t})_t$ has a large variance $\sigma_1^2$ but a very rapidly decaying auto-covariance function $\gamma_1(\cdot)$, while $(y_{2t})_t$ has a smaller variance but a slow decaying auto-covariance function.
Then we can easily have $\mu_{1, 1}> \mu_{2,1}$ as well as $\mu_{1, \tau}< \mu_{2,\tau}$ for a larger $\tau$ so the assumption $\mu_{1, \tau}>\mu_{2, \tau}$ for all $\tau$ is unrealistic.
Instead, we will assume that the sequence $(\mu_{i,\tau})_i$ is well separated only asymptotically, i.e.
we assume there exists $\tau_0$
large enough and $\epsilon>0$ such that
\begin{align*}
	\mu_{i, \tau}/\mu_{i+1, \tau} >1 + \epsilon, \quad \forall \tau>\tau_0, \quad i =1,\ldots, K.
\end{align*}
For simplicity and transparency of our results, we will assume that all $\gamma_i(\tau), i=1, 2, \ldots, K$ decay at the same speed asymptotically, i.e.  $\gamma_i(\tau)/\gamma_j(\tau)<C_1$ for $i,j=1,\ldots,K$ and some constant $C_1$. This implies that the $\mu_{i,\tau}$'s are of the same order as well and a comparison between them is more reasonable.

Finally, following \cite{LiWangYao2017} we will assume that the error time series $(\epsilonb_t)_t$ is standard Gaussian. This ensures a particularly transparent model \eqref{equation - canonical form of y} and th{}eoretical results. We remark that the Gaussianity assumption can be significantly weakened. For instance, from the discussion leading up to \eqref{equation - canonical form of y}, we see that we may still obtain \eqref{equation - canonical form of y} under the assumption that $\epsilonb_t$ is orthogonally invariant. However, as a trade-off, we would have to impose other less intuitive assumptions on $\epsilonb_t$ to ensure the concentration of certain quadratic forms to obtain a CLT. An alternative approach is to invoke a Linderberg type of argument and bound the difference between the Gaussian and non-Gaussian model as a sum of low rank perturbations, see for example \cite{bao2015universality,lee2016tracy}. This type of argument, commonly used in studying the universality of random matrices, is more suitable for a standalone work and we will not pursue such an extension here.

For clarity and the convenience of the reader we summarize our settings into the following sets of conditions which will be referred to in later parts of the paper.
\begin{assumption}\label{assumptions}
	\begin{enumerate}
		\item \label{assumption - order of p, T, K}
		$p, T\to\infty$ and $p/T\to c\in (0, +\infty)$.
		
		\item \label{assumption - diverging eigenvalues}
		$\sigma_i\to\infty$  and $\sigma_i = o(\sigma_j^2)$ for all $i,j=1,\ldots, K$.
		
		\item
		$(z_{it})_{1\le i\le K, 1-l \le t\le T+1}$ is independent, identically distributed with $\Em[z_{it}] =0$, $\Em[z_{it}^2]=1$ and uniformly bounded $(4+\epsilon)$-th moment for some $\epsilon>0$.
		
		\item
		$(\epsilon_{it})_{1\le i\le p+K, 1 \le t\le T+1}$ is i.i.d. standard Gaussian.
		
		\item
		$\sup_{i}\norm{\phib_i}_{\ell_1}<\infty$.

	\end{enumerate}
\end{assumption}

Part (a) and (b) of Assumption \ref{assumptions} capture our asymptotic regime where $p$ diverges at the same rate as $T$ and the strength of all factors diverges at some non-specific rates. 
In factor model analysis literature it is common to require all factors diverge at the same rate, see \citep{BaiNg2002}. We will not impose such a restriction, instead we opt for a mild assumption on the relative size of factors not being too different. 

Moment conditions such as (c) of Assumption \ref{assumptions} are standard in the random matrix literature; see for instance \cite{BaoDingWang2018,CaiHanPan2017,LiPanYao2015,WangYao2016,WangYao2017}. The normality assumption in (d) is discussed above and finally condition (e) is very standard in the time series literature, see \cite{BrockwellDavisFienberg1991}. For instance, condition (e) is satisfied by an auto-regressive moving average process written in the form \eqref{equation - fit = timeseries}.

The following two sets of assumptions encapsulate the two asymptotic schemed discussed above relating to the parameter $\tau$. Our main results are formulated in such a way that they hold under either set of assumptions.

\begin{assumption}\label{assumptions - tau fixed}
	\begin{enumerate}

		\item $\tau$ is a fixed, positive integer. 

		\item $K = o\left(T^{1/4}\right)$ as $T\to\infty$.

		\item the set of parameters $(\sigma_i)_{i=1}^K$ satisfy
		\begin{align*}
			\frac{K \sigma_i}{\sigma_j^{3/2}} = O(1), \quad
			\frac{K^2 \sigma_i^2}{\sigma_j^2} = o(\rt T), \quad \forall i,j=1,\ldots, K,
		\end{align*}

		\item the sequence $(\mu_{1, \tau}, \ldots, \mu_{K, \tau})$ is arranged in decreasing order and there exists $\epsilon>0$ such that $\mu_{i, \tau}/ \mu_{i+1, \tau}>1+\epsilon$ for all $i=1, \ldots, K-1$.
	\end{enumerate}
\end{assumption}

\begin{assumption}\label{assumptions - tau div}
	\begin{enumerate}
		\item $\tau$ is a positive integer and
		      $\tau\to\infty$ as $T\to\infty$.

		\item $K = o\left(T^{1/4}\right)$ and $\gamma_i(\tau)^{-1} = o(\rt T)$ as $T\to\infty$.

		\item the set of parameters $(\sigma_i)_{i=1}^K$ and $(\gamma_i(\tau))_i$ satisfy
		\begin{align*}
			\frac{K \sigma_i}{\sigma_j^{3/2} \gamma_{j}(\tau)^{1/5}} = O(1), \quad
			\frac{K^2 \sigma_i^2}{\sigma_j^2 \gamma_j(\tau)^2} = o(\rt T), \quad \forall i,j=1,\ldots, K,
		\end{align*}
		\item
		      there exists $\tau_0$ large enough and some $\epsilon>0$ such that
		      $\mu_{i, \tau}/ \mu_{i+1, \tau}>1+\epsilon$ for all $i=1, \ldots, K-1$ and $\tau>\tau_0$.
	\end{enumerate}
\end{assumption}
As already mentioned, we do not impose specific restrictions on the speed of divergence of $\sigma_i$, nor do we assume all spikes are of the same size. Instead, it suffices to impose mild restrictions on the relative sizes of the spikes. As such, condition these assumptions are trivially satisfied when the spike sizes are comparable. The exponents appearing in the assumptions are chosen for convenient and are likely suboptimal. We will leave the optimization of these exponents to future work.  Finally, we note that Assumption~\ref{assumptions - tau div} is strictly stronger than Assumption~\ref{assumptions - tau fixed}, since $\gamma_j(\tau)^{-1}\to\infty$ as $T\to\infty$.

\subsection{Notations and Preliminaries}
In our exposition and proofs, we will often encounter various resolvent matrices, which capture the spectral information of the random matrices we are studying.
Since we are constantly dealing with many different matrices,  assigning to each a different letter will easily exhaust the alphabet.
Instead, we adopt some non-standard notations for matrices and sub-matrices. Write $(a_{ij})$ for a matrix where the $(i,j)$-th entry is equal to $a_{ij}$. For such a matrix $(a_{ij})$, we  will write
\begin{align*}
	\aa_{[i: j], [k: l]} := \begin{pmatrix}
		a_{ik} & \ldots & a_{il}
		\\
		\vdots & \ddots & \vdots
		\\
		a_{jk} & \ldots & a_{jl}
	\end{pmatrix}
\end{align*}
for a specified sub-matrix. Similarly we will write
$\aa_{i, [j:k]}$ and $\aa_{[i:j], k}$ for the column vectors $(a_{ij}, \ldots, a_{ik})\tp $ and $(a_{ik},\ldots, a_{jk})\tp $ respectively.

First, we introduce notations for some of the more important random matrices in our study.
We denote

\begin{align*}
	x_{it} = \sigma_i f_{it} +  \epsilon_{it}, i=1, \ldots, K, t=1,\ldots, T
	\numberthis\label{equation - def of xit}
\end{align*}
and write $\xx=(x_{it})$, 
\begin{align*}
	\numberthis\label{equation - definition of xx}
	X_0 := \frac{1}{\rt T} \xx_{[1:K], [1:T-\tau]},             &
	\quad
	X_\tau :=\frac{1}{\rt T} \xx_{[1:K], [\tau +1 : T]},
	\\
	E_0 := \frac{1}{\rt T} \epsilonb _{[K+1:K+p], [1:T- \tau]}, & \quad E_\tau := \frac{1}{\rt T} \epsilonb _{[K+1:K+p], [\tau+1:T]},
\end{align*}
for matrices used later that contain the factors and noises in our model. We will also write
\begin{align*}
	\numberthis\label{equation - definition of yy}
	Y_0
	:=\frac{1}{\rt{T}} \yy_{[1:p+K], [1:T- \tau]}, &
	\quad
	Y_{\tau}
	:=\frac{1}{\rt{T}} \yy_{[1:p+K], [\tau+1:T]} ,
\end{align*}
i.e. we have $Y_0 =
	(X_0\tp, E_0\tp )\tp$
and $Y_{\tau} =(X_\tau\tp, E_\tau\tp )\tp$.
For an integer $\tau> 0$, the lag-$\tau$ sample auto-covariance matrix of $\yy_t$ can then be written as
\begin{align*}
	\widehat\Sigma_\tau & :=  \frac{1}{T} \sum_{t=1}^{T- \tau} \yy_{t+\tau} \yy_{t}\tp
	=
	\begin{pmatrix}
		X_{\tau} \\ E_{\tau}
	\end{pmatrix}
	\begin{pmatrix}
		X_0 \\E_0
	\end{pmatrix}\tp
	=
	\begin{pmatrix}
		X_{\tau}X_0\tp & X_{\tau} E_0\tp \\
		E_{\tau}X_0\tp & E_{\tau} E_0\tp
	\end{pmatrix}.
\end{align*}

Next, we introduce resolvent matrices which are central to the study of spectral properties of random matrices. Most of our results rely on certain bilinear forms formed using the resolvents.
For $a\in\Rm$ outside of the spectrum of the matrix $E_\tau \tp E_\tau E_0\tp E_0$ write
\begin{align*}
	\numberthis \label{equation - definition of R}
	R({a})  := (  I_{T- \tau} - a^{-1} E_\tau \tp E_\tau E_0\tp E_0)^{-1} = a(a - E_\tau \tp E_\tau  E_0\tp E_0)^{-1}
\end{align*}
for the (scaled) resolvent of $E_\tau \tp E_\tau  E_0\tp E_0$ at $a$.
The resolvent $R(a)$ satisfies
\begin{align*}
	\numberthis \label{equation - first resolvent identity}
	R(a)= I_{T- \tau} + a^{-1} R(a)  E_\tau \tp E_\tau E_0\tp E_0,
\end{align*}
which follows from rearranging $R(a)(I_{T- \tau} - a^{-1} E_\tau \tp E_\tau E\tp E) = I_{T- \tau}$.
Using the identity
\begin{align*}
	\numberthis\label{equation - bullet pencil identity}
	A(\lambda I - BA)^{-1} = (\lambda I - AB)^{-1}A
\end{align*}
we may also obtain the following identities
\begin{align*}
	R(a)  E_\tau \tp E_\tau  = E_\tau \tp E_\tau  R(a) \tp,
	\quad
	E_0\tp E_0 R(a)  =  R(a) \tp  E_0\tp E_0.
	\numberthis\label{equation - bullet pencil identity for REE}
\end{align*}
In our analysis we will constantly be dealing with certain quadratic forms involving matrices $X_0, X_\tau$, $E_0,E_\tau$ and the resolvent $R(a)$. To simplify notations we will write
\begin{align*}
	\numberthis\label{equation - definition of A and B}
	A(a):= \frac{1}{\rt a}{ X_0 R(a) X_\tau\tp},
	\quad
	& B(a):= \frac{1}{a}{ X_\tau  E_0\tp E_0 R(a) X_\tau\tp},
	\\
	\overline Q(a) := I_K -  a^{-1}X_\tau  E_0\tp E_0 R(a)  X_\tau \tp ,
	\quad
	& Q(a) := I_K - a^{-1}X_0 R(a) E_\tau \tp E_\tau  X_0\tp.
	\numberthis\label{equation - definition of Q}
\end{align*}
For any $a$ outside the spectrum of the matrix $X_0 R(a) E_\tau \tp E_\tau  X_0\tp$, the matrix $Q(a)$ defined above is invertible and similar to \eqref{equation - first resolvent identity}, we have
\begin{align*}
	Q(a)^{-1} = I_K + \frac{1}{a}Q(a)^{-1} X_0 R(a) E_\tau\tp E_\tau X_0\tp.
	\numberthis\label{equation - first resolvent identity for Q}
\end{align*}

For two sequences of positive numbers $(a_n)$ and $(b_n)$, we write $a_n\lesssim b_n$ if there exists a constant $c>0$ such that $a_n\le c b_n$.
We write $a_n\asymp b_n$ if $a_n\lesssim b_n$ and $b_n\lesssim a_n$ hold simultaneously.
A sequence of events $(F_n)$ is said to hold with high probability if there exists constants $c,C>0$ such that $\Pm(F_n^c)\le C n^{-c}$.
The operator and Hilbert-Schmidt norms of a matrix $M$ are denoted by $\norm{M}$ and $\norm{M}_{F}$ respectively, and we write $\norm{(a_n)}_{\ell^p}$ for the $\ell_p$ norm of a sequence $(a_n)$.
We will write $(\ee_i)_{i=1}^n$ for the standard orthonormal basis of Euclidean space $\Rm^n$, often without specifying the dimension $n$.

We will use the usual $o_p$ and $O_p$ notations for convergence in probability and stochastic compactness. For $p\ge 1$, we will write $o_{L^p}$ and $O_{L^p}$ for convergence to zero and boundedness in $L^p$, i.e. for a sequence of random variables $(X_n)_n$ and real numbers $(a_n)$, we write $X_n = O_{L^p}(a_n)$ if $\Em|X_n/a_n|^p = O(1)$ and $X_n = o_{L^p}(a_n)$ if $\Em|X_n/a_n|^p = o(1)$. For matrices $(A_n)$ we will write $A_n = O_{p, \norm{\cdot}}(a_n)$ if $\norm{A_n} = O_p(a_n)$.

Throughout the paper we will make use of certain events of high probability.
Define
\begin{align*}
	 & \cB_0 := \left\{ \norm{E_0\tp E_0}+\norm{E_\tau\tp E_\tau} \le 4\left(1+ \frac{p}{T}\right) \right\},
	\\
	 & \cB_1 := \left\{ \norm{X_0\tp X_0} + \norm{X_\tau\tp X_\tau} \le 2\sum_{i=1}^K\sigma_i^2  \right\}
	\numberthis\label{equation - high prob event}
\end{align*}
and
$\cB_2:= \cB_0\cap \cB_1$.
We first state a preliminary result showing that these events happen with high probability as $T\to\infty$. The proof will be given in Appendix \ref{section - proofs}.
\begin{lemma}\label{lemma - high prob event}
	Under Assumption \ref{assumptions} and either Assumption \ref{assumptions - tau fixed} or \ref{assumptions - tau div}, we have
	\begin{enumerate}

		\item $\cB_0$ holds with probability  $\Pm(\cB_0)=1- o(T^{-l})$ for any $l\in\N^+$
		as $T\to\infty$.
		      \label{lemma - high prob event - EE}

		\item \label{lemma - high prob event XX}
		      For $k=1,2$, $\cB_k$ holds with probability $\Pm(\cB_k) = 1 - O(KT^{-1})$ as $T\to\infty$.
	\end{enumerate}
\end{lemma}
As an immediate consequence of this lemma and (b) of Assumption \ref{assumptions}, we have
\begin{align*}
	\norm{E_0\tp E_0}+\norm{E_\tau\tp E_\tau} = O_p(1),
	\quad
	\norm{X_0\tp X_0} + \norm{X_\tau\tp X_\tau} = O_p(K\sigma_1^2).
	\numberthis\label{equation - high prob event Op}
\end{align*}
Furthermore,
we observe that under the event $\cB_0$, for any sequence $(a_T)_T$ such that $a_T\to\infty$,  the matrix
$I_{T- \tau} - a_T^{-1} E_\tau \tp E_\tau E_0\tp E_0$ is eventually invertible. Moreover, we note that under $\cB_0$ we have $\norm{a_{T}^{-1} E_\tau\tp E_\tau E_0\tp E_01_{\cB_0}} \le  4 a_{T}^{-1} (1+ p/T) = O(a_{T}^{-1})$, which is a non-random  upper-bound. By the reverse triangle inequality we immediately have $\norm{I_{T- \tau} - a_T^{-1} E_\tau \tp E_\tau E_0\tp E_0 1_{\cB_0}}\ge 1- O(a_{T}^{-1})$ and therefore
\begin{align*}
	\norm{R(a_T)1_{\cB_0}} = 1+ o(1),
	\quad
	\norm{R(a_T)} = 1+ o_p(1),  \numberthis\label{equation - norm of R 1cB}
\end{align*}
where the definition of $R(\cdot)$ is in (\ref{equation - definition of R}). 

Similarly, under the event $\cB_2$ the matrix $Q(a_T)$ is eventually invertible as $a_T\to\infty$ and
\begin{align*}
	\norm{Q(a_T)^{-1}1_{\cB_2}} = 1+ o(1),
	\quad
	\norm{Q(a_T)^{-1}} = 1+ o_p(1). \numberthis\label{equation - norm of Q 1cB}
\end{align*}
Finally, let $\cF_p$ be the $\sigma$-algebra generated by the noise time series $(\epsilonb_t)$, i.e.
\begin{align*}
	\cF_p := \sigma\big(\{\epsilon_{it}, i=K+1, \ldots, K+p, t=1, \ldots, T\}\big).
	\numberthis\label{equation - definition of sigalg CFp}
\end{align*}
We will often take expectations conditional on the noise series, in which case we shall write
\begin{align*}
	\underline\Em[\ \cdot \ ] := \Em[\ \cdot\ |\cF_p]
	.
	\numberthis\label{equation - Emcond}
\end{align*}

\section{Main results} \label{section - main clt}
Write $\lambda_{n, \tau}$ for the $n$-th largest spiked eigenvalue of the symmetrized lag-$\tau$ sample auto-covariance matrix $\widehat\Sigma_\tau  \widehat\Sigma_\tau \tp $.  The main goal of our work is to establish the asymptotic normality of $\lambda_{n, \tau}$ for $n\le K$ after appropriate centering and scaling.
We will first in Section~\ref{section - location of spikes in main results} establish the asymptotic location of the eigenvalue $\lambda_{n, \tau}$ as well as identify the correct centering for $\lambda_{n, \tau}$ in order to obtain a central limit theorem. 
The proof will be presented in Appendix \ref{section - proofs} of the supplement.
The central limit theorem itself, which is the main result of our work, is stated in Theorem~\ref{theorem - CLT} of Section~\ref{section - subsection CLT for }.

Due to its length, the proof of Theorem~\ref{theorem - CLT} will be divided into a series of propositions and technical lemmas, which are collected in Appendixes \ref{CLT - proofs} to \ref{section - resolvents} of the supplement. For the convenience of the reader, we will summarize the strategy of the proof of Theorem~\ref{theorem - CLT} and explain how the intermediate results are used in Section~\ref{section - subsection CLT for }.

\subsection{Location of Spiked Eigenvalues}\label{section - location of spikes in main results}
We first show that the spiked eigenvalue $\lambda_{n, \tau}$ is close to its population counterpart $\mu_{n, \tau}$ asymptotically. This will in particular give the asymptotic order of $\lambda_{n, \tau}$ as $T\to\infty$.
\begin{theorem}\label{theorem - 2.1}
	Under Assumption~\ref{assumptions} and either Assumption~\ref{assumptions - tau fixed} or \ref{assumptions - tau div}, we have
	\begin{align*}
		\frac{\lambda_{n, \tau}}{\mu_{n, \tau}} -1
		=
		O_p\left(\frac{1 }{\gamma_n(\tau)\rt T }\right) +  O_p\left(\frac{K \sigma_1^2}{\sigma_n^4 \gamma_n(\tau)^2 }\right),\quad n=1,\ldots, K.
		\numberthis\label{equation in theorem 2.1}
	\end{align*}
	where $\mu_{n, \tau}$ and $\gamma_n(\tau)$ are defined in \eqref{equation - definition of mu n tau} and \eqref{equation - definition of gammitau} respectively.
\end{theorem}

\begin{remark}\label{remark - remark on location of eigenvalue}
	A closer inspection of the convergence rate in Theorem~\ref{theorem - 2.1} shows that $\mu_{n, \tau}$ is not the appropriate centering constant for $\lambda_{n, \tau}$ for the purpose of obtaining a CLT.
	The first term on the right-hand side of \eqref{equation in theorem 2.1} can indeed be shown to be asymptotically normal at a scaling of $\gamma_n(\tau) \rt T$, which is the same scaling as our main result in Theorem~\ref{theorem - CLT}.
	However, the second term in \eqref{equation in theorem 2.1} is in general not negligible after scaling by $\gamma_n(\tau) \rt T$ unless some restrictions on the rate of divergence of $\mu_{n, \tau}$ are imposed. 
	
	If we were to impose stronger assumptions on the rate of $\mu_{n, \tau}$, for example assuming the rate $\mu_{n, \tau}\asymp p^{1- \delta}$ required in \cite{LamYao2012}, then the second term in \eqref{equation in theorem 2.1} indeed becomes negligible. Under such assumptions the $K$ spiked eigenvalues of the $(p+K)\times (p+K)$ dimensional matrix $\widehat \Sigma_\tau \widehat \Sigma_\tau\tp$ are extremely close to the eigenvalues of the $K\times K$ matrix $X_\tau X_0\tp X_0 X_\tau\tp$, as can be deduced from the proof of Theorem~\ref{theorem - 2.1}. The analysis of the matrix $\widehat \Sigma_\tau \widehat \Sigma_\tau\tp$ reduces to the analysis of the much simpler matrix $X_\tau X_0\tp X_0 X_\tau\tp$, which is essentially a low-dimensional problem. In this case, the derivation of a CLT is much easier. 
\end{remark}
As can be seen from the proof of Theorem~\ref{theorem - 2.1}, the second term in \eqref{equation in theorem 2.1} represents the bias incurred when estimating $\mu_{n, \tau}$ using $\lambda_{n, \tau}$. In order to obtain a CLT, we need a more accurate centering term for $\lambda_{n, \tau}$ to remove or reduce this bias. This centering term, which we write as $\theta_{n, \tau}$, will be defined implicitly as the  unique solution to the equation
\begin{align*}
	\numberthis\label{equation - definition of thetak}
	1= \Em[B ({\theta_{n, \tau}}) _{nn}1_{\cB_0}] - \Em[A ({\theta_{n, \tau}}) _{nn}1_{\cB_0}]^2 \Em[Q ({\theta_{n, \tau}})^{-1} _{nn}1_{\cB_2}],
\end{align*}
where the matrices $B(a)$, $A(a)$ and $Q(a)$ are defined in \eqref{equation - definition of A and B} and \eqref{equation - definition of Q} for $a\in\Rm$.

To make this definition rigorous, we start with Proposition~\ref{proposition - solution for theta} which shows that \eqref{equation - definition of thetak} indeed has a unique solution for $T$ large enough. Furthermore, this solution is shown to exist in some small interval containing $\mu_{n, \tau} = \sigma_n^4 \gamma_n(\tau)^2$. This in particular establishes the asymptotic order of $\theta_{n, \tau}$.

\begin{proposition}\label{proposition - solution for theta}
	Suppose Assumption~\ref{assumptions} and either Assumption~\ref{assumptions - tau fixed} or Assumption~\ref{assumptions - tau div} hold.
	Fix $n\in \{1, \ldots, K\}$ and let $\epsilon \in (0,1)$ be an arbitrary constant not related to $p,T$.
	Then there exists $T_0$ large enough such that for $T>T_0$, the function
	\begin{align*}
		a\mapsto g(a) = 1- \Em[B(a) _{nn}1_{\cB_0}] - \Em [A(a)_{nn}1_{\cB_0}]^2 \Em [Q(a)^{-1}_{nn}1_{\cB_2}]
	\end{align*}
	has a unique root in the interval
	\begin{align*}
		\numberthis\label{equation - interval for a around true theta}
		\sigma_n^4 \gamma_n(\tau)^2 [1- \epsilon	 , 1+ \epsilon].
	\end{align*}
\end{proposition}

\subsection{Central Limit Theorem for Spiked Eigenvalues}\label{section - subsection CLT for }
The constant $\theta_{n, \tau}$ defined in \eqref{equation - definition of thetak} turns out to be the appropriate centering constant for $\lambda_{n, \tau}$, in the sense that the second term in \eqref{equation in theorem 2.1} becomes negligible after centering by $\theta_{n, \tau}$. We are ready to state the main result of our work.
Define
\begin{align*}
	\delta_{n,\tau} := \frac{\lambda_{n, \tau} - \theta_{n, \tau}}{\theta_{n, \tau}} = \frac{\lambda_{n, \tau}}{\theta_{n, \tau}}-1.
\end{align*}
\begin{theorem}\label{theorem - CLT}
	Under Assumption~\ref{assumptions} and either Assumption~\ref{assumptions - tau fixed} or \ref{assumptions - tau div}, we have
	\begin{align*}
		\rt T\frac{\gamma_n(\tau)}{2  v_{n, \tau}}\delta_{n,\tau} \weakly N(0,1), \ \ \ n=1, 2, \ldots, K
	\end{align*}
	where
	$v_{n, \tau}$ is defined by
	\begin{align*}
		\numberthis\label{equation - definition of v i tau}
		v^2_{n,\tau}:= \frac{1}{T}\Var(\ff_{n0}\tp \ff_{n \tau}) = &
		\sum_{|k|< T - \tau} \left(1 - \frac{|k|}{ T - \tau}\right)
		u_k,
	\end{align*}
	and $(u_k)_{|k|<T- \tau}$ is a sequence of constants given by
	\begin{align*}
		u_k:=u_{nk}:= \gamma_n(k)^2 + \gamma_n(k+ \tau) \gamma_{n}(k- \tau)
		+ & (\Em[z_{11}^4]-3)\sum_{l=0}^\infty \phi_{n,l} \phi_{n,l+\tau} \phi_{n, l+k} \phi_{n, l+k+ \tau}.
	\end{align*}
\end{theorem}
\begin{remark}\label{remark - remark on clt}
	We remark that for generality as well as the tidiness of presentation we choose to formulate Theorem~\ref{theorem - CLT} in a form that holds under either one of Assumption~\ref{assumptions - tau fixed} and \ref{assumptions - tau div}. A closer inspection shows that the two cases are quite different.
	In the case where $\tau\to\infty$, we observe that $\gamma_n(\tau)\to 0$ while the term $v_{n, \tau}^2$ defined by \eqref{equation - definition of v i tau} can easily be shown to be bounded from zero. This implies that the scalings of CLT in the two cases are not of the same order.
	In the case where $\tau$ is fixed, the variance of $\rt T\delta_{n, \tau}$, which is equal to $4 v_{n, \tau} \gamma_n(\tau)^{-2}$,  is bounded both from above and away from zero from below. On the other hand, when $\tau\to\infty$, the variance of $\rt T\delta_{n, \tau}$ tends to infinity at a speed of $\gamma_{n}(\tau)^{-1}$ while $\nu_{n, \tau}$ remains bounded.

	This result might seem surprising since $\delta_{n, \tau} = (\lambda_{n, \tau}- \theta_{n, \tau})/\theta_{n, \tau}$ is already normalized in an obvious way so one might expect $\delta_{n, \tau}$ to be of order $T^{-1/2}$. One might be tempted to draw the conclusion that $\lambda_{n, \tau}$ is less accurate of an estimator of $\theta_{n, \tau}$ for larger values of $\tau$, since the variance of $\delta_{n, \tau}$ increases with $\tau$. However, the exact opposite is true here.
	Since $\theta_{n, \tau}\asymp \sigma_n^4 \gamma_n(\tau)^2$ by Proposition~\ref{proposition - solution for theta}, this implies that in fact $\lambda_{n, \tau} - \theta_{n, \tau} = O_p(\sigma_n^4 \gamma_n(\tau) T^{-1/2})$, which is faster than the rate $\lambda_{n, \tau} - \theta_{n, \tau} = O_p(\sigma_n^4 T^{-1/2})$ obtained in the case where $\tau$ is fixed.
	In practical situations where we deal with the auto-covariance matrix with a larger $\tau$, the CLT under Assumption~\ref{assumptions - tau div} provides a much more accurate convergence speed and asymptotic variance than using fixed $\tau$ results.
\end{remark}

\subsubsection*{Strategy of the proof}
The initial step is to derive an expression for the eigenvalue $\lambda= \lambda_{n, \tau}$ and the related quantity $\delta= \delta_{n, \tau}$.
In general, the eigenvalue $\lambda$ of the matrix $\widehat \Sigma_\tau \widehat \Sigma_\tau\tp$, in general, depends on its entries in complicated and non-linear ways. We take an approach commonly seen in the random matrix literature (e.g. \cite{BaiYao2008,CaiHanPan2017,LiWangYao2017}) and express $\delta$ as the solution to an equation involving the determinant of certain random matrices.
This is established in Proposition~\ref{proposition - detequation }:
\begin{myprop}{2}
	Suppose Assumption~\ref{assumptions} and either Assumption~\ref{assumptions - tau fixed} or \ref{assumptions - tau div} hold. Then the ratio $\delta$ is the solution to the following equation
\begin{align*}
	\det \left(   M   + \frac{\delta}{\theta} X_\tau  X_0\tp   X_0   X_\tau\tp   + \delta o_{p,\norm{\cdot}}(1) \right) =0. 
 \numberthis\label{equation - detequation 2}
\end{align*}
where 
\begin{align*}
	\numberthis\label{equation - definition of M}
	M  := I_K - \frac{1}{\theta}X_\tau  E_0 \tp E_0  R X_\tau \tp  - \frac{1}{\theta} X_\tau  R \tp X_0 \tp Q^{-1} X_0  R X_\tau \tp.
\end{align*}
\end{myprop}
The main idea is then to apply Leibniz's formula to compute this determinant.
In doing so we will express $\delta$ as a polynomial function of the entries of the matrices $M$ and $\theta^{-1} X_\tau  X_0\tp  X_0   X_\tau\tp$ plus many higher order terms.
After estimating the terms in this polynomial, it can be shown that the asymptotic normality of the ratio $\delta$ eventually follows from the asymptotic normality of the $n$-th diagonal entry $M$. This argument is carried out in the proof of Theorem~\ref{theorem - CLT}. More specifically, we establish the CLT, it suffices to (a) show that
\begin{align*}
	\rt T\frac{M_{nn}}{2 \gamma_{n}(\tau) v_{n, \tau}}\weakly N(0, 1),
	\quad
	M_{ii} \gtrsim 1,\quad \forall i\ne n,
 \numberthis\label{equation - deteq in strat}
\end{align*}
then (b) establish a bound of sufficient sharpness on the off-diagonals of $M$, and (c) identify the limits in probability of the matrix $\theta^{-1} X_\tau  X_0\tp  X_0   X_\tau\tp$.
It is clear that the matrix $M$ and the resolvents $R$ and $Q^{-1}$ appearing in the definition of $M$ are the central objects of our analysis.
The following proposition provides an asymptotic approximation to $M$ suited to our purpose. 
\begin{myprop}{3}
	Define the matrices
	\begin{align*}
		A:= \frac{1}{\rt{\theta}}X_0 R X_\tau\tp,
		\quad B:= \frac{1}{\theta} X_\tau E_0\tp E_0 R X_\tau\tp,
	\end{align*}
	so that $M = I_K - B - A\tp Q^{-1}A$. For each $i=1,\ldots, K$, define
	\begin{align*}
		\overline M_{ii} :=1- \Em[B_{ii}1_{\cB_0}] - \Em[A_{ii}1_{\cB_0}]^2 \Em[Q_{ii}^{-1}1_{\cB_2}].
	\end{align*}
	Then under Assumption~\ref{assumptions} and either Assumption~\ref{assumptions - tau fixed} or \ref{assumptions - tau div},
	we have
	\begin{align*}
		{M_{ii}  - \overline M_{ii}}
		 & =
		- 2\big(A_{ii} - \underline\Em [A_{ii}]\big)
		{\Em[A_{ii}1_{\cB_0}]}{\Em [Q_{ii}^{-1} 1_{\cB_2}]}
		+ O_{p}\left(\frac{\sigma_i^2 \norm{\sigmab}_{\ell_1}^2}{\theta T}+ \frac{\sigma_i^2}{\theta \rt T} +KT^{-1}\right),
	\end{align*}
	for all $i=1,\ldots, K$,
	where $\underline \Em[\ \cdot \ ]$ is  defined in \eqref{equation - Emcond}.
	Furthermore,
	\begin{align*}
		\max_{i\ne j} |M_{ij}|
		=  O_p\left( \frac{K^4 \sigma_1^4 \sigma_i \sigma_j}{\theta^2 \rt T}\right).
	\end{align*}
\end{myprop}
A few remarks are in order to explain why the approximation in Proposition~\ref{proposition - M} is constructed in a seemingly unusual way. Since $\theta$ diverges, the resolvents $R$ and $Q^{-1}$ defined in (\ref{equation - definition of R}) and (\ref{equation - definition of Q}), respectively, are very close to identity matrices for large $T$, a fact used frequently in our proofs. However one cannot simply replace them with identities to simplify \eqref{equation - definition of M}.
Indeed, it is easy to show that $R - I_{T- \tau} = O_{p, \norm{\cdot}}(\theta^{-1})$, which converges to zero but not fast enough for obtaining a CLT after scaling by $\rt T$. This is a fundamental difficulty under our setting since we allow $\theta$  to diverge at any rate and not as a specified function of $T$. In fact, if we were to impose for instance $\theta \gg \rt T$, our proofs will be greatly simplified.

It can be shown however that this approximation error of order $\theta^{-1}$ appears only in the mean of the asymptotic distribution, see for instance \eqref{equation - Aii - xx only differnt in mean}. That is, we can safely use identity matrices to approximate $R$ and $Q^{-1}$ in Proposition~\ref{proposition - M} as long as we include an appropriate centering term to adjust the expectation of $M$ before multiplying by $\rt T$.
In fact, identifying the correct centering for $M$ results in the equation \eqref{equation - definition of thetak} that determines the asymptotic location $\theta$ of the eigenvalues $\lambda$. 

Essentially, we construct approximations to $R$ and $Q^{-1}$ that are more accurate than the identity, which in our case turn out to be their expectations under certain events of high probability.
To bound the approximation errors, we establish the concentration of $R$ around its expectation in Lemma~\ref{lemma - concentration of R}, the concentration of $Q^{-1}$ around a certain conditional expectation in Lemma~\ref{lemma - concentration of Q inverse}, and estimates on the differences between conditional and unconditional expectations in Lemma~\ref{lemma - conditional expectations}.
After obtaining these technical results, we show in Proposition~\ref{proposition - M} that after centering by a certain conditional expectation (which is later replaced by an unconditional one using Lemma~\ref{lemma - conditional expectations}), the asymptotic distribution of $M$ can be obtained from the asymptotic distribution of the bilinear form $X_0 R X_\tau\tp$, up to adjustments in the expectations.

Therefore it remains to establish the asymptotics of $X_0 R X_\tau\tp$.
Using tools developed in
Lemmas~\ref{lemma - properties of Psi}~to~\ref{lemma - expectation of ABQ}, we study the bilinear form $X_0 R X_\tau\tp$ and establish its concentration around some conditional expectation. Using these results we show in the proof of Proposition~\ref{proposition - CLT of Mii} that the asymptotic normality for $X_0 R X_\tau\tp$ follows from the asymptotic normality of the much simpler auto-covariance matrix $X_0 X_\tau\tp$, again up to adjustments in the expectations. The CLT for this matrix $X_0 X_\tau\tp$ is established in Proposition~\ref{proposition - CLT for fii}. Finally, Proposition~\ref{proposition - CLT of Mii} gives the CLT for diagonals of the matrix $M$.

\begin{myprop}{5}
	Under Assumption~\ref{assumptions} and either Assumption~\ref{assumptions - tau fixed} or \ref{assumptions - tau div}, we have
	\begin{align*}
		\rt T
		\frac{\theta}{2\sigma_i^4 \gamma_i(\tau	) v_{i, \tau	}}
		\big(M_{ii} - \overline M_{ii} \big)
		 = Z_T + O_p\left(\frac{\norm{\sigmab}_{\ell_1}^2}{\sigma_i^2 \gamma_i(\tau)^2 \rt T}+ \frac{1}{ \sigma_i^2 \gamma_i(\tau) } +\frac{\sigma_n^4\gamma_n(\tau)^2 K }{\sigma_i^4\gamma_i(\tau)^2 \rt T}\right)
	\end{align*}
	where $Z_T\weakly N(0,1)$, the centering $\overline M_{ii}$ is as defined in \eqref{equation - decomposition of Mii} and $v_{i, \tau}$ is defined as in \eqref{equation - definition of v i tau}.
\end{myprop}

The proof of Theorem~\ref{theorem - CLT} can then be assembled from the above described ingredients, we present the proof at the end of Appendix \ref{CLT - proofs}.
To summarize, the quantity of interest $\delta$ is first shown to satisfy equation \eqref{equation - detequation 2}. Through a series of approximations, we establish the asymptotic normality of the diagonals of the matrix $M$. The off-diagonals of $M$ are bounded in probability and we establish the limit in probability of the matrix $X_\tau X_0\tp X_0 X_\tau\tp$ appearing. Leibniz's formula is then applied to compute the determinant in Proposition~\ref{proposition - detequation }, and our main result Theorem~\ref{theorem - CLT} follows.

\subsection{Outline of the proof}
We finally include a brief outline of the proofs to help the readers navigate.

Firstly, in Proposition~\ref{proposition - detequation }, the quantity of interest $\delta$ is expressed as the solution to equation \eqref{equation - deteq in strat}. 
As will be shown in the proof of Theorem~\ref{theorem - CLT}, the asymptotic normality of $\delta$ follows from the asymptotic normality of the matrix $M$ appearing in Proposition~\ref{proposition - detequation }. 

The asymptotic normality of $M$ is established in Propositions~\ref{proposition - M}~to~\ref{proposition - CLT of Mii}. 
We first identify an appropriate centering for the matrix $M$ in \eqref{equation - decomposition of Mii}. Using this centering, we show in Proposition~\ref{proposition - M} that the diagonal elements of $M$ can be approximated by certain bilinear forms defined in \eqref{equation - definition of A B}, at the scale of $o(T^{-1/2})$. Proposition~\ref{proposition - CLT of Mii} obtains the asymptotic normality of $M$ by further reducing this approximation into a much simpler bilinear form, the asymptotic distribution of which is established in Proposition~\ref{proposition - CLT for fii}. 

The results described above all rely on a collection of technical lemmas. 
Throughout the proofs, we often encounter bilinear forms involving resolvent matrices $R$ and $Q^{-1}$. We routinely approximate these resolvent matrices and the bilinear forms by certain expectations. In Lemmas~\ref{lemma - concentration of R}~to~\ref{lemma - conditional expectations} we establish the concentration of $R$ and $Q^{-1}$ around certain expectations, and in
Lemmas~\ref{lemma - properties of Psi}~to~\ref{lemma - expectation of ABQ} we establish the concentration of bilinear forms involving $R$ and $Q^{-1}$. 
\section{Statistical application: autocovariance test}\label{section - applications}
In this section, a novel test, called the autocovariance test, is proposed to detect the equivalence of spikiness for two high-dimensional time series.   
As analyzed in \citet{LamYao2012}, the eigenvectors corresponding to the $K$ spiked eigenvalues of $\Sigma_\yy(\tau)\Sigma^{\top}_\yy(\tau)$ span a $K$-dimensional linear subspace, where the projection of the original high-dimensional time series holds all the temporal dependence. For easy reference, we call this subspace as $K$-dimensional temporal subspace and denote it as $\mathcal{M}_K$. 
When two high-dimensional time series share the same $K$-dimensional temporal subspace, the proposed test is equivalent to checking whether the two projected time series have the same autocovariance. 
As a further application of the proposed autocovariance test, new hierarchical clustering analysis is constructed to cluster a large set of high-dimensional time series, where the dissimilarity between two populations is measured via the p-value of the proposed autocovariance test. The major aim of this clustering analysis is to group high-dimensional time series with similarly projected autocovariances.  

To explain the idea of the autocovariance test and its application on the hierarchical clustering in detail, we will simply revisit the factor structures for high-dimensional time series and introduce the proposed test statistic with its asymptotic properties in Section \ref{3:sec:2.1}. Section \ref{3:sec:3} describes how the hypothesis test can be implemented in practice where a flow chart is also provided to clarify the essential idea of the test procedure.
We then use numerical simulations to investigate the empirical sizes and powers of the proposed test 
Finally, the proposed test and the hierarchical clustering method with p-values acting as the measure of dissimilarities are applied to mortality data from multiple countries.




\subsection{Hypotheses and test statistic}\label{3:sec:2.1}

Consider for $\left\{\yy_{t}^{(1)} \in \mathbb{R}^{K_1+p_1},\ t = 1,2,...,T\right\}$ and $\left\{\yy_{t}^{(2)} \in \mathbb{R}^{K_2+p_2},\ t = 1,2,...,T\right\}$, which are two high-dimensional time series following the factor model in canonical form \eqref{canonical form}, that is we have
\begin{align*}
	\yy_t^{(m)} = L^{(m)} \ff_t^{(m)} + \epsilonb_t^{(m)},\quad t=1,\ldots, T,
	\quad m=1,2,
	\numberthis\label{3:e1}
\end{align*}
where $\left\{\ff_t^{(m)} \in \mathbb{R}^{K_m},\ t = 1,2,...,T\right\}$ are stationary factor time series with variances normalized to 1, $K_m \ll p_m$, and $L^{(m)}$ is a $(p_m+K_m) \times K_m$ factor loading matrix which takes the form
\begin{align*}
	L^{(m)} = \begin{pmatrix}
		\mathrm{diag}(\sigma_1^{(m)}, \ldots, \sigma_K^{(m)}) \\ \boldsymbol{0}_{p\times K}^{(m)}
	\end{pmatrix}.
	\numberthis\label{canonical form 4}
\end{align*}
To simplify the notations, we also let $N_m \coloneqq (p_m+K_m)$ be the dimension of $\left\{\yy_t^{(m)}\right\}$. 

For high-dimensional time series $\left\{\yy_t^{(m)}\right\}$ following factor models such as (\ref{3:e1}), $L^{(m)}$ is the time invariant factor loading matrix. As discussed in \citet{LamYaoBathia2011} and \citet{LamYao2012}, when $\left\{\epsilonb_t^{(m)}\right\}$ are i.i.d, 
the temporal dependence of $\left\{\yy_t^{(m)}\right\}$ is fully captured by $\left\{\ff_t^{(m)}\right\}$.
Denote by $\mu_{i,\tau}^{(m)}$ the eigenvalues of $\Sigma_\yy^{(m)}(\tau) \Sigma_\yy^{(m)}(\tau)^\top$. 
We then consider in this section the setting where there are $K$ spiked eigenvalues in $\Sigma_\yy^{(m)}(\tau) \Sigma_\yy^{(m)}(\tau)^\top $ and $\mu_{1,\tau}^{(m)} > \mu_{2,\tau}^{(m)} > ... > \mu_{K_m,\tau}^{(m)}$ tend to infinity with $T,p$, while $\mu_{K_m+1,\tau}^{(m)} = \mu_{K_m+2,\tau}^{(m)} = ... = \mu_{N_m,\tau}^{(m)} = 0$ for any $\tau \ge 1$.
With this definition of spiked eigenvalues, we can show that the columns of $L^{(m)}$ are the eigenvectors of $\Sigma_\yy^{(m)}(\tau) \Sigma_\yy^{(m)}(\tau)^\top$ corresponding to the spiked eigenvalues, as follow.

Write $\pW^{(m)}$ for a $N_m \times p_m$ matrix where $\left(L^{(m)},\pW^{(m)}\right)$ forms a $N_m \times N_m$ orthogonal matrix so that $L^{(m)\top}\pW^{(m)} = \boldsymbol{0}$ and $\pW^{(m)\top}\pW^{(m)} = I_{p_m}$. It then follows 
that $\Sigma_\yy^{(m)}(\tau) \Sigma_\yy^{(m)}(\tau)^\top \pW^{(m)} = \boldsymbol{0}$, which means the columns of $\pW^{(m)}$ are precisely the eigenvectors associated with zero-eigenvalues.
In other words, the columns of $L^{(m)}$ are the $K_m$ eigenvectors of $\Sigma_\yy^{(m)}(\tau) \Sigma_\yy^{(m)}(\tau)^\top$ corresponding to those non-zero eigenvalues, and those non-zero eigenvalues of $\Sigma_\yy^{(m)}(\tau) \Sigma_\yy^{(m)}(\tau)^\top$ are precisely $\sigma_1^{2(m)}, \ldots, \sigma_K^{2(m)}.$ 


Consequently, 
$\mathcal{M}\left(L^{(m)}\right)$ is the temporal subspace spanned by the columns of $L^{(m)}$, which are also the eigenvectors corresponding to the spiked eigenvalues of the symmetrized autocovariance matrix of $\left\{\yy_t^{(m)}\right\}$. 
Therefore, when $\mathcal{M}\left(L^{(1)}\right) = \mathcal{M}\left(L^{(2)}\right)$, we can build a test statistic based on the difference between spiked eigenvalues of the symmetrized lag-$\tau$ sample autocovariance matrices of two high-dimensional time series $\left\{\yy_t^{(1)}\right\}$ and $\left\{\yy_t^{(2)}\right\}$. This test statistic is  to detect the equivalence of autocovariances for two projected time series in the temporal subspace.  In fact, the analysis of $\mathcal{M}\left(L^{(m)}\right)$ does not rely on the canonical form \eqref{canonical form 4}, and is valid for any identification conditions on $L^{(m)}$.



In this section, it is worth noting that we typically focus on testing the equivalence of spiked eigenvalues, but not the eigenspace of autocovariance matrices for two high-dimensional time series $\left\{\yy_t^{(1)}\right\}$ and $\left\{\yy_t^{(2)}\right\}$. Consequently, when $\mathcal{M}\left(L^{(1)}\right) = \mathcal{M}\left(L^{(2)}\right)$ and $K_1 = K_2 =K$, the null and alternative hypotheses of the autocovariance test for $\left\{\yy_t^{(1)}\right\}$ and $\left\{\yy_t^{(2)}\right\}$ with a finite time lag $\tau$ can be summarized as

\begin{test*}\label{3:pet}
	Autocovariance test for two high-dimensional time series $\left\{\yy_t^{(1)}\right\}$ and $\left\{\yy_t^{(2)}\right\}$\\
	\centerline{H$_0$: $\mu_{i,\tau}^{(1)} = \mu_{i,\tau}^{(2)}$ for all $i=1,2,...,K$} \\
	\centerline{H$_1$: $\mu_{i,\tau}^{(1)} \ne \mu_{i,\tau}^{(2)}$ for at least one $i$, $i=1,2,...,K$}
\end{test*}

Recall that for factor models in canonical form (\ref{canonical form 4}), 
we can write $\gamma_{i,\tau}^{(m)} := \mathbb{E}\left(f_{i,1}^{(m)}f_{i,\tau+1}^{(m)}\right)$ and $\left(v_{i,\tau}^{(m)}\right)^{2} := \frac{1}{T-\tau} Var\left( \sum_{t=1}^{T-\tau} f_{i,t}^{(m)} f_{i,t+\tau}^{(m)} \right)$ for a finite time lag $\tau$, $i = 1,2,...,K$ and $m=1,2$. Denote by $\lambda_{i,\tau}^{(m)}$ the $i$-th largest spiked eigenvalue of the symmetrized lag-$\tau$ sample autocovariance matrix $\tGa_\yy^{(m)}(\tau) \tGa_\yy^{(m)}(\tau)^\top$, where $\tGa_\yy^{(m)}(\tau) = \frac{1}{T-\tau-1} \sum_{t=1}^{T-\tau} (\yy_t^{(m)} - \overline{\yy}_T^{(m)}) (\yy_{t+\tau}^{(m)} - \overline{\yy}_T^{(m)})^\top$, for $m=1,2$. Then, for $i=1,2,...,K$ and some finite $\tau$, the test statistic is given by
\begin{equation}\label{3:stat0}
	Z_{i,\tau} = \sqrt{T} \frac{\gamma_{i,\tau}}{2\sqrt{2}v_{i,\tau}} \frac{\lambda_{i,\tau}^{(1)}-\lambda_{i,\tau}^{(2)}}{\theta_{i,\tau}}, \ \ \ i=1, 2, \ldots, K,
\end{equation}
where
\begin{align}\label{3:para}
	{\theta}_{i,\tau} = \frac{ {\theta}_{i,\tau}^{(1)} +  {\theta}_{i,\tau}^{(2)}}{2},\
	{v}_{i,\tau} = \frac{ {v}_{i,\tau}^{(1)} +  {v}_{i,\tau}^{(2)}}{2},\ \text{and}\
	{\gamma}_{i,\tau} = \frac{ {\gamma}_{i,\tau}^{(1)} +  {\gamma}_{i,\tau}^{(2)}}{2},
\end{align}
and ${\theta}_{i,\tau}^{(m)}$ is the asymptotic centering of ${\lambda}_{i,\tau}^{(m)}$ defined in Proposition \ref{proposition - solution for theta}. 
It is then clearly that $\left|Z_{i,\tau}\right|$ is generally large if $\left\{\yy_t^{(1)}\right\}$ and $\left\{\yy_t^{(2)}\right\}$ follow different factor models where the $i$-th largest eigenvalues of the symmetrized lag-$\tau$ sample autocovariance matrix for two factor models are different. We name this test by autocovariance test since the idea behind is testing whether two independent high-dimensional time series observations share the same spiked eigenvalues of the autocovariance matrices. 


For simplicity, we assume the idiosyncratic components $\left\{\epsilonb_t^{(m)} \in \mathbb{R}^{N_m}, t=1,2,...,T\right\}$ are independent of the factors $\left\{\ff_t^{(m)}\right\}$, with $\mathbb{E}\left(\epsilon_{j,t}^{(m)}\right) = 0$ for all $j=1,2,...,N_m,$ and $\mathbb{E}\left(\epsilonb_{t}^{(m)}\right)^{2} \eqqcolon \Sigma_\epsilon^{(m)} = diag \left( \left(\sigma_{\epsilon,1}^{(m)}\right)^2,\left(\sigma_{\epsilon,2}^{(m)}\right)^2,...,\left(\sigma_{\epsilon,N_m}^{(m)}\right)^2 \right)$. Without loss of generality, we can again work on standardized factor models in canonical form, where the variance of $\epsilon_{j,t}^{(m)}$ is normalized to one, i.e. $\left(\sigma_{\epsilon,j}^{(m)}\right)^{2} = 1$. This standardization is just a transformation on $\left\{\yy_t^{(1)}\right\}$ and $\left\{\yy_t^{(2)}\right\}$ so that they can be transformed to the same canonical form if they share the same number of factors. 
For factor models with $\left(\sigma_{\epsilon,j}^{(m)}\right)^{2} \ne 1$, we can simply standardize them by dividing $\sigma_{\epsilon,j}^{(m)}$. 
In this section, we only consider the case for a fixed time lag $\tau$ and follow Assumptions \ref{assumptions - tau fixed} to simplify the factor models into canonical form \eqref{canonical form}. 

In summary, we consider factor models \eqref{3:e1} in canonical form with the loading matrix $L^{(m)}$ defined by \eqref{canonical form 4} and the variances of $\left\{f_{i,t}^{(m)}\right\}$ and $\left\{\epsilon_{j,t}^{(m)}\right\}$ normalized to one. In addition, we assume the data $\left\{\yy_t^{(m)}\right\}$ comes from strong factor models where $\sigma_i^{(m)}$ is divergent as $N \to \infty$ for $i=1,2,...,K$ and $m=1,2$. Besides, for a general strong factor model that is not in the canonical form (\ref{canonical form 4}), it can be normalized by standardizing the variance of $\left\{\epsilon_{j,t}^{(m)}\right\}$ to one first and then rotating the original data such that the loading matrix $L^{(m)}$ is in the canonical form (\ref{canonical form 4}).
Moreover, recall that for a finite time lag $\tau$, $\gamma_{i,\tau}^{(m)} := \mathbb{E}\left(f_{i,1}^{(m)}f_{i,\tau+1}^{(m)}\right)$ is the population lag-$\tau$ autocovariance of the $i$-th factor time series $\left\{f_{i,t}^{(m)}\right\}$. 
Following (\ref{equation - definition of gammitau}), (\ref{equation - definition of mu n tau}) and (\ref{equation - definition of v i tau}), $\gamma_{i,\tau}^{(m)}$ can be written as
\begin{align*}
	\gamma_{i,\tau}^{(m)} = \mathbb{E}\left(f_{i,1}^{(m)}f_{i,\tau+1}^{(m)}\right) = \sum_{l=0}^{\infty} \phi_{i,l}^{(m)} \phi_{i,l+\tau}^{(m)},
\end{align*}
with the constraint $\norm{\phib_i}_{\ell_2}=1$, the population lag-$\tau$ autocovariance can be defined as
\begin{align*}
	\mu_{i,\tau}^{(m)} \coloneqq \mathbb{E} \left(y_{i,t}^{(m)} y_{i,t+\tau}^{(m)}\right) = \left(\sigma_i^{(m)}\right)^2 \gamma_{i,\tau}^{(m)},
\end{align*}
and $\left(v_{i,\tau}^{(m)}\right)^2 = \frac{1}{T-\tau} Var\left( \sum_{t=1}^{T-\tau} f_{i,t}^{(m)} f_{i,t+\tau}^{(m)} \right)$ 
for a finite positive time lag $\tau$, $i = 1,2,...,K,$ and $m=1,2$. 
If $\left\{\yy_t^{(1)}\right\}$ and $\left\{\yy_t^{(2)}\right\}$ are assumed following the same canonical factor model under Assumptions \ref{assumptions} and \ref{assumptions - tau fixed}, independently, it is clearly that $\lambda_{i,\tau}^{(1)}$ and $\lambda_{i,\tau}^{(2)}$ share the same asymptotic distribution as shown in Theorem \ref{theorem - CLT}, independently. Therefore, to test whether $\left\{\yy_t^{(1)}\right\}$ and $\left\{\yy_t^{(2)}\right\}$ share the same spiked eigenvalues of the autocovariance matrices, it is natural to create the test statistic (\ref{3:stat0}) based on the difference between $\lambda_{i,\tau}^{(1)}$ and $\lambda_{i,\tau}^{(2)}$. When $\left\{\yy_t^{(1)}\right\}$ and $\left\{\yy_t^{(2)}\right\}$ follow the same factor model in the canonical form \eqref{canonical form 4}, we have the following CLT on the difference between $\lambda_{i,\tau}^{(1)}$ and $\lambda_{i,\tau}^{(2)}$.



\begin{theorem}\label{3:th:1}
	Under Assumptions \ref{assumptions} and \ref{assumptions - tau fixed}, for two independent high-dimensional time series $\left\{\yy_t^{(1)}\right\}$ and $\left\{\yy_t^{(2)}\right\}$ following the same factors in canonical form (\ref{canonical form 4}), we have
	\begin{equation}
		Z_{i,\tau} = \sqrt{T} \frac{\gamma_{i,\tau}}{2\sqrt{2}v_{i,\tau}} \frac{\lambda_{i,\tau}^{(1)}-\lambda_{i,\tau}^{(2)}}{\theta_{i,\tau}} \Rightarrow \mathcal{N}(0,1),
	\end{equation}
	as $T,p \to \infty$, where ${\theta}_{i,\tau}$, ${v}_{i,\tau}$ and ${\gamma}_{i,\tau}$ are defined in (\ref{3:para}).
\end{theorem}
Theorem \ref{3:th:1} is a direct result of Theorem \ref{theorem - CLT}, since an asymptotic distribution of $\frac{\lambda_{i,\tau}^{(1)}-\lambda_{i,\tau}^{(2)}}{\theta_{i,\tau}}$ can be derived using the independence between $\lambda_{i,\tau}^{(1)}$ and $\lambda_{i,\tau}^{(2)}$. Consequently, under the null hypothesis of the autocovariance test, the test statistic $Z_{i,\tau}$ converges weakly to a standard normal random variable when $T,p \to \infty$. Nonetheless, under certain alternative hypotheses such as $K_1 = K_2=K$,\ 
but 
$\mu_{i,\tau}^{(1)} \ne \mu_{i,\tau}^{(2)}$ and $\theta_{i,\tau}^{(1)} \ne \theta_{i,\tau}^{(2)}$, it can be shown in the next theorem that, under a local alternative hypothesis, the power of the autocovariance test converges to 1 as $T,p \to \infty$.


\begin{theorem}\label{3:th:2}
	Under Assumptions \ref{assumptions} and \ref{assumptions - tau fixed}, if we additionally assume two independent high-dimensional time series $\left\{\yy_t^{(1)}\right\}$ and $\left\{\yy_t^{(2)}\right\}$ follow two different canonical factor models (\ref{canonical form 4}) such that $K_1 = K_2 = K$ and $\theta_{i,\tau}^{(1)} = (1+c) \theta_{i,\tau}^{(2)}$.
	Then, for any $c$ such that $\sqrt{T}\frac{2c}{2+c} \to \infty$ as $T,p \to \infty$ and $\lambda_{i,\tau}^{(1)} \ne \lambda_{i,\tau}^{(2)}$, it holds that
	\begin{equation}
		Pr\left(\left|Z_{i,\tau}\right| > z_{\alpha} | H_1\right) \rightarrow 1,
	\end{equation}
	for $T,p \to \infty$, where $z_{\alpha}$ is the $\alpha$-th quantile of the standard normal distribution.
\end{theorem}

\begin{remark}
	The condition $\sqrt{T}\frac{2c}{2+c} \to \infty$ as $T,p \to \infty$ in Theorem \ref{3:th:2} indicates a local alternative hypothesis. It implies that for $T,p \to \infty$, the power of the test converges to $1$ not only for a constant $c$, but also for some $c \to 0$ as long as $\sqrt{T}{c} \to \infty$. In other words, this test even works asymptotically for a local alternative hypothesis where the difference between $\theta_{i,\tau}^{(1)}$ and $\theta_{i,\tau}^{(2)}$ tends to $0$, but slower than $1/\sqrt{T}$.

\end{remark}

\subsection{Implementation of testing procedure} \label{3:sec:3}
The test statistic $Z_{i, \tau}$ is an infeasible statistic in practice as it involves some unknown parameters $\gamma_{i, \tau}$, $v_{i, \tau}$ and $\theta_{i, \tau}$. In this part, we will propose 
a practical procedure for the autocovariance test. 

For two high-dimensional time series, the test procedure can be summarized into four steps. Firstly, estimates of the factor models for both populations should be conducted, where the number of factors needs to be determined. Secondly, the original high-dimensional observations and the factor models' estimates need to be standardized to fulfill the canonical factor model (\ref{canonical form 4}). Thirdly, the quantities required to compute the feasible test statistic $\widetilde{Z}_{i,\tau}$ should be estimated from both populations. Furthermore, we can compute the feasible test statistic $\widetilde{Z}_{i,\tau}$ and its corresponding $p$-values for testing the equivalence of eigenvalues. The details of the estimation and testing procedures are illustrated and discussed as follows.
\begin{enumerate}
	\item[]\textbf{Step 1}: Estimation of factor models.
	
	For de-meaned high-dimensional time series observations $\left\{\yy_t^{(m)}\right\}$ with $m=1, 2$, we first compute the symmetrized lag-$\tau$ sample autocovariance matrix $ \tGa_\yy^{(m)}(\tau)\tGa_\yy^{(m)}(\tau)^\top$, where $\tGa_\yy^{(m)}(\tau) = \frac{1}{T-\tau-1} \sum_{t=1}^{T-\tau} \yy_t^{(m)} \yy_{t+\tau}^{(m)\top}$ is the lag-$\tau$ sample autocovariance matrix of $\left\{\yy_t^{(m)}\right\}$. By applying spectral (eigenvalue) decomposition on $\tGa_\yy^{(m)}(\tau) \tGa_\yy^{(m)}(\tau)^\top$, we can obtain an estimate of the factor loading matrix as $\widehat{L}_{\tau}^{(m)}=\left(\widehat{L}_{1,\tau}^{(m)},\widehat{L}_{2,\tau}^{(m)},...,\widehat{L}_{p,\tau}^{(m)}\right)$ with $\widehat{L}_{i,\tau}^{(m)}$ the eigenvector of $\tGa_\yy^{(m)}(\tau) \tGa_\yy^{(m)}(\tau)^\top$ corresponding to the $i$-th largest eigenvalue $\widehat{\lambda}_{i,\tau}^{(m)}$. We then use a ratio-based estimator, which has been considered by \citet{LamYaoBathia2011}, to determine the number of factors. The number of factors is determined as $\widehat{K}_{m} = \argmin_{1 \le j \le R} \widehat{\lambda}_{j+1,\tau}^{(m)}/\widehat{\lambda}_{j,\tau}^{(m)}$ where $\widehat{\lambda}_{1,\tau}^{(m)} \ge \widehat{\lambda}_{2,\tau}^{(m)} \ge \cdots \ge \widehat{\lambda}_{N_m,\tau}^{(m)}$ and $R$ is an integer satisfying $K_{m} \le R < N_m$.
	
	With $\widehat{L}_{\tau}^{(m)}$, the factors can then be estimated by $\widehat{\ff}_{t}^{(m)} = \widehat{L}_{\tau}^{(m)\top} \yy_t^{(m)}$ and the high-dimensional time series can be recovered by $\widehat{\yy}_t^{(m)} = \widehat{L}_{\tau}^{(m)} \widehat{\ff}_{t}^{(m)}$. Hence we have estimates of the factor model that is not in the canonical form (\ref{canonical form 4}) and the residuals can be estimated by
	\begin{align}\label{3:estf}
		\hu_t^{(m)} = \yy_t^{(m)} - \widehat{L}_{\tau}^{(m)} \widehat{\ff}_{t}^{(m)}.
	\end{align}
	Moreover, to standardize the estimated factor model into canonical form (\ref{canonical form 4}), we need to find an estimate of $\Sigma_\epsilon^{(m)}$, the covariance of $\epsilonb_{t}^{(m)}$. To achieve that, we can obtain an estimate of the variance of $\epsilon_{j,t}^{(m)}$ as
	$$\left(\widehat{\sigma}_{\epsilon,j}^{(m)}\right)^2 = \frac{1}{T-1} \sum_{t=1}^{T} \left(\widehat{\epsilon}_{j,t}^{(m)} - \overline{\widehat{\epsilon}}_{j,t}^{(m)}\right)^2.$$
	And $\Sigma_\epsilon^{(m)}$ can then be estimated by $$\widehat{\Sigma}_\epsilon^{(m)} = diag \left( \left(\widehat{\sigma}_{\epsilon,1}^{(m)}\right)^2,\left(\widehat{\sigma}_{\epsilon,2}^{(m)}\right)^2,...,\left(\widehat{\sigma}_{\epsilon,N_m}^{(m)}\right)^2 \right).$$
	\begin{remark}\label{3:re:1}
		It is clear that for two high-dimensional time series where the estimated numbers of factors are different, i.e., $\widehat{K}_1 \ne \widehat{K}_2$, one can conclude that the two high-dimensional data follow different factor models where $\mathcal{M}\left(L^{(1)}\right) \ne \mathcal{M}\left(L^{(2)}\right)$ and the numbers of spiked eigenvalues for their autocovariance matrices are different. However, if we are interested in testing the equivalence for the particular spiked eigenvalue of the autocovariance matrices for two high-dimensional data, it is still possible to perform the autocovariance test even if $\widehat{K}_1 \ne \widehat{K}_2$. 
        For example, in analyzing mortality data, the first factor represents human characteristics that lead the trend of mortality improvement across ages \citep[see][for example]{lee_modeling_1992, LiLee2005}. Testing the equivalence of the first eigenvalue of autocovariance matrices across countries or regions may tell whether human characteristics are of the same importance in affecting mortality rates across countries. 
        In meteorology, \citet{ZHANG2022151} study the joint prediction of temperature from multi-stations, where the autocovariance test can be applied since the spiked eigenvalues measure the extent of co-movements of temperature from multi-stations. When predicting human trajectory in crowded spaces \citep[see,][for example]{Alexandre2016}, the spiked eigenvalues can be considered for testing the equivalence of the importance of common sense rules and social conventions in different scenarios. Therefore, The autocovariance test performed based on the first several factors is in its own interest, even if $\widehat{K}_1 \ne \widehat{K}_2$.
	\end{remark}

	\item[]\textbf{Step 2}: Standardizing factor models to the canonical form.
	
	With $\widehat{L}_{\tau}^{(m)}$ and $\widehat{\Sigma}_\epsilon^{(m)}$, we can now standardize the estimated factor models (\ref{3:estf}) to fulfill the canonical form. Firstly, we define a $N_m \times N_m$ matrix $\pM_\tau^{(m)} = \left(\widehat{L}_{\tau}^{(m)},\zero_{p_m+K_m-\widehat{K}_m} \right)$. Then we can define $\widetilde{\yy}_t^{(m)} \coloneqq \left(\widehat{\Sigma}_\epsilon^{(m)}\right)^{-1/2} \pM_\tau^{(m)\top}  \yy_t^{(m)}$ for the normalized data and $\tu_t^{(m)} \coloneqq \left(\widehat{\Sigma}_\epsilon^{(m)}\right)^{-1/2} \hu_t^{(m)}$ for the normalized residuals. By left multiplying $\left(\widehat{\Sigma}_\epsilon^{(m)}\right)^{-1/2} \pM_\tau^{(m)\top}$, the estimated factor model is reduced to
	\begin{align*}
		\widetilde{\yy}_t^{(m)} = \left(\widehat{\Sigma}_\epsilon^{(m)}\right)^{-1/2} \pM_\tau^{(m)\top} \widehat{L}_{\tau}^{(m)} \widehat{\ff}_{t}^{(m)} + \tu_t^{(m)},
	\end{align*}
	where note that
	\begin{align*}
		\pM_\tau^{(m)\top} \widehat{L}_{\tau}^{(m)} =
		\begin{pmatrix}
			I_{\widehat{K}_m}     \\
			\zero_{(p_m+K_m - \widehat{K}_m) \times \widehat{K}_m} \\
		\end{pmatrix}.
	\end{align*}
	
	To normalize $\widehat{\ff}_t^{(m)}$, we estimate the variance of $\widehat{f}_{i,t}^{(m)}$ by $\left(\widehat{\sigma}_i^{(m)}\right)^2 = \frac{1}{T-1}\sum_{t=1}^{T} \left(\widehat{f}_{i,t}^{(m)} - \overline{\widehat{f}}_{i,t}^{(m)}\right)^2$, for $i=1,2,...,\widehat{K}_m$. Hence the covariance of $\widehat{\ff}_t^{(m)}$ can be obtained as $$\widehat{\Sigma}_\ff^{(m)} = diag \left( \left(\widehat{\sigma}_{1}^{(m)}\right)^2,\left(\widehat{\sigma}_{2}^{(m)}\right)^2,...,\left(\widehat{\sigma}_{\widehat{K}_m}^{(m)}\right)^2 \right).$$
	Write $\widetilde{\ff}_t^{(m)} = \left(\widehat{\Sigma}_\epsilon^{(m)}\right)^{-1/2} \widehat{\ff}_t^{(m)} \left(\widehat{\Sigma}_\ff^{(m)}\right)^{-1/2}$ for the normalized estimates of factors, and $\widetilde{L}_\tau^{(m)} = \begin{pmatrix}
		\mathrm{diag}(\widehat{\sigma}_1^{(m)}, \ldots, \widehat{\sigma}_{\widehat{K}_m}^{(m)}) \\ \boldsymbol{0}_{(p_m+K_m - \widehat{K}_m) \times \widehat{K}_m}
	\end{pmatrix}$ for the estimates of loading matrices. Then we have standardized the estimated factor models to
	\begin{align}\label{3:econf}
		\widetilde{\yy}_t^{(m)} = \widetilde{L}_{\tau}^{(m)} \widetilde{\ff}_{t}^{(m)} + \tu_t^{(m)},
	\end{align}
	which follows the canonical form defined by (\ref{canonical form 4}).
	
	\item[]\textbf{Step 3}: \label{3:s3} Estimation of unknown parameters in the test statistic.
	
	For standardized data $\left\{\widetilde{\yy}_t^{(m)}\right\}$ following the estimated factor model (\ref{3:econf}), ${\lambda}_{i,\tau}^{(m)}$ can be computed as the $i$-th largest eigenvalue of the symmetrized lag-$\tau$ sample autocovariance matrix $\tGa_{\widetilde{\yy}}^{(m)}(\tau) \tGa_{\widetilde{\yy}}^{(m)}(\tau)^\top$, where $\tGa_{\widetilde{\yy}}^{(m)}(\tau) = \frac{1}{T-\tau-1} \sum_{t=1}^{T-\tau} \widetilde{\yy}_t^{(m)}\widetilde{\yy}_{t+\tau}^{(m)\top}$ and ${\gamma}_{i,\tau}^{(m)}$ can be estimated from the sample lag-$\tau$ autocovariance of the $i$-th estimated factor $\left\{\widetilde{f}_{i,t}^{(m)}\right\}$. Besides, we also need to estimate the quantities ${v}_{i,\tau}^{(m)}$ and ${\theta}_{i,\tau}^{(m)}$, as defined in the autocovariance test, for each sample to compute the test statistic. However, since $\left({v}_{i,\tau}^{(m)}\right)^2 = \frac{1}{T-\tau} Var\left( \sum_{t=1}^{T-\tau} f_{i,t}^{(m)} f_{i,t+\tau}^{(m)} \right)$ depends on the variance of $\sum_{t=1}^{T-\tau} f_{i,t}^{(m)} f_{i,t+\tau}^{(m)}$ and ${\theta}_{i,\tau}^{(m)}$ is the asymptotic centering of ${\lambda}_{i,\tau}^{(m)}$, they cannot be directly estimated from original sample observations.
	Instead, we can use bootstrap to estimate both quantities. It is worth noting that since the bootstrap is conducted on the estimated low-dimensional factor time series $\left\{\widetilde{\ff}_t^{(m)}\right\}$, the bootstrap estimators are not affected by the increasing dimensions.
	
	
	Therefore, bootstrap methods for time series such as the sieve bootstrap can be conducted on the estimated factors for estimating ${v}_{i,\tau}^{(m)}$ and ${\theta}_{i,\tau}^{(m)}$. Next, we will apply the AR-sieve bootstrap method in \cite{bi2021arsieve} to get a bootstrap estimation for the unknown parameters. 
	In specific, an AR($p$) model can be fitted for each estimated factor $\widetilde{\ff}_i^{(m)}$ and the residuals can be taken as
	\begin{align*}
		\widetilde{u}_{i,t}^{(m)} = \widetilde{f}_{i,t}^{(m)} - \sum_{l=1}^{p} \widetilde{\phi}_{i,l}^{(m)} \widetilde{f}_{i,t-l}^{(m)},
	\end{align*}
	where $\left\{ \widetilde{\phi}_{i,l}^{(m)},\ l=1,2,...,\widehat{K}_m \right\}$ are the AR coefficients. Then by resampling from the empirical distribution of the centralized residual $\left( \widetilde{u}_{i,t}^{(m)} - \overline{\widetilde{u}}_{i}^{(m)} \right)$, the bootstrap factors can be generated as
	\begin{align*}
		{f}_{i,t}^{(m)b} = \sum_{l=1}^{p} \widetilde{\phi}_{i,l}^{(m)} {f}_{i,t-l}^{(m)b} + {u}_{i,t}^{(m)b}, 
	\end{align*}
	where $b = 1,2,...,B$ for $B$ bootstrap samples of $\left\{{f}_{i,t}^{(m)b}\right\}$ and ${u}_{i,t}^{(m)b}$ is the bootstrap residual. Hence, we can estimate ${v}_{i,\tau}^{(m)}$ by
	\begin{align*}
		\widetilde{v}_{i,\tau}^{(m)\ast} = \sqrt{\frac{1}{T-\tau} \left(  \frac{1}{B-1} \sum_{b=1}^{B} \left( \sum_{t=1}^{T-\tau} {f}_{i,t}^{(m)b} {f}_{i,t+\tau}^{(m)b} -     \frac{1}{B} \sum_{b=1}^{B} \left( \sum_{t=1}^{T-\tau} {f}_{i,t}^{(m)b} {f}_{i,t+\tau}^{(m)b}\right) \right)^2 \right) }.
	\end{align*}
	
	In addition, since $\widetilde{\theta}_{i,\tau}^{(m)}$ is an estimate of the asymptotic centering of ${\lambda}_{i,\tau}^{(m)}$, we can also bootstrap $\left\{\widetilde{\yy}_{t}^{(m)}\right\}$ by ${\yy}_{t}^{(m)b} = \widetilde{L}_\tau^{(m)} {\ff}_{t}^{(m)b}$
	for $B$ times and estimate ${\theta}_{i,\tau}^{(m)}$ by $\widetilde{\theta}_{i,\tau}^{(m)\ast} = \frac{1}{B} \sum_{b=1}^{B} {\lambda}_{i,\tau}^{(m)b}$,
	where ${\lambda}_{i,\tau}^{(m)b}$ is the $i$-th largest eigenvalue of the symmetrized lag-$\tau$ sample autocovariance matrices of $\left\{{\yy}_{t}^{(m)b}\right\}$. Meanwhile, since the sieve bootstrap is conducted to estimate ${v}_{i,\tau}^{(m)}$ and ${\theta}_{i,\tau}^{(m)}$, an alternative estimate of ${\gamma}_{i,\tau}^{(m)}$ can also be computed based on $B$ bootstrap samples, as
	\begin{align*}
		\widetilde{\gamma}_{i,\tau}^{(m)\ast} =
		\frac{1}{B} \sum_{b=1}^{B}       \left(  \frac{1}{T-\tau-1} \sum_{t=1}^{T-\tau} \left( f_{i,1}^{(m)b} - \frac{1}{T} \sum_{t=1}^{T} f_{i,t}^{(m)b} \right)  \left( f_{i,\tau+1}^{(m)b} - \frac{1}{T} \sum_{t=1}^{T} f_{i,t}^{(m)b} \right)   \right).
	\end{align*}
	
	\item[]\textbf{Step 4}: Computing the test statistic and $p$-value.
	
	When the first three steps have been conducted on both high-dimensional times series $\left\{\yy_t^{(1)}\right\}$ and $\left\{\yy_t^{(2)}\right\}$, we can estimate the unknown parameters in (\ref{3:stat0}) by
	\begin{align*}
		\widetilde{\theta}_{i,\tau}^\ast \coloneqq \frac{T_1 \widetilde{\theta}_{i,\tau}^{(1)\ast} + T_2 \widetilde{\theta}_{i,\tau}^{(2)\ast}}{T_1+T_2},\
		\widetilde{v}_{i,\tau}^\ast \coloneqq \frac{T_1 \widetilde{v}_{i,\tau}^{(1)\ast} + T_2 \widetilde{v}_{i,\tau}^{(2)\ast}}{T_1+T_2},\
		\widetilde{\gamma}_{i,\tau}^\ast \coloneqq \frac{T_1 \widetilde{\gamma}_{i,\tau}^{(1)\ast} + T_2 \widetilde{\gamma}_{i,\tau}^{(2)\ast}}{T_1+T_2},
	\end{align*}
	where $\widetilde{\theta}_{i,\tau}^{(m)\ast}$, $\widetilde{v}_{i,\tau}^{(m)\ast}$ and $\widetilde{\gamma}_{i,\tau}^{(m)\ast}$ are computed from two high-dimensional time series following the procedure in Step 3. Then, the test statistic can be computed as
	\begin{align} \label{3:teststat}
		\widetilde{Z}_{i,\tau} \coloneqq \left({\lambda}_{i,\tau}^{(1)} - {\lambda}_{i,\tau}^{(2)}\right) \sqrt{\frac{T_1 T_2 }{T_1+T_2}} \frac{\widetilde{\gamma}_{i,\tau}^\ast} {2 \widetilde{v}_{i,\tau}^\ast \widetilde{\theta}_{i,\tau}^\ast},
	\end{align}
	where ${\lambda}_{i,\tau}^{(1)}$ and ${\lambda}_{i,\tau}^{(2)}$ are the $i$-th ($1\le i \le \widehat{K}$) largest eigenvalues of the symmetrized lag-$\tau$ sample autocovariance matrices for the standardized data $\left\{\widetilde{\yy}_{t}^{(1)}\right\}$ and $\left\{\widetilde{\yy}_{t}^{(2)}\right\}$, respectively.
	At last, the $p$-values of the test statistic $\widetilde{Z}_{i,\tau}$ are computed as $\Pr \left(z > \left|\widetilde{Z}_{i,\tau}\right| \right) = 2\left( 1-\Phi\left(|\widetilde{Z}_{i,\tau}|\right) \right)$ for a two-sided test, and $\Pr \left(z > \widetilde{Z}_{i,\tau}\right) = 1-\Phi\left(\widetilde{Z}_{i,\tau}\right)$ or $\Pr \left(z < \widetilde{Z}_{i,\tau}\right) = \Phi\left(\widetilde{Z}_{i,\tau}\right)$ for one-sided tests, where $\Phi\left(\cdot\right)$ denotes the cumulative distribution function (CDF) of a standard normal random variable.
\end{enumerate}

To clarify the four steps mentioned above, a flow chart is provided in Figure~\ref{new1}, which summarizes the basic logic and procedure for the autocovariance test.
\begin{figure}[!htbp]
	\centering
	\includegraphics[scale=0.8]{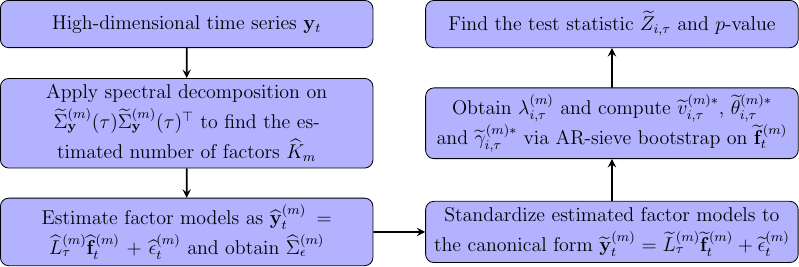}
	\caption{Flow chart for the autocovariance test}\label{new1}
\end{figure} 

	\section{Simulation studies}\label{3:sec:4}

This section uses numerical simulations to investigate the empirical sizes and powers of the proposed autocovariance test.

To start, we first of all explore the empirical sizes of the autocovariance test under various scenarios, including various orders of factor strength and ratios between the sample size and the data dimension. In this section of simulation studies, we again write $N_m \coloneqq K_m+p_m$ as the data dimension of $\left\{\yy_t^{(m)}\right\}$ and consider only the case that $N_1 = N_2 \eqqcolon N$ for simplicity. We assume the high-dimensional observations $\left\{\yy_t^{(1)}\right\}$ and $\left\{\yy_t^{(2)}\right\}$ are generated from the one-factor model $\yy_t^{(m)} = L^{(m)} \ff_t^{(m)} + \epsilonb_t^{(m)}$ in the canonical form (\ref{canonical form 4}). Moreover, we assume that the factors $\left\{f_{1,t}^{(m)}\right\}$ for both time series follow AR($1$) models with zero means, AR coefficients $\phi_1^{(m)} = 0.5$ and variances equal to 1. In other words, the factors for both time series are generated by
\begin{align}\label{3:dgpf}
	f_{1,t}^{(m)} = \phi_1^{(m)} f_{1,t-1}^{(m)} + z_{1,t}^{(m)},\ m=1,2,
\end{align}
where $\phi_1^{(m)} = 0.5$ and $\left\{z_{1,t}^{(m)}\right\}$ are i.i.d \ $\mathcal{N}\left(0,\left(\sigma_z^{(m)}\right)^2\right)$ with $\left(\sigma_z^{(m)}\right)^2 = 1/{\left(1-\left(\phi_1^{(m)}\right)^2\right)} = 3/4$, so that $Var\left(f_{1,t}^{(m)}\right) = 1$. As discussed for the canonical form of factor models, the variance $\left(\sigma_i^{(m)}\right)^2$ of unnormalized factors are contained in the loading matrix $L^{(m)}$. To study the empirical sizes of the autocovariance test under various factor strengths, we consider the case $\left(\sigma_1^{(m)}\right)^2 \asymp N^{1-\delta}$ for $\delta \in [0,1)$ utilized in \citet{LamYaoBathia2011}. Using this definition, $\delta = 0$ refers to the strongest factors with the pervasiveness, and factor strengths drop when $\delta$ increases from $0$ to $1$. In this section, we consider four different cases for factor strengths, where $\delta = 0, 0.1, 0.3,$ and $0.5$. Specifically, $\left(\sigma_1^{(m)}\right)^2$ in the loading matrix $L^{(m)}$ that follows canonical form (\ref{canonical form 4}) is assumed to be $N, N^{0.9}, N^{0.7},$ and $N^{0.5}$, respectively, and $\left\{\epsilon_{j,t}^{(m)}\right\}$ are assumed to be i.i.d.\ $\mathcal{N}(0,1)$. 

In summary, both $N$-dimensional time series observations are generated by 
\begin{align}\label{3:dgpy}
	\yy_{t}^{(m)} =  \begin{pmatrix}
		\sigma_{1}^{(m)} \\
		\zero_{N-1}      \\
	\end{pmatrix}
	f_{1,t}^{(m)} + \epsilonb_t^{(m)},\ m=1,2,
\end{align}
where $\sigma_1^{(m)} = N^{1-\delta}$, $\left\{\epsilon_{j,t}\right\}$ are i.i.d.\ $\mathcal{N}(0,1)$, and $\left\{f_{1,t}^{(m)}\right\}$ are generated by (\ref{3:dgpf}).

To explore the impact of ratios between sample size $T$ and data dimension $N$, we generate data with $T=400, 800$ and $N = 100, 200, 400, 800, 1600$. 
To compute the empirical sizes, for each combination of $T,N$ and $\delta$, the observations of two high-dimensional time series are first of all generated. 
Then, by utilizing the estimation and testing procedures in Section~\ref{3:sec:3}, the test statistic $\widetilde{Z}_{i,\tau}$ can be computed by (\ref{3:teststat}), where $B=500$ bootstrap samples are generated to compute $\widetilde{\theta}_{i,\tau}^{(m)\ast}$, $\widetilde{v}_{i,\tau}^{(m)\ast}$ and $\widetilde{\gamma}_{i,\tau}^{(m)\ast}$, and the numbers of factors are assumed to be known (i.e., $\widetilde{K}_m=1$) for both samples. The empirical sizes of a one-sided autocovariance test for $i=1$, $\tau=1,$ and significant level $\alpha= 0.1$ 
are computed as the empirical probabilities that $\widetilde{Z}_{1,1}$ is less than $z_{\alpha}$ or greater than $z_{1-\alpha}$, i.e.,
\begin{align*}
	\frac{1}{M} \sum_{m=1}^M \mathbf{1}_{\left\{\widetilde{Z}_{1,1}(m) < z_{\alpha}\right\}},\ \text{or} \ \frac{1}{M} \sum_{m=1}^M \mathbf{1}_{\left\{\widetilde{Z}_{1,1}(m) > z_{1-\alpha}\right\}} ,
\end{align*}
for $M = 500$ Monte Carlo simulations, where $\widetilde{Z}_{1,1}(m)$ is the test statistic computed from the $m$-th simulation.

\begin{figure}[!htbp]
	\centering
	\includegraphics[scale=0.45]{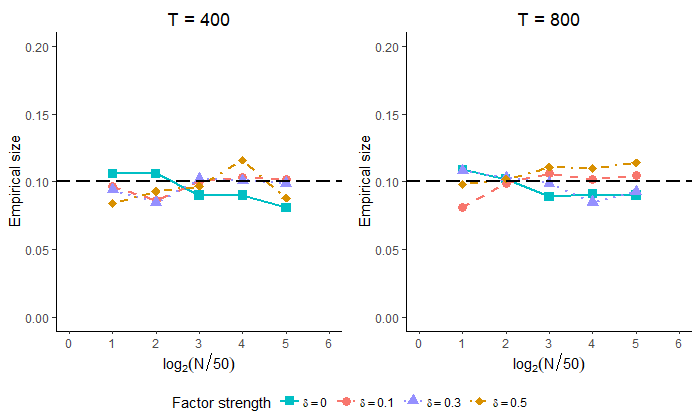}
	\caption{Empirical sizes of the autocovariance test in the first scenario with $T=400, 800,$ $N=100,200,400,800,1600,$ and $\delta = 0, 0.1, 0.3, 0.5.$}\label{3:f1}
\end{figure}

As presented in Figure~\ref{3:f1}, despite some minor fluctuations, the empirical sizes of the autocovariance test are close to the nominal significant level $\alpha = 0.1$ for all choices of $N, T$ and $\delta$. That is, when the numbers of factors are known or can be correctly estimated, the nominal type-I errors of the autocovariance test can be verified via empirical simulation studies for $\delta = 0, 0.1, 0.3, 0.5$, $T=400, 800,$ and $N = 100, 200, 400, 800, 1600$. The choice of $\tau=1$ for the autocovariance test is to acquire the most information on temporal dependence of the observations 
and to achieve the best accuracy on corresponding estimators $\widetilde{\theta}_{i,\tau}^{(m)\ast}$, $\widetilde{v}_{i,\tau}^{(m)\ast}$ and $\widetilde{\gamma}_{i,\tau}^{(m)\ast}$, while other choices of finite $\tau$ may be considered with cautions as the temporal correlation ${\gamma}_{i,\tau}^{(m)}$ tends to $0$ when $\tau$ increases. 

To study the empirical powers of the autocovariance test, we notice that for two high-dimensional time series following factor models that are normalized to the canonical form (\ref{3:econf}), the difference between spiked eigenvalues $\mu_{i,\tau}^{(m)}, m=1, 2$ may arise from either the difference between factor strength $\left(\sigma_{i}^{(m)}\right)^2, m=1, 2$ or the difference between temporal autocorrelation $\gamma_{i, \tau}^{(m)}, m=1, 2$. Therefore, to empirically investigate the autocovariance test's power, we study two typical scenarios where either variances or autocorrelations of factors are different between two factor models. We are particularly interested in whether the autocovariance test's empirical power grows with the difference between variances or autocorrelations for two high-dimensional time series. 

Specifically, to explore the impacts of $\delta$, $N$ and $T$ on empirical powers, we again generate observations from two populations with $T = 400,800$, $N=200,400,800,$ and $\delta = 0, 0.1, 0.3, 0.5$. The data in the first population is generated by (\ref{3:dgpy}), which is precisely the same as we study the empirical sizes, while the data in the second population is generated with a different $\sigma_{1}^{(2)}$ or $\phi_{1}^{(2)}$ in the factor model. 
In the current work, we study the impact of variance numerically while left the investigation on the autocorrelations to the supplement. 
We keep the temporal autocorrelation unchanged, i.e. AR coefficient $\phi_{1}^{(2)}$ is the same as $\phi_{1}^{(1)}$ (i.e. $\phi_{1}^{(2)} = \phi_{1}^{(1)} = 0.5$), and set
$\left(\sigma_{1}^{(2)}\right)^2=1.1\left(\sigma_{1}^{(1)}\right)^2, 1.3\left(\sigma_{1}^{(1)}\right)^2, 1.5\left(\sigma_{1}^{(1)}\right)^2, 1.7\left(\sigma_{1}^{(1)}\right)^2, 1.9\left(\sigma_{1}^{(1)}\right)^2,$ respectively.
By doing that, we can investigate how the empirical powers of the autocovariance test are affected by the difference between variances of factors in two factor models. Moreover, it is worth to mention that when generating $\left\{f_{i,t}^{(1)}\right\}$ and $\left\{f_{i,t}^{(2)}\right\}$, $\left\{z_{1,t}^{(1)}\right\}$ are i.i.d.\ $\mathcal{N}\left(0,\left(\sigma_z^{(1)}\right)^2\right)$ with $\left(\sigma_z^{(1)}\right)^2 = 1/{\left(1-\left(\phi_1^{(1)}\right)^2\right)}$, whereas $\left\{z_{1,t}^{(2)}\right\}$ are i.i.d.\ $\mathcal{N}\left(0,\left(\sigma_z^{(2)}\right)^2\right)$ with $\left(\sigma_z^{(2)}\right)^2 = 1/{\left(1-\left(\phi_1^{(2)}\right)^2\right)}.$

To compute the empirical powers, for each combination of $T,N$ and $\delta$, two high-dimensional time series observations are generated 
first. Then, we can follow the estimation and testing procedures in Section~\ref{3:sec:3} and compute the test statistic $\widetilde{Z}_{i,\tau}$ by (\ref{3:teststat}), where again $B=500$ bootstrap samples are generated to find $\widetilde{\theta}_{i,\tau}^{(m)\ast}$, $\widetilde{v}_{i,\tau}^{(m)\ast}$, and $\widetilde{\gamma}_{i,\tau}^{(m)\ast}$ for both samples with the number of factors assumed to be known (i.e., $\widetilde{K}_{m} = 1$). Lastly, based on $M=500$ Monte Carlo simulations, the empirical powers of a one-sided autocovariance test for $i=1,$ $\tau=1$, and $\alpha=0.1$ can be estimated by the empirical probability that $\widetilde{Z}_{1,1}$ is less than $z_{\alpha}$, i.e.,
$$\frac{1}{M} \sum_{m=1}^M \mathbf{1}_{\left\{\widetilde{Z}_{1,1}(m) < z_{\alpha}\right\}},$$
where we have assumed $\mu_{1,1}^{(1)} < \mu_{1,1}^{(2)}$ for various choices of $\sigma_{1}^{(2)}$.

Empirical powers of the autocovariance test
with various choices of $N$, $T$, and $\delta$ are presented in Figures~\ref{3:f2}~to~\ref{3:f3}. 
It is clear that for all combinations of $N$ and $T$, empirical powers
increase towards $1$ when $\left(\sigma_{1}^{(2)}\right)^2$ increases from $1.1\left(\sigma_{1}^{(1)}\right)^2$ to $1.9\left(\sigma_{1}^{(1)}\right)^2$. Therefore, numerical results in Figure~\ref{3:f2}~and~\ref{3:f3} suggest that the autocovariance test can correctly reject the null hypothesis when two high-dimensional time series follow factor models with different variances of factors. 
Besides, despite the common temporal autocorrelation, for the same amount of increase in $\sigma_{1}^{(2)}/\sigma_{1}^{(1)}$, 
the empirical powers of one-sided autocovariance tests for $T = 800$ are generally higher than those associated with $T = 400$, which can be justified by the order $\sqrt{T}$ in (\ref{3:stat0}). In detail, a larger value of $T$ could incur a larger power. Also, the powers of stronger factor models with smaller $\delta$ are slightly higher than those of weaker factor models with larger $\delta$, especially for $T=400$.

\begin{figure}[!htbp]
	\centering
	\includegraphics[width=0.75\textwidth]{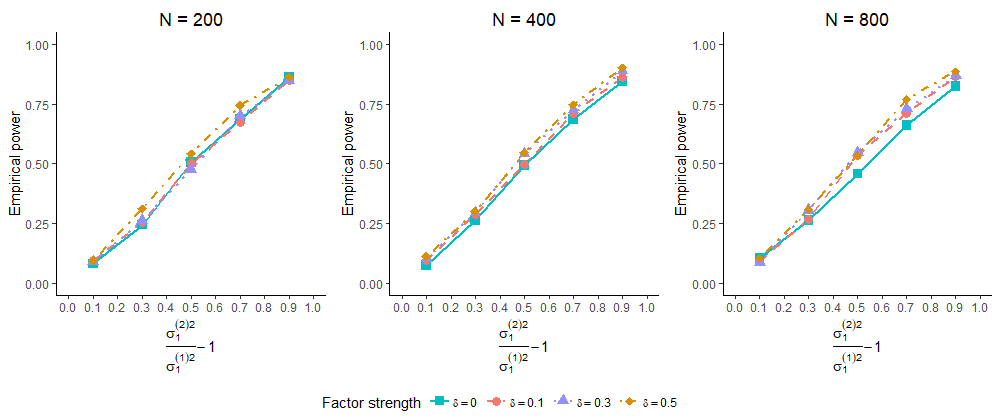}
	\caption{Empirical powers of the autocovariance test 
		with $T=400$, $N=200,400,800,$ and $\delta = 0, 0.1, 0.3, 0.5.$}\label{3:f2}
\end{figure}

\begin{figure}[!htbp]
	\centering
	\includegraphics[width=0.75\textwidth]{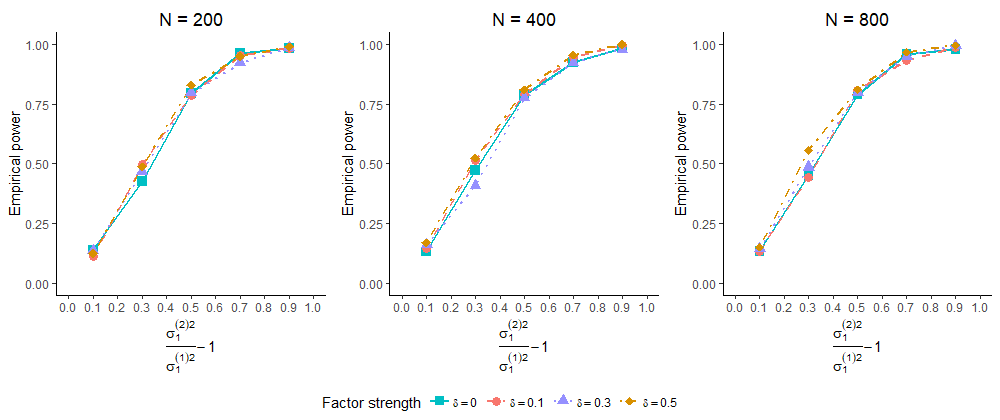}
	\caption{Empirical powers of the autocovariance test 
		with $T=800$, $N=200,400,800,$ and $\delta = 0, 0.1, 0.3, 0.5.$}\label{3:f3}
\end{figure}

%
%
	\section{Hierarchical Clustering for Multi-country Mortality Data}\label{3:sec:5}
To incorporate the proposed autocovariance test into hierarchical clustering analysis on real-world data, we study age-specific mortality rates from countries worldwide. In the past century, age-specific mortality rates have received massive attention, especially by insurance companies and governments, as accurate forecasting of mortality rates is crucial for the pricing of life insurance products and is highly related to social and economic policies. Among many works on forecasting age-specific mortality rates, the Lee-Carter model \citep{lee_modeling_1992} is prevalent and has been used globally. Despite some extensions on the original model \citep[see, e.g.,][]{hyndman_robust_2007, li_extending_2013}, one drawback of the Lee-Carter model is that it only focuses on the death rates of a single country, therefore may produce quite different long-run forecasts of mortality rates from different countries.



This section uses the proposed autocovariance test to explore multiple countries' mortality data, especially the spiked eigenvalues of the autocovariance matrices, and proposes a novel hierarchical clustering method for mortality data from different countries. To achieve this, we collect the total death rates for various countries from the Human Mortality Database \citep{HMD}. The plots of log mortality rates for Australia and Belgium are shown in Figure~\ref{3:f6} as an example, where similar patterns are observed for mortality rates in both countries. For the best quality of data, we choose the death rates from age 0 to 90 and require each country's sample size to be relatively large. 
As a consequence, we study countries with total death rates available from 1957 to 2017. Besides, 
for countries with small populations, zero death rates are replaced by the averages of death rates in adjacent years. In summary, the data we study has dimension $N=91$ and sample size $T=60$ for each country. 

\begin{figure}[!htbp]
	\centering
	\includegraphics[width=0.65\textwidth]{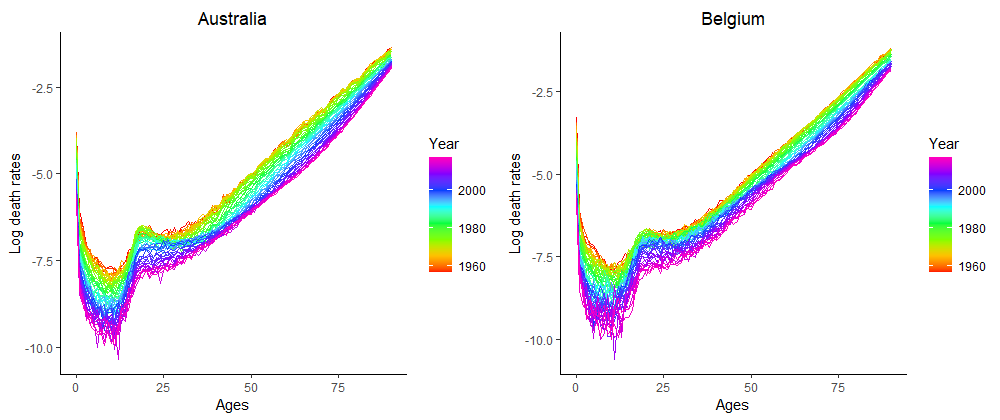}
	\caption{Observed time series of log death rates in Australia}\label{3:f6}
\end{figure}

According to the estimation and testing procedure in Section~\ref{3:sec:3}, factor models in canonical form (\ref{canonical form 4}) are firstly estimated and normalized from the differenced log death rates for each country. In the meantime, the number of factors in the factor model for each country is estimated and compared. As shown in Table~\ref{3:tab:1factor}, for most countries, there is only one factor estimated from the differenced log death rates, while there are some exceptions where two, three, and five factors are estimated. For countries with the same number of factors, we can compute the test statistic $\widetilde{Z}_{i,\tau}$ to test the equivalence of the temporal covariance in the temporal subspace. For the best accuracy in estimating the number of factors and temporal dependence among death rates, the autocovariance test is performed based on $\tau = 1$ throughout this section.

{\renewcommand{\arraystretch}{1.5}
	\begin{table}[!htbp]
		\centering
		\caption{Estimated number of factors in the factor model for each country}
		\label{3:tab:1factor}
		\resizebox{0.8\textwidth}{!}{%
			\begin{tabular}{@{}cl@{}}
				\toprule
				Estimated number of factors & \multicolumn{1}{c}{Countries}   \\ \midrule
				1                           & \begin{tabular}[c]{@{}l@{}}Australia, Belgium, Bulgaria, Czechia, Finland, Greece, Hungary,\\ Japan, Netherlands, Sweden, Switzerland, U.K. \end{tabular}     \\
				2                           & Denmark                         \\
				3                           & Canada, France, Italy, Portugal \\
				5                           & Poland                          \\ \bottomrule
			\end{tabular}%
		}
	\end{table}
}

For countries with one factor, the test statistic $\widetilde{Z}_{1,1}$ for each pair of countries can be computed by following the procedure in Section~\ref{3:sec:3}.
For all other countries with one factor, the $p$-values associated with all test statistics are computed. As illustrated in Figure~\ref{3:f7}, the spiked eigenvalues of the autocovariance matrices in the majority of European countries are similar as most $p$-values of test statistics between two European countries are greater than $0.1$. However, the $p$-values between Finland and Bulgaria, the U.K., and Finland are relatively small. Following the results of the autocovariance test between each pair of countries with the same number of factors, a hierarchical clustering method can be proposed where the dissimilarity can be measured using the $p$-value, such as ($1$-$p$)-value or ($1$/$p$)-value. For the analysis of mortality data, we define the dissimilarity for all countries with one factor as ($1$-$p$)-value, and the result of hierarchical clustering using average linkage for all countries with one factor is presented in Figure~\ref{3:f7}.   

\begin{figure}[!htbp]
	\begin{subfigure}{.5\textwidth}
	\centering
	\includegraphics[width=0.9\textwidth]{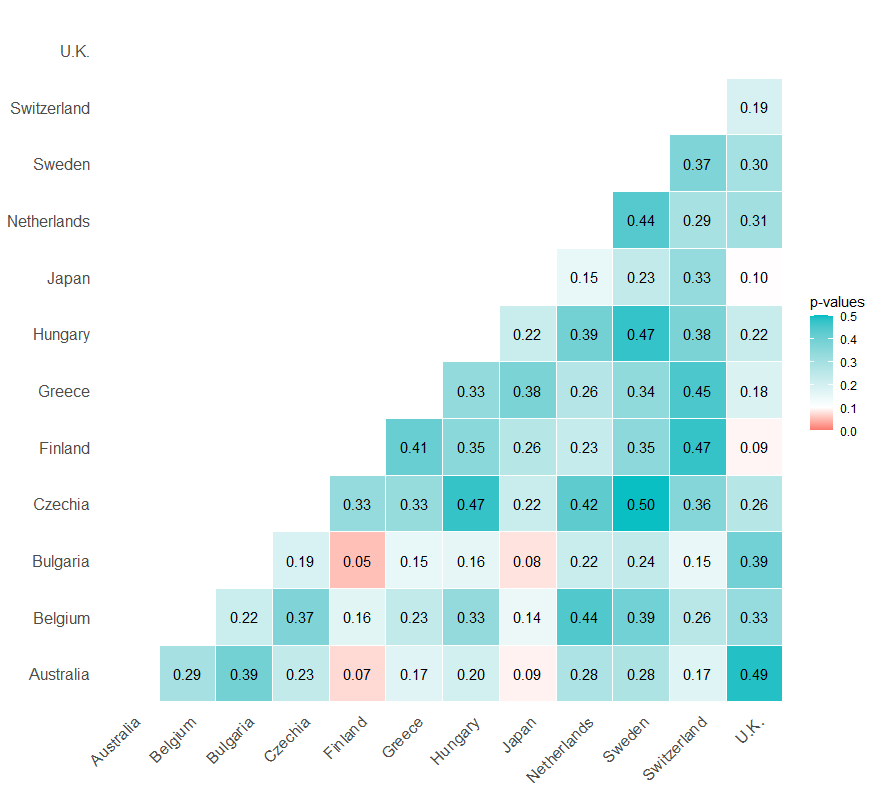}
	\end{subfigure}%
	\begin{subfigure}{.5\textwidth}
	\centering
	\includegraphics[width=0.9\textwidth]{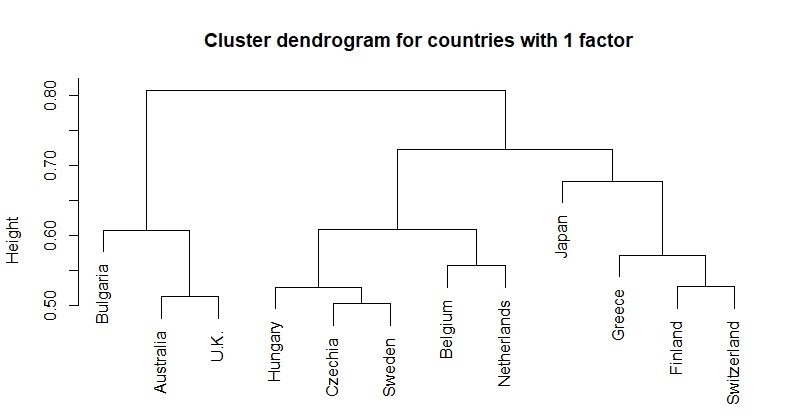}
	\end{subfigure}
	\caption{$p$-values of the autocovariance test and the cluster dendrogram for countries that have one factor in the estimated factor model}
	\label{3:f7}	
\end{figure}

For countries with more than one factor, a more sophisticated testing and clustering procedure can be developed to incorporate multiple test statistics for different factors. This procedure and the corresponding empirical results are, to some extent, beyond the scope of the current work hence is discussed in the supplement. 

	\section*{ACKNOWLEDGEMENTS}
	This work was partially supported by the National Natural Science Foundation of China (No. 72201093, 12001518).
	
	\section*{DATA AVAILABILITY STATEMENT}
	The data that support the findings of this study are openly available in the Human Mortality Database (HMD) at \url{https://mortality.org/}.
	
	\section*{CONFLICTS OF INTEREST STATEMENT}
	The authors declared that they have no conflicts of interest to this work.

	\bibliographystyle{chicago}
	\bibliography{references}
	
				\pagebreak
				\pagenumbering{arabic}
				\setcounter{page}{1}

				\begin{center}
						\title{ \textbf{\Large Supplement to ``Spiked eigenvalues of high-dimensional sample autocovariance matrices: CLT and applications''}  \newline}
						\author{Daning Bi, Xiao Han, Adam Nie, Yanrong Yang}
					\end{center}
				\maketitle
				
				This supplementary material contains technical proofs of results in the original paper `Spiked eigenvalues of high-dimensional sample autocovariance matrices: CLT and applications'. In Appendix \ref{section - proofs}, proofs of Lemma \ref{lemma - high prob event}, Theorem \ref{theorem - 2.1} and Proposition \ref{proposition - solution for theta} are presented. The proof of Theorem \ref{theorem - CLT} is located in Appendix \ref{CLT - proofs} and technical lemmas are collected in Appendixes~\ref{section - clt lemmas} and \ref{section - resolvents}. Lastly, Appendix \ref{3:sec:Appendix_A}~to~\ref{Appendix_app} collects technical proof and additional numerical results for the autocovariance test and its application on the multi-country mortality data.
				
				\begin{appendices}
\section{Proofs of Lemma~\ref{lemma - high prob event}, Theorem~\ref{theorem - 2.1} and Proposition~\ref{proposition - solution for theta}}
\label{section - proofs}

\begin{proof}[Proof of Lemma~\ref{lemma - high prob event}]
	\eqref{lemma - high prob event - EE} of the lemma can be found in \cite{CaiHanPan2017}, here we give a proof for \eqref{lemma - high prob event XX}.
	Since $X_0\tp X_0$ is symmetric and positive definite, we have
	\begin{align*}
		\norm{X_0\tp X_0} \le \tr(X_0 \tp X_0)
		=
		\frac{1}{T}\sum_{i=1}^K \sum_{t=1}^{T- \tau} x_{it}^2
		\le
		\sum_{i=1}^K\frac{1}{T} \sum_{t=1}^{T} x_{it}^2.
	\end{align*}
	By (\ref{lemma- concentration of xBx (XBX) }) of Lemma~\ref{lemma - concentration of xBx} (whose proof
	does not depend the current lemma) we have
	\begin{align*}
		\Em\left[ \frac{1}{T}\sum_{t=1}^T x_{it}^2\right] = \sigma_i^2+1,
		\quad
		\Var \left(\frac{1}{T}\sum_{t=1}^{T} x_{it}^2
		\right)
		=O\left(\frac{\sigma_i^4}{ T}\right).
		\numberthis\label{equation - proof 1.1 1}
	\end{align*}
	Taking a union bound we obtain
	\begin{align*}
		\Pm\Big(  \norm{X_0\tp X_0} >   2\sum_{i=1}^K \sigma_i^2 \Big)
		&\le
		\Pm\Big(  \sum_{i=1}^K\frac{1}{T} \sum_{t=1}^{T} x_{it}^2
		>   2\sum_{i=1}^K \sigma_i^2 \Big)
		\le
		\sum_{i=1}^K \Pm\left( \frac{1}{T} \sum_{t=1}^{T} x_{it}^2  >  2 \sigma_i^2 \right)
		\\
		         & =
		\sum_{i=1}^K \Pm\left( \frac{1}{T} \sum_{t=1}^{T} x_{it}^2 - (\sigma_i^2+1)  > \sigma_i^2 -1 \right).
	\end{align*}
	Finally by Chebyshev's inequality and \eqref{equation - proof 1.1 1} we have
	\begin{align*}
		\Pm\Big(  \norm{X_0\tp X_0} >   2\sum_{i=1}^K \sigma_i^2 \Big)= O\left( \frac{K \sigma_i^4 }{(\sigma_i^2-1)^2 T} \right)
		= O\left(\frac{K}{T}\right).
	\end{align*}
	and the proof is complete.
\end{proof}

\begin{proof}[Proof of Theorem~\ref{theorem - 2.1}]
	We shall write $\Lambda_n(A)$ for the $n$-th largest eigenvalue of a matrix $A$.
	Note that the non-zero eigenvalues of $\hat \Sigma_\tau \hat \Sigma_\tau\tp = Y_\tau Y_0\tp Y_0 Y_\tau \tp  $ coincide with those of the matrix $Y_0\tp Y_0 Y_\tau\tp Y_\tau= (X_0\tp X_0 + E_0\tp E_0)(X_\tau\tp X_\tau + E_\tau\tp E_\tau)$.
	We first show that  the eigenvalue $\Lambda_n(\hat \Sigma_\tau \hat \Sigma_\tau\tp)$ is close to  $\Lambda_n(X_0\tp X_0 X_\tau\tp X_\tau)$.
	By Weyl's inequality (Lemma B.1 of \cite{FanLiaoMincheva2013}) we have
	\begin{align*}
		 & \left|\Lambda_n(\hat \Sigma_\tau \hat \Sigma_\tau\tp )
		-
		\Lambda_n(X_0\tp X_0 X_\tau\tp X_\tau) \right|
		=
		\left|\Lambda_n(Y_0\tp Y_0 Y_\tau\tp Y_\tau)
		-
		\Lambda_n(X_0\tp X_0 X_\tau\tp X_\tau) \right|
		\\&\quad\qquad\le
		\norm{X_0\tp X_0 E_\tau\tp E_\tau  + E_0\tp E_0 X_\tau\tp X_\tau + E_0\tp E_0 E_\tau\tp E_\tau}
		= O_p(K \sigma_1^2		)
		,
	\end{align*}
	where the last equality follows from \eqref{equation - high prob event Op}. Dividing by $\mu_{n,\tau} = \sigma_n^4 \gamma_n(\tau)^2$ we have
	\begin{align*}
		\numberthis\label{equation - proof of theorem 2.1 -1}
		\frac{\Lambda_n(\hat \Sigma_\tau \hat \Sigma_\tau\tp )
			-
			\Lambda_n(X_\tau X_0\tp X_0 X_\tau\tp )}{\sigma_n^4 \gamma_n(\tau)^2}
		= O_p\left( \frac{K\sigma_1^2}{\sigma_n^4\gamma_n(\tau)^2}\right).
	\end{align*}
	Next we compute $\Lambda_n(X_\tau X_0\tp X_0 X_\tau\tp )$ in more details.
	It is shown in Lemma~\ref{lemma - concentration of xBx} that
	\begin{align*}
		(X_0 X_\tau\tp)_{ij} = \Em[(X_0 X_\tau\tp)_{ij}]
		+ O_{L^2}({\sigma_i \sigma_j}{ T^{-1/2}}),
		\numberthis\label{equation - place1}
	\end{align*}
	where from equation \eqref{equation - expectation of xoxtau 2} we know $\Em[(X_0 X_\tau\tp)_{ij}] = 1_{i=j}\sigma_i^2 \gamma_i(\tau)$.
	Therefore for any $i\ne j$, the off-diagonal elements of $X_\tau X_0\tp X_0 X_\tau\tp$  can be written into
	\begin{align*}
		(X_\tau & X_0\tp X_0 X_\tau\tp)_{ij}
		=
		\sum_{k=1}^K (X_0 X_\tau\tp)_{ki} (X_0 X_\tau\tp)_{kj}
		\\
		        & = (X_0 X_\tau\tp)_{ii} (X_0 X_\tau\tp)_{ij}
		+(X_0 X_\tau\tp)_{ji} (X_0 X_\tau\tp)_{jj}+
		\sum_{k\ne i,j} (X_0 X_\tau\tp)_{ki} (X_0 X_\tau\tp)_{kj}
		\\
		& = 
		O_{L^1}\left(\frac{\sigma_i^3 	\sigma_j\gamma_i(\tau) + \sigma_i\sigma_j^3 \gamma_j(\tau)}{\rt T}
		+ \frac{\sigma_i^3 \sigma_j + \sigma_i\sigma_j^3 + \sigma_i \sigma_j\norm{\sigmab}_{\ell_2}^2}{T}
		\right)
		  \\
		& = O_{L^1}\left(\frac{\sigma_i \sigma_j(\sigma_i^2 \gamma_i(\tau)+ \sigma_j^2 \gamma_j(\tau))}{\rt T}\right)
		\numberthis\label{equation - place1.2}
	\end{align*}
	where  the last line follows from Assumptions \ref{assumptions - tau fixed} and \ref{assumptions - tau div}. This gives
	\begin{align*}
		\mu_{i, \tau}^{-1/2} \mu_{j, \tau}^{-1/2}(X_\tau & X_0\tp X_0 X_\tau\tp)_{ij}
		=
		O_{L^1}\left(\frac{\sigma_i}{\sigma_j\gamma_j(\tau)\rt T}+
		\frac{\sigma_j }{\sigma_i \gamma_i(\tau)\rt T}
		\right)
	\end{align*}
	Similarly, the diagonal elements of  $X_\tau X_0\tp X_0 X_\tau\tp$ satisfy
	\begin{align*}
		(X_\tau X_0\tp  X_0 X_\tau\tp)_{ii}
		& = 
		(X_0 X_\tau\tp )_{ii}^2 + \sum_{k\ne i} (X_0 X_\tau\tp)_{ki}^2
		\\
		 & =\left(\sigma_i^2\gamma_i(\tau) + O_{L^2}(\sigma_i^2 T^{-1/2}) \right)^2 + O_{L^1}( \sigma_i^2 \norm{\sigmab}_{\ell_2}^2 T^{-1})
		   \\
		& =  \mu_{i, \tau} + O_{L_1} \left(\frac{\sigma_i^4 \gamma_i(\tau)}{\rt T}
		+ \frac{\sigma_i^4 + \sigma_i^2 \norm{\sigmab}_{\ell_2}^2}{T}
		\right).\numberthis\label{equation - place1.3}
	\end{align*}
	Using Assumptions \ref{assumptions - tau fixed} and \ref{assumptions - tau div} again we have
	\begin{align*}
		\mu_{i,\tau}^{-1}(X_\tau X_0\tp  X_0 X_\tau\tp)_{ii}
		= 1 + O_{L_1}\left(\frac{1}{\gamma_i(\tau)\rt T}\right).
	\end{align*}
	Using \eqref{equation - place1.2}, \eqref{equation - place1.3} and taking a union bound over $i,j$ we obtain
	\begin{align*}
		\mathrm {diag}(\mu_{i, \tau}^{-1/2}) X_\tau X_0\tp X_0 X_\tau\tp \mathrm {diag}(\mu_{i, \tau}^{-1/2}) =I_K
		+ O_{p, \norm{\cdot}_\infty} \left(\alpha_T\right),
		\numberthis\label{equation - proof of 2.1, infty estimate}
	\end{align*}
	where
	\begin{align*}
		\alpha_T:= \frac{K^2}{\rt T}\sup_{ij} \frac{\sigma_j}{\sigma_i \gamma_i(\tau)}.
	\end{align*}
	Let $\omega_1,\ldots, \omega_K$ be the eigenvalues of $X_\tau X_0\tp X_0 X_\tau\tp$ arranged in decreasing order. Let $\omega$ be one of these eigenvalues.
	Define the function
	\begin{align*}
		G(\omega):= \mathrm {diag}(\mu_{i, \tau}^{-1/2})\left(X_\tau X_0\tp X_0 X_\tau\tp  - \omega I_K\right)
		\mathrm {diag}(\mu_{i, \tau}^{-1/2}),
	\end{align*}
	then clearly we have $0 = |X_\tau X_0\tp X_0 X_\tau\tp - \omega I_K|
		= |G(\omega)|$.
	From \eqref{equation - proof of 2.1, infty estimate} we get
	\begin{align*}
		0  =  \left|G(\omega)\right|
		 & =
		\left|I_K + O_{p, \norm{\cdot}_\infty} \left(\alpha_T\right)
		- \omega \mathrm {diag}\big(\mu_{i, \tau}^{-1}\big)  \right|
		\\
		 & =  \left| I_K-  \mathrm {diag}
		(\omega \mu_i^{-1})
		+ O_{p, \norm{\cdot}_\infty} \left(\alpha_T\right) \right|,
	\end{align*}
	and using Leibniz's formula analogous to the derivation of \eqref{equation - diagonal is dominant in leibniz formula} we obtain
	\begin{align*}
		\numberthis\label{equation - detQ is small proof of 2.1}
		0= |G(\omega)| =
		\prod_{i=1}^K G(\omega)_{ii} + O_p\left(\alpha_T^2\right).
	\end{align*}

	Since $\prod_i G(\omega)_{ii} = o_p(1)$, there is at least one $i\in\{1, \ldots, K\}$ such that $G(\omega)_{ii} = o_p(1)$. Now we claim that in fact there can be only one such $i$. Indeed, suppose for a contradiction that for some $i<j$ we have $G(w)_{ii} =o_p(1)$ and $G(w)_{jj} = o_p(1)$. By Assumption \ref{assumptions - tau fixed} or \ref{assumptions - tau div} we know that
	\begin{align*}
		G(\omega)_{ii}- G(\omega)_{jj}
		= \omega (\mu_i^{-1} - \mu_j^{-1})
		\ge {\omega}{\mu^{-1}_i} \epsilon
		\numberthis\label{equation - Gii-Gjj in first Leibniz formula}
	\end{align*}
	for some $\epsilon>0$, which implies $\omega \mu_i^{-1} = o_p(1)$. However, this is clearly impossible since $G(\omega)_{ii}$ is assumed to be $o_p(1)$. 
	
	Therefore, for \eqref{equation - detQ is small proof of 2.1} to hold,
	there must exist
	some $i\in\{1,\ldots, K\}$ such that
	\begin{align*}
		0= G(\omega)_{ii} + O_p(\alpha_T^2),
	\end{align*}
	which gives one solution to $|G(\omega)|=0$. On the other hand, note that $|G(w)|=0$ has $K$ solutions in total. By the above arguments, it should be clear that each solution then corresponds to a particular $i\in \{1, \ldots, K\}$, i.e.
	 the $K$ solutions to $|G(\omega)|=0$ satisfy the system of equations
	\begin{align*}
		0= G(\omega)_{ii} + O_p( \alpha_T^2),
		\quad
		i=1,\ldots, K.
	\end{align*}
	Using \eqref{equation - place1.3}, we see that each $G(\omega)_{ii}$ satisfies
	\begin{align*}
		G(\omega)_{ii} = \frac{(X_\tau X_0\tp X_0 X_\tau\tp)_{ii}- \omega}{\mu_{i, \tau}} =  1-  \frac{\omega}{\mu_{i, \tau}}
		+O_{p}\left(\frac{1}{\gamma_i(\tau)\rt T}\right),
	\end{align*}
	which implies that the $K$ solutions to $|G(\omega)|=0$ satisfy the system of equations
	\begin{align*}
		\frac{\omega}{\mu_{i, \tau} } - 1 = O_{p}\left(\frac{1}{\gamma_i(\tau)\rt T}\right), \quad i=1,\ldots, K.
		\numberthis\label{equation - place 1.5}
	\end{align*}
	Note that by definition there are $K$ possible choices of $\omega$, which are the order eigenvalues of $X_\tau X_0\tp X_0 X_\tau\tp$. Since $\{\mu_{i, \tau}\}$ are ordered under Assumption \ref{assumptions - tau fixed} or asymptotically ordered under Assumption \ref{assumptions - tau div}, we can easily conclude that
	\begin{align*}
		\frac{\Lambda_i(X_\tau X_0\tp X_0 X_\tau\tp ) }{\mu_{i, \tau}}-1  = O_{p} \left(\frac{1 }{\gamma_i(\tau) \rt T }\right).
	\end{align*}
	Combining this result with \eqref{equation - proof of theorem 2.1 -1} we get
	\begin{align*}
		\frac{\Lambda_i(\hat \Sigma_\tau \hat \Sigma_\tau\tp )
		}{\mu_{i, \tau}}
		-1
		= O_p\left(\frac{1 }{\gamma_i(\tau)\rt T } +  \frac{K \sigma_1^2}{\sigma_i^4 \gamma_i(\tau)^2}\right)
	\end{align*}
	which completes the proof.
\end{proof}

\begin{proof}[proof of Proposition~\ref{proposition - solution for theta}]
		We first consider the invertibility of the matrix $Q(a)$ defined in \eqref{equation - definition of Q}. Recall the matrix $R(a)$ from \eqref{equation - definition of R} and the event $\cB_2$ from \eqref{equation - high prob event}.
		Since $a\asymp \sigma_n^4 \gamma_n(\tau)^2\to\infty$ and $\norm{E_\tau\tp E_\tau E_0\tp E_0 1_{\cB_2}}$ is bounded by definition of $\cB_2$, the matrix $I - a^{-1}E_\tau\tp E_\tau E_0\tp E_0$ is invertible under $\cB_2$ and
		we have $\norm{R(a)}1_{\cB_2}=O(1)$.
		Therefore we have
		$\norm{a^{-1} X_0 R_{a} E_\tau \tp E_\tau  X_0\tp}1_{\cB_2} = O(\sigma_n^{-2} \gamma_n(\tau)^{-2}) = o(1)$ and thus $Q(a)$ is invertible under $\cB_2$ for $T$ large enough.

		It will be shown in  Lemma~\ref{lemma - expectation of ABQ}  that
		\begin{align*}
			\Em[A(a)_{nn}1_{\cB_0}] = \frac{\sigma_n^2 \gamma_n(\tau)}{\rt a} + o(1),
			\quad
			\Em[B(a)_{nn}1_{\cB_0}] = o(1),
			\quad
			\Em[Q(a)_{nn}^{-1}1_{\cB_2}] = 1+ o(1).
		\end{align*}
		From \eqref{equation - interval for a around true theta} we have $a^{-1/2}{\sigma_n^2 \gamma_n(\tau)} \asymp O(1)$, using which we can obtain
		\begin{align*}
			g(a) =1- \Em[A(a)1_{\cB_0}]_{nn}^2 \Em[Q(a)^{-1}_{nn}1_{\cB_2}] + o(1) =1-  \frac{\sigma_n^4 \gamma_n(\tau)^2}{a} + o(1).
		\end{align*}
		Substituting the endpoints of the interval \eqref{equation - interval for a around true theta} into the function $g$, we have
		\begin{align*}
			g((1 \pm \epsilon) \sigma_n^4 \gamma_n(\tau)^2)
			= 1-  \frac{1}{1 \pm \epsilon} + o(1)=
			\frac{\mp \epsilon}{1 \pm \epsilon} + o(1).
		\end{align*}
		For $T$ large enough,
		the signs of $g$ differ at the two endpoints of the interval \eqref{equation - interval for a around true theta} and therefore $g$ has a root inside the interval. It is not difficult to observe that $g$ is a monotone function in $a$ for $T$ large enough which implies the root is unique.
\end{proof}

\section{Proof of Theorem~\ref{theorem - CLT}}\label{CLT - proofs}
We begin with the statements and proofs of the four propositions described in Section~\ref{section - main clt} of the paper. The proof of our main result Theorem~\ref{theorem - CLT} is given at the end of this appendix.

We first give an expression for $\delta:= \delta_{n, \tau}$. Recall the matrix $M$ from \eqref{equation - definition of M}
\begin{align*}
	M  := I_K - \frac{1}{\theta}X_\tau  E_0 \tp E_0  R X_\tau \tp  - \frac{1}{\theta} X_\tau  R \tp X_0 \tp Q^{-1} X_0  R X_\tau \tp.
\end{align*}
\begin{proposition}\label{proposition - detequation }
	Suppose Assumption \ref{assumptions} and either Assumption \ref{assumptions - tau fixed} or \ref{assumptions - tau div} hold. Then the ratio $\delta$ is the solution to the following equation
	\begin{align*}
		\det \left(   M   + \frac{\delta}{\theta} X_\tau  X_0\tp  X_0   X_\tau\tp   + \delta o_{p,\norm{\cdot}}\big(1\big) \right) =0.
		\numberthis\label{equation - determinant equation final}
	\end{align*}

	\begin{proof}
		Suppose $\lambda$ is an eigenvalue of $\widehat\Sigma_\tau  \widehat\Sigma_\tau \tp$, then $\rt{\lambda}$ is a singular value of the matrix $\widehat \Sigma_\tau $, or equivalently an eigenvalue of the $(2p+2K)\times (2p+2K)$ matrix
		\begin{align*}
			\begin{pmatrix}
				0 & \widehat \Sigma_\tau \\ \widehat \Sigma_\tau \tp  &0
			\end{pmatrix}
			=
			\begin{pmatrix}
				0               & 0               & X_\tau X_0 \tp & X_\tau  E_0\tp
				\\
				0               & 0               & E_\tau X_0 \tp & E_\tau  E_0\tp
				\\
				X_0  X_\tau \tp & X_0  E_\tau \tp & 0              & 0
				\\
				E_0  X_\tau \tp & E_0  E_\tau \tp & 0              & 0
			\end{pmatrix}.
		\end{align*}
		By definition the eigenvalue $\lambda$ satisfies
		\begin{align*}
			0=\left|
			\begin{pmatrix}
				\rt \lambda I_{K+p} & \0 \\  \0 & \rt \lambda I_{K+p} \tp
			\end{pmatrix}
			-
			\begin{pmatrix}
				0 & \widehat \Sigma_\tau \\ \widehat \Sigma_\tau \tp  &0
			\end{pmatrix}
			\right|
			=\left|
			\begin{pmatrix}
				\rt{\lambda}I_K & 0                 & -X_\tau X_0 \tp  & -X_\tau E_0 \tp
				\\
				0               & \rt{\lambda}I_{p} & -E_\tau X_0 \tp  & -E_\tau E_0 \tp
				\\
				-X_0 X_\tau \tp & -X_0 E_\tau \tp   & \rt{\lambda} I_K & 0
				\\
				-E_0 X_\tau \tp & -E_0 E_\tau \tp   & 0                & \rt{\lambda}I_{p}
			\end{pmatrix}
			\right|,
		\end{align*}
		which, after interchanging the columns and rows, becomes
		\begin{align*}
			0 & =
			\left|
			\begin{pmatrix}
				\rt{\lambda}I_K & -X_\tau X_0 \tp  & 0                 & -X_\tau E_0 \tp
				\\
				-X_0 X_\tau \tp & \rt{\lambda} I_K & -X_0 E_\tau \tp   & 0
				\\
				0               & -E_\tau X_0 \tp  & \rt{\lambda}I_{p} & -E_\tau E_0 \tp
				\\
				-E_0 X_\tau \tp & 0                & -E_0 E_\tau \tp   & \rt{\lambda}I_{p}
			\end{pmatrix}
			\right|.
			\numberthis\label{equation - proof of first prop, last placeholder}
		\end{align*}
		From Theorem~\ref{theorem - 2.1} we know that the spiked eigenvalue
		$\lambda\to\infty$ as $T\to\infty$.  From Lemma~\ref{lemma - high prob event} we recall that
		the spectral norm of $E_\tau E_0\tp$ is bounded with probability tending to 1 as $T\to\infty$.
		Therefore
		the bottom right sub-matrix
		$
			\left(\begin{smallmatrix}
					\rt \lambda I_p & -E_\tau E_0\tp \\
					-E_0E_\tau\tp & \rt \lambda I_p
				\end{smallmatrix}\right)
		$
		is invertible with probability tending to $1$. Using the matrix identity
		\begin{align*}
			\begin{pmatrix}
				A & B \\C&D
			\end{pmatrix}^{-1}
			=
			\begin{pmatrix}
				(A-B D^{-1}C)^{-1}          & - A^{-1}B(D- CA^{-1}B)^{-1}
				\\
				-D^{-1} C (A-BD^{-1}C)^{-1} & (D-CA^{-1}B)^{-1})
			\end{pmatrix}
		\end{align*}
		we can compute the inverse of the submatrix $
			\left(\begin{smallmatrix}
					\rt \lambda I_p & -E_\tau E_0\tp \\
					-E_0E_\tau\tp & \rt \lambda I_p
				\end{smallmatrix}\right)
		$ and get
		\begin{align*}
			 & \begin{pmatrix}
				\rt{\lambda}I_{p} & -E_\tau E_0 \tp
				\\
				-E E_\tau \tp     & \rt{\lambda}I_{p}
			\end{pmatrix}^{-1}
			\\ &\qquad=
			\begin{pmatrix}
				\left(\rt{\lambda}I_{p}- 	\frac{1}{\rt{\lambda}}E_\tau E_0 \tp  	E_0 E_\tau \tp \right)^{-1}
				 &
				\frac{1}{\rt{\lambda}}
				E_\tau E_0 \tp
				\left(\rt{\lambda}I_{p}- 	\frac{1}{\rt{\lambda}}E_0 E_\tau \tp  	E_\tau E_0 \tp \right)^{-1}
				\\
				\frac{1}{\rt{\lambda}} E_0 E_\tau \tp
				\left(\rt{\lambda}I_{p}- 	\frac{1}{\rt{\lambda}}E_\tau E_0 \tp  	E_0 E_\tau \tp \right)^{-1}
				 &
				\left(\rt{\lambda}I_{p}- 	\frac{1}{\rt{\lambda}}E_0 E_\tau \tp  	E_\tau E_0 \tp \right)^{-1}
			\end{pmatrix}
			\\
			 & \qquad=
			\begin{pmatrix}
				\rt{\lambda}\left({\lambda}I_{p}- E_\tau E_0 \tp  	E_0 E_\tau \tp \right)^{-1}
				 &
				E_\tau E_0 \tp \left({\lambda}I_{p}- 	E_0 E_\tau \tp  	E_\tau E_0 \tp \right)^{-1}
				\\
				E_0 E_\tau \tp
				\left({\lambda}I_{p}- 	E_\tau E_0 \tp  	E_0 E_\tau \tp \right)^{-1}
				 &
				\rt{\lambda}	\left({\lambda}I_{p}- 	E_0 E_\tau \tp  	E_\tau E_0 \tp \right)^{-1}
			\end{pmatrix}
			.
		\end{align*}
		Observe that
		\begin{align*}
			\begin{pmatrix}
				0 & \alpha \\ \beta&0
			\end{pmatrix}
			\begin{pmatrix}
				A & B \\ C &D
			\end{pmatrix}
			\begin{pmatrix}
				0 & \beta\tp \\ \alpha\tp &0
			\end{pmatrix}
			=
			\begin{pmatrix}
				\alpha D \alpha\tp & \alpha C \beta\tp \\ \beta B \alpha\tp  & \beta A \beta\tp
			\end{pmatrix}.
		\end{align*}
		Substituting the above computations back into \eqref{equation - proof of first prop, last placeholder} we have
		\scriptsize
		\begin{align*}
			0 & =
			\left|
			\begin{pmatrix}
				\rt{\lambda}I_K & -X_\tau X_0 \tp
				\\
				-X_0 X_\tau \tp & \rt{\lambda} I_K
			\end{pmatrix}
			-
			\begin{pmatrix}
				0               & -X_\tau E_0 \tp
				\\
				-X_0 E_\tau \tp & 0
			\end{pmatrix}
			\begin{pmatrix}
				\rt{\lambda}I_{p} & -E_\tau E_0 \tp
				\\
				-E_0 E_\tau \tp   & \rt{\lambda}I_{p}
			\end{pmatrix}^{-1}
			\begin{pmatrix}
				0               & -E_\tau X_0 \tp
				\\
				-E_0 X_\tau \tp & 0
			\end{pmatrix}
			\right|
			\\
			  & =
			\left|
			\begin{pmatrix}
				\rt{\lambda}I_K & -X_\tau X_0 \tp
				\\
				-X_0 X_\tau \tp & \rt{\lambda} I_K
			\end{pmatrix}
			-
			\begin{pmatrix}
				X_\tau E_0 \tp \rt{\lambda}	\left({\lambda}I_{p}- 	E_0 E_\tau \tp  	E_\tau E_0 \tp \right)^{-1}	E_0  X_\tau \tp
				 &
				X_\tau E_0 \tp E_0 E_\tau \tp  \left({\lambda}I_{p}- 	E_\tau E_0 \tp  	E_0 E_\tau \tp \right)^{-1}	E_\tau X_0 \tp
				\\
				X_0 E_\tau \tp E_\tau E_0 \tp \left({\lambda}I_{p}- 	E_0 E_\tau \tp  	E_\tau E_0 \tp \right)^{-1}	E_0 X_\tau \tp
				 &
				X_0 E_\tau \tp \rt{\lambda}\left({\lambda}I_{p}- E_\tau E_0 \tp  	E_0 E_\tau \tp \right)^{-1}	E_\tau X_0 \tp
			\end{pmatrix}
			\right|
			\\
			  & =
			\left|
			\begin{pmatrix}
				\rt{\lambda} (I_K	- X_\tau E_0 \tp
				({\lambda}I_{p}- 	E_0 E_\tau \tp  	E_\tau E_0 \tp)^{-1}	E_0 X_\tau \tp )
				 &
				-X_\tau
				(I_{T- \tau}+E_0 \tp E_0 E_\tau \tp  ({\lambda}I_{p}- 	E_\tau E_0 \tp  	E_0 E_\tau \tp )^{-1}	E_\tau
				)
				X_0 \tp
				\\
				-X_0
				(I_{T- \tau}+E_\tau \tp E_\tau E_0 \tp ({\lambda}I_{p}- 	E_0 E_\tau \tp  	E_\tau E_0 \tp )^{-1}	E_0
				)X_\tau \tp
				 &
				\rt{\lambda}(I_K - X_0 E_\tau \tp ({\lambda}I_{p}- E_\tau E_0 \tp  	E_0 E_\tau \tp )^{-1}	E_\tau X_0 \tp )
			\end{pmatrix}
			\right|,
			\\
			  & =
			\left|
			\begin{pmatrix}
				\rt{\lambda} (I_K	- X_\tau
				E_0 \tp E_0 ({\lambda}I_{T- \tau}- 	E_\tau \tp  	E_\tau E_0 \tp E_0 )^{-1}	X_\tau \tp )
				 &
				-X_\tau
				(I_{T- \tau}+E_0 \tp E_0 E_\tau \tp  E_\tau ({\lambda}I_{T- \tau}- 	E_0 \tp E_0 E_\tau \tp E_\tau )^{-1}
				)
				X_0 \tp
				\\
				-X_0
				(I_{T- \tau}+({\lambda}I_{T- \tau}- 	E_\tau \tp E_\tau E_0 \tp E_0 )^{-1}E_\tau \tp  	E_\tau E_0 \tp 	E_0
				)X_\tau \tp
				 &
				\rt{\lambda}(I_K - X_0 ({\lambda}I_{T- \tau}- E_\tau \tp E_\tau E_0 \tp  	E_0 )^{-1}	E_\tau \tp E_\tau X_0 \tp )
			\end{pmatrix}
			\right|,
		\end{align*}
		\normalsize
		where the last equality holds by \eqref{equation - bullet pencil identity}.
		Recalling the notations we introduced in Section~\ref{section - setting} and identity \eqref{equation - first resolvent identity}, we obtain
		\begin{align*}
			\numberthis\label{equation - det equation first form}
			0 & =  \left|
			\begin{pmatrix}
				\rt{\lambda}\ \overline Q_{\lambda}
				 &
				-X_\tau  R_{\lambda} \tp   X_0 \tp
				\\
				-X_0   R_{\lambda}  X_\tau \tp
				 &
				\rt{\lambda}  Q_{\lambda}
			\end{pmatrix}
			\right|
			=
			\left| \overline Q_\lambda
			- \lambda^{-1} X_\tau  R_{\lambda} \tp   X_0 \tp
			Q_\lambda^{-1}
			X_0   R_{\lambda}  X_\tau \tp
			\right|.
		\end{align*}

		Next, we center $\lambda$ around the quantity $\theta$ defined in \eqref{equation - definition of thetak}. Since $\lambda$ and $\theta$ diverge, they are outside of the spectrum of $E_\tau \tp E_\tau E_0 \tp E_0 $ with probability tending to 1.
		Then
		\begin{align*}
			\frac{1}{\lambda}R_\lambda & - \frac{1}{\theta}R
			=  (\lambda I_{T- \tau} - E_\tau \tp E_\tau  E_0 \tp E_0 )^{-1} - (\theta I_{T- \tau} - E_\tau \tp E_\tau  E_0 \tp E_0 )^{-1}
			\\
			                           & = (\theta - \lambda) (\lambda I_{T- \tau} - E_\tau \tp E_\tau  E_0 \tp E_0 )^{-1} (\theta I_{T- \tau} - E_\tau \tp E_\tau  E_0 \tp E_0 )^{-1}
			=   -  \frac{\delta}{\lambda} R_\lambda R.
		\end{align*}
		Substituting back into itself, we obtain
		\begin{align*}
			\frac{1}{\lambda}R_\lambda
			 & =  \frac{1}{\theta}R - \delta \left[\frac{1}{\theta}R -  \frac{\delta}{\lambda} R_\lambda R  \right]R
			=  \frac{1}{\theta}R -\frac{ \delta }{\theta} R^2 +   \frac{ \delta^2}{\lambda}R_{\lambda}R^2.
			\numberthis\label{equation - centering R/lambda}
		\end{align*}
		Using the bounds in \eqref{equation - high prob event Op} and \eqref{equation - norm of R 1cB} we have
		\begin{align*}
			\numberthis\label{equation - centering resolvent R-I is small}
			R - I_{T- \tau} = \frac{1}{\theta} E_\tau\tp E_\tau E_0\tp E_0 R = O_{p,\norm{\cdot}}(\theta^{-1}),
			\quad
			R^2 =I_{T- \tau}+ O_{p,\norm{\cdot}}(\theta^{-1}),
		\end{align*}
		where the second equation follows from expanding $(R - I)^2$.
		By Theorem~\ref{theorem - 2.1} we have $\delta = o_p(1)$. Substituting back into  \eqref{equation - centering R/lambda} we get
		\begin{align*}
			\frac{1}{\lambda}R_{\lambda}
			 & = \frac{1}{\theta}R-  \frac{\delta}{\theta} R^2 +  \delta o_{p,\norm{\cdot}}(\lambda^{-1})
			= \frac{1}{\theta}R- \frac{\delta }{\theta} I_{T- \tau} +  \delta o_{p,\norm{\cdot}}(\lambda^{-1}).
			\numberthis\label{equation - centering resolvent R/lambda}
		\end{align*}
		Using this we can get
		\begin{align*}
			\overline Q_{\lambda} - \overline Q
			 & =
			(I_K - X_\tau  E_0 \tp E_0  \lambda^{-1} R_\lambda X_\tau \tp ) - (I_K - X_\tau  E_0 \tp E_0  \theta^{-1}R X_\tau \tp )
			\\
			 & = X_\tau  E_0 \tp E_0   ( \theta^{-1}R - \lambda^{-1} R_\lambda   )X_\tau  \tp
			= X_\tau  E_0 \tp E_0   \left[ \frac{\delta}{\theta}I_{T - \tau}
			+  \delta o_{p,\norm{\cdot}}(\lambda^{-1})\right] X_\tau  \tp.
		\end{align*}
		From \eqref{equation - high prob event Op}
		we recall that $\norm{X_\tau}^2 = O_p(K\sigma_1^2)$. Using $\delta = o_p(1)$ again we get
		\begin{align*}
			\overline Q_{\lambda} & = \overline Q +   \frac{\delta}{\theta}X_\tau  E_0 \tp E_0  X_\tau \tp
			+\delta o_{p,\norm{\cdot}}(K\sigma_1^2\lambda^{-1})
			= \overline Q +  \delta o_{p,\norm{\cdot}}(1)
			,
		\end{align*}
		and similarly $Q_{\lambda}
			=  Q +  \delta o_{p,\norm{\cdot}}(1).$
		Finally, since $\norm{Q_\lambda^{-1}} = O_{p}(1)$, we have
		\begin{align*}
			 & Q_{\lambda}^{-1} - Q^{-1}
			=
			Q_{\lambda}^{-1}(Q - Q_{\lambda})  Q^{-1}
			=o_{p,\norm{\cdot}}(1).
			\numberthis\label{equation - centering resolvent Q inv}
		\end{align*}

		Next we consider the matrix $X_0 R_\lambda X_\tau\tp$ appearing in \eqref{equation - det equation first form}.
		From \eqref{equation - centering R/lambda} we have
		\begin{align*}
			\numberthis\label{equation - centering fo XRX}
			\frac{\rt \theta}{\lambda}X_0  R_\lambda   X_\tau \tp
			=
			\frac{1}{\rt \theta} X_0 R X_\tau \tp-\frac{ \delta }{\rt \theta} X_0 R^2X_\tau \tp
			+   \frac{ \delta^2\rt \theta}{\lambda} X_0 R_{\lambda} R^2X_\tau \tp.
		\end{align*}
		For the second term on the right hand side of
		\eqref{equation - centering fo XRX}, using \eqref{equation - centering resolvent R-I is small} and \eqref{equation - high prob event Op} we have
		\begin{align*}
			\frac{ \delta }{\rt \theta} X_0 R^2X_\tau \tp
			 & =
			\frac{ \delta }{\rt \theta} X_0 X_\tau \tp +
			\frac{ \delta }{\rt \theta} X_0 (R^2-I)X_\tau\tp
			\\
			 & = \frac{ \delta }{\rt \theta} X_0 X_\tau \tp
			+
			\delta O_{p, \norm{\cdot}}\left(\frac{K \sigma_1^2}{\theta^{3/2}}\right)
			=\frac{ \delta }{\rt \theta} X_0 X_\tau \tp +\delta o_{p, \norm{\cdot}}(1).
		\end{align*}
		Similarly the last term in \eqref{equation - centering fo XRX} satisfies
		$\frac{ \delta^2\rt \theta}{\lambda} X_0 R_{\lambda} R^2X_\tau \tp
			=\delta o_{p,\norm{\cdot}}(1)$. Therefore
		\begin{align*}
			\frac{\rt \theta}{\lambda}X_0  R_\lambda   X_\tau \tp
			=
			\frac{1}{\rt \theta} X_0 R X_\tau \tp
			-\frac{ \delta }{\rt \theta} X_0 X_\tau \tp
			+   \delta o_{p, \norm{\cdot}}(1).
			\numberthis\label{equation - centering fo XRX order of XRX 2}
		\end{align*}

		To deal with the second term appearing in the determinant in \eqref{equation - det equation first form}, we first make the following computations.
		Using \eqref{equation - centering fo XRX}-\eqref{equation - centering fo XRX order of XRX 2} as well as  \eqref{equation - centering resolvent Q inv} we have
		\begin{align*}
			\frac{\theta}{\lambda^2} & X_\tau  R_\lambda\tp   X_0 \tp
			Q_\lambda^{-1} X_0  R_\lambda   X_\tau \tp
			= \Big(
			\frac{1}{\rt \theta} X_\tau R\tp  X_0\tp
			- \frac{\delta}{\rt \theta} X_\tau   X_0\tp
			+  \delta o_{p,\norm{\cdot}}(1)
			\Big)
			\\
			                         & \qquad\times
			\Big(
			Q^{-1}
			+ \delta o_{p,\norm{\cdot}}(1)
			\Big)
			\Big(
			\frac{1}{\rt \theta} X_0 R X_\tau\tp
			-  \frac{\delta}{\rt \theta} X_0    X_\tau\tp
			+  \delta o_{p,\norm{\cdot}}(1)
			\Big)
			\\
			                         & = \Big(
			\frac{1}{\rt \theta} X_\tau R\tp  X_0\tp
			- \frac{\delta}{\rt \theta} X_\tau   X_0\tp
			\Big)
			Q^{-1}
			\Big(
			\frac{1}{\rt \theta} X_0 R X_\tau\tp
			-  \frac{\delta}{\rt \theta} X_0    X_\tau\tp
			\Big)
			+  \delta o_{p,\norm{\cdot}}(1).
		\end{align*}
		Expanding the expression above and using \eqref{equation - centering fo XRX}-\eqref{equation - centering fo XRX order of XRX 2} again we obtain
		\begin{align*}
			\frac{\theta}{\lambda^2} & X_\tau  R_\lambda\tp   X_0 \tp
			Q_\lambda^{-1} X_0  R_\lambda   X_\tau \tp
			=\frac{1}{ \theta} X_\tau R\tp  X_0\tp Q^{-1} X_0 R X_\tau\tp
			\\
			                         & \qquad
			- \frac{\delta}{\theta}
			\Big(X_\tau R\tp  X_0\tp  Q^{-1} X_0   X_\tau\tp
			+X_\tau   X_0\tp Q^{-1}  X_0 R X_\tau\tp  \Big)
			+  \delta o_{p,\norm{\cdot}}(1)
			\\
			                         & =
			\frac{1}{ \theta} X_\tau R\tp  X_0\tp Q^{-1} X_0 R X_\tau\tp
			- \frac{2\delta}{\theta} X_\tau  X_0\tp  Q^{-1} X_0   X_\tau\tp
			+  \delta o_{p,\norm{\cdot}}(1)
		\end{align*}
		Finally, recalling $\lambda/\theta = 1+ \delta$, we can conclude
		\begin{align*}
			\frac{1 }{\lambda}X_\tau & R_{\lambda}\tp   X_0 \tp Q_{\lambda}^{-1}X_0   R_{\lambda} X_\tau \tp
			=
			(1+ \delta) \frac{\theta }{\lambda^2} X_\tau   \tp R_\lambda\tp   X_0 \tp Q_{\lambda}^{-1} X_0   R_\lambda X_\tau \tp
			\\
			                         & = \frac{1}{ \theta} X_\tau R\tp  X_0\tp Q^{-1} X_0 R X_\tau\tp
			- \frac{2\delta}{\theta} X_\tau  X_0\tp  Q^{-1} X_0   X_\tau\tp
			\\
			                         & \qquad+ \delta \left(\frac{1}{ \theta} X_\tau R\tp  X_0\tp Q^{-1} X_0 R X_\tau\tp
			- \frac{2\delta}{\theta} X_\tau  X_0\tp  Q^{-1} X_0   X_\tau\tp\right)
			+  \delta o_{p,\norm{\cdot}}(1)
			\\
			                         & = \frac{1}{ \theta} X_\tau R\tp  X_0\tp Q^{-1} X_0 R X_\tau\tp
			- \frac{\delta}{\theta} X_\tau  X_0\tp   X_0   X_\tau\tp
			+  \delta o_{p,\norm{\cdot}}(1),
		\end{align*}
		where in the last line we have used \eqref{equation - centering resolvent R-I is small}-\eqref{equation - centering fo XRX order of XRX 2} again.
		To conclude, we have shown
		\begin{align*}
			\overline Q_{\lambda} & = I_K - \frac{1}{\theta} X_\tau E_0\tp E_0 R X_\tau\tp +  \delta o_{p,\norm{\cdot}}(1),
		\end{align*}
		for the first term in the right hand side of \eqref{equation - det equation first form} and
		\begin{align*}
			\frac{1 }{\lambda} & X_\tau  R_{\lambda}\tp   X_0 \tp Q_{\lambda}^{-1}X_0   R_{\lambda} X_\tau \tp
			= \frac{1}{ \theta} X_\tau R\tp  X_0\tp Q^{-1} X_0 R X_\tau\tp
			- \frac{\delta}{\theta} X_\tau  X_0\tp   X_0   X_\tau\tp
			+  \delta o_{p,\norm{\cdot}}(1)
		\end{align*}
		for the second term. The claim then follows.
	\end{proof}
\end{proposition}

We now work towards establishing the asymptotic distribution of the matrix $M$ from \eqref{equation - definition of M} with the help of Lemma  \ref{lemma - concentration of xBx}-\ref{lemma - conditional expectations}.
For notational convenience we define
\begin{align*}
	\numberthis\label{equation - definition of A B}
	A:= \frac{1}{\rt{\theta}}X_0 R X_\tau\tp,
	\quad B:= \frac{1}{\theta} X_\tau E_0\tp E_0 R X_\tau\tp,
\end{align*}
so that $M = I_K - B - A\tp Q^{-1}A$. For each $i=1,\ldots, K$, define
\begin{align*}
	\overline M_{ii} :=1- \Em[B_{ii}1_{\cB_0}] - \Em[A_{ii}1_{\cB_0}]^2 \Em[Q_{ii}^{-1}1_{\cB_2}]
	\numberthis\label{equation - decomposition of Mii}
\end{align*}
which serves as a deterministic centering for the $i$-th diagonal entry of  $M$. We first give an approximation for $M_{ii} - \overline M_{ii}$ up to the scaling of $T^{-1/2}$.

\begin{proposition}
	\label{proposition - M}
	Under Assumption \ref{assumptions} and either Assumption \ref{assumptions - tau fixed} or \ref{assumptions - tau div},
	we have
	\begin{align*}
		\numberthis\label{equation - expression for Mii - Miibar}
		{M_{ii}  - \overline M_{ii}}
		 & =
		- 2\big(A_{ii} - \underline\Em [A_{ii}]\big)
		{\Em[A_{ii}1_{\cB_0}]}{\Em [Q_{ii}^{-1} 1_{\cB_2}]}
		+ O_{p}\left(\frac{\sigma_i^2 \norm{\sigmab}_{\ell_1}^2}{\theta T}+ \frac{\sigma_i^2}{\theta \rt T} +KT^{-1}\right),
	\end{align*}
	for all $i=1,\ldots, K$,
	where $\underline \Em[\ \cdot \ ]$ is  defined in \eqref{equation - Emcond}.
	Furthermore,
	\begin{align*}
		\max_{i\ne j} |M_{ij}|
		=  O_p\left( \frac{K^4 \sigma_1^4 \sigma_i \sigma_j}{\theta^2 \rt T}\right).
	\end{align*}

	\begin{proof}
		We first recall from Lemma~\ref{lemma - expectation of ABQ}
		and Assumption \ref{assumptions} that
		\begin{align*}
			\Em[A_{ij}1_{\cB_0}] & =  1_{i=j}\left(\frac{\sigma_i^2 \gamma_i(\tau)}{\theta^{1/2}} + o(1)\right),
			\quad
			\Var(A_{ij} 1_{\cB_0})  =   O\left(\frac{\sigma_{i}^2 \sigma_j^2}{\theta T}\right) .
			\numberthis\label{equation - proof of AQA 1}
		\end{align*}
		We also recall from Lemma~\ref{lemma - concentration of Q inverse} that
		\begin{align*}
			\numberthis\label{equation - proof of AQA 2}
			Q_{kk}^{-1}1_{\cB_2}  = \underline \Em  [Q_{kk}^{-1}1_{\cB_2}] + O_{L^1}\left(\frac{\sigma_k^2 }{\theta \rt T}\right),
			\quad
			Q_{ij}^{-1} 1_{\cB_2} = O_{L^2}\left(\frac{\sigma_i \sigma_j}{\theta \rt T}\right).
		\end{align*}
		Recall that $M = I_K - B - A\tp Q^{-1}A$.
		We first consider the $i$-th diagonal of $A\tp Q^{-1} A$  and show that it is close to $A_{ii}^2 \underline E[Q_{ii}^{-1} 1_{\cB_2}]$ under the event $\cB_2$. Note that we can write
		\begin{align*}
			(A\tp & Q^{-1} A)_{ii}
			=  \sum_{m,n} A_{mi} A_{ni} Q^{-1}_{mn}
			\\
			      & =   A_{ii}^2 Q^{-1}_{ii} + \sum_{\substack{m,n\ne i}}A_{mi} A_{ni} Q^{-1}_{mn}
			+ A_{ii}\Big(\sum_{n\ne i} A_{ni} Q^{-1}_{in} + \sum_{m\ne i} A_{mi} Q^{-1}_{mi}\Big).
			\numberthis\label{equation -equation - proof of AQA 2.5 }
		\end{align*}
		We will consider each term in \eqref{equation -equation - proof of AQA 2.5 } separately.
		Recall from \eqref{equation - norm of Q 1cB} that $\norm{Q^{-1} 1_{\cB_2}} = 1+ o(1)$ which implies $Q^{-1}_{ij}1_{\cB_2} = 1+o(1)$ for all $i,j\le K$. Using the triangle inequality followed by the Cauchy Schwarz inequality we have
		\begin{align*}
			\Em\big|\sum_{\substack{m,n\ne i}}&A_{mi} A_{ni} Q^{-1}_{mn}1_{\cB_2}\big|
			\lesssim
			\sum_{\substack{m,n\ne i}}
			{ \Em [A_{mi}^2 1_{\cB_0}]}^{\frac{1}{2}}
			\Em  [A_{ni}^2 1_{\cB_0}]^{\frac{1}{2}}
			\lesssim \frac{\sigma_i^2 }{\theta T} \norm{\sigmab}_{\ell_1}^2,
		\end{align*}
		where the first inequality follows since $1_{\cB_2}\le 1_{\cB_0}$ and the last equality follows from \eqref{equation - proof of AQA 1}.
		
		For the last term in \eqref{equation -equation - proof of AQA 2.5 }, we note that $1_{\cB_2} = 1_{\cB_2}1_{\cB_0}$ by definition.  Using $\norm{Q^{-1} 1_{\cB_2}} = 1+ o(1)$ and the triangle inequality we have
		\begin{align*}
			\Em \big|\sum_{n\ne i}A_{ii} A_{ni} Q^{-1}_{in}1_{\cB_2}\big|
			&\lesssim
			\sum_{n\ne i}  \Em \big| (A_{ii} 1_{\cB_0}- \Em[A_{ii}1_{\cB_0}])  A_{ni} 1_{\cB_2}\big|+
			\big| \Em[A_{ii}1_{\cB_0}] \big|\sum_{n\ne i}\Em \big| A_{ni}Q^{-1}_{in}1_{\cB_2}\big|.
		\end{align*}
		By the Cauchy Schwarz inequality and \eqref{equation - proof of AQA 1} we have
		\begin{align*}
			\Em \Big|A_{ii}\sum_{n\ne i} A_{ni} Q^{-1}_{in}1_{\cB_2}\Big|
			  &\lesssim
			\sum_{n\ne i}  \Var(A_{ii}1_{\cB_0})^{1/2} \Em [A_{ni}^21_{\cB_0}]^{1/2 }
			+
			\Em|A_{ii}^21_{\cB_0}|^{1/2}\left(\sum_{n\ne i} \Em[A_{ni}^21_{\cB_0}]^{1/2} \Em[(Q^{-1})_{in}^21_{\cB_2}]^{1/2}\right)
			  \\
			&\lesssim  
			\sum_{n\ne i}\frac{ \sigma_i^3 \sigma_n }{\theta T}
			+
			\left(\frac{\sigma_i^2}{\theta^{1/2}} + \frac{\sigma_i^2}{\theta^{1/2} \rt T}\right)
			\sum_{n\ne i}
			\frac{\sigma_i^2 \sigma_n^2 }{\theta T}
			\lesssim 
			\frac{ \sigma_i^3  \norm{\sigmab}_{\ell_1}}{\theta T}
			+
			\frac{\sigma_i^4 \norm{\sigmab}_{\ell_2}^2}{\theta^{3/2} T}
			.
		\end{align*}
		Substituting back into \eqref{equation -equation - proof of AQA 2.5 } we obtain
		\begin{align*}
			(A\tp 	Q^{-1} A)_{ii}  1_{\cB_2}
			 & =   A_{ii}^2 1_{\cB_2}\underline\Em[Q^{-1}_{ii}1_{\cB_2}]   + O_{L^1}\left(\frac{\sigma_i^2 \norm{\sigmab}_{\ell_1}^2}{\theta T} \right).
		\end{align*}
		From Lemma~\ref{lemma - high prob event} we know that $1_{\cB_2} = 1+o_p(1)$, therefore
		\begin{align*}
			(A\tp 	Q^{-1} A)_{ii}
			 & =   A_{ii}^2 \underline\Em[Q^{-1}_{ii}1_{\cB_2}]+ O_{p}\left(\frac{\sigma_i^2 \norm{\sigmab}_{\ell_1}^2}{\theta T} +KT^{-1}\right).
			\numberthis\label{equation - first approx  of AQA - proof of M - Mbar}
		\end{align*}
		Next, we expand $A_{ii}^2$ around the conditional mean $\underline\Em[A_{ii}1_{\cB_0}]^2$. Write
		\begin{align*}
			A_{ii}^21_{\cB_0} = \underline\Em[A_{ii}1_{\cB_0}]^2 &+ 2 \underline\Em[A_{ii}1_{\cB_0}](A_{ii}1_{\cB_0}- \underline\Em[A_{ii}1_{\cB_0}]) 
			\\&+ (A_{ii}1_{\cB_0}- \underline\Em[A_{ii}1_{\cB_0}])^2.
			\numberthis\label{equation -  approx  of AQA expanding A2 - proof of M - Mbar}
		\end{align*}
		Recall from (\ref{lemma- concentration of xBx (XRX) }) of Lemma~\ref{lemma - concentration of xBx} that the last term satisfies
		\begin{align*}
			\left(A_{ii}1_{\cB_0} - \underline\Em[A_{ii}1_{\cB_0}]\right)^2
			= O_{L^1} \left(   \frac{\sigma_i^4}{\theta T}  \right).
		\end{align*}
		Next, we note that by definition of $\underline\Em$ and $\cB_0$, we have
		\begin{align*}
			A_{ii}1_{\cB_0}- \underline\Em[A_{ii}1_{\cB_0}]
			=
			\left(A_{ii}- \underline\Em[A_{ii}]\right)1_{\cB_0}
			=
			A_{ii}- \underline\Em[A_{ii}] + O_p(KT^{-1}),
		\end{align*}
		where the last equality follows from Lemma~\ref{lemma - high prob event}. Therefore from \eqref{equation -  approx  of AQA expanding A2 - proof of M - Mbar} we may obtain
		\begin{align*}
			A_{ii}^2 = \underline\Em[A_{ii}1_{\cB_0}]^2 + 2 \underline\Em[A_{ii}1_{\cB_0}](A_{ii}- \underline\Em[A_{ii}]) + O_p\left(   \frac{\sigma_i^4}{\theta T}+KT^{-1}  \right).
		\end{align*}
		Substituting back into \eqref{equation - first approx  of AQA - proof of M - Mbar} we have
		\begin{align*}
			(A\tp  Q^{-1} A)_{ii}
			 & =
			\underline\Em[A_{ii}1_{\cB_0}]^2\underline\Em[Q_{ii}^{-1}1_{\cB_2}] + 2 \underline\Em[A_{ii}1_{\cB_0}]\underline\Em[Q_{ii}^{-1}1_{\cB_2}](A_{ii}- \underline\Em[A_{ii}])
			+ O_{p}\left(\frac{\sigma_i^2 \norm{\sigmab}_{\ell_1}^2}{\theta T} +KT^{-1}\right)
			\\
			 & = \Em[A_{ii}1_{\cB_0}]^2\Em[Q_{ii}^{-1}1_{\cB_2}] + 2 \Em[A_{ii}1_{\cB_0}]\Em[Q_{ii}^{-1}1_{\cB_2}](A_{ii}- \underline\Em[A_{ii}])
			+ O_{p}\left(\frac{\sigma_i^2 \norm{\sigmab}_{\ell_1}^2}{\theta T} +KT^{-1}\right),
		\end{align*}
		where in the last equality, Lemma~\ref{lemma - conditional expectations} is used to replace the conditional expectations with the unconditional ones (except for the centering of $A_{ii}$ where the conditional expectation is intentionally kept).

		Finally, we recall $M_{ii} = 1 - B_{ii} - (A\tp Q^{-1} A)_{ii}$, so it remains to consider the matrix $B= \frac{1}{\theta} X_\tau E_0\tp E_0 R X_\tau\tp$ in the same manner as above.
		By Lemma~\ref{lemma - concentration of xBx},
		we have
		\begin{align*}
			\numberthis\label{equation - B in expression for M}
			\Em \big|B_{ij}1_{\cB_0} - 1_{i=j}\underline\Em[B_{ii}1_{\cB_0}] \big|^2
			 & \lesssim \frac{1}{\theta^2 T^2}O(\sigma_i^2 \sigma_j^2 T)
			 =
			 \frac{\sigma_i^2 \sigma_j^2}{\theta^2 T}.
		\end{align*}
		Using Lemma~\ref{lemma - conditional expectations} to replace $\underline \Em[B_{ii}1_{\cB_0}]$ with $\Em[B_{ii}1_{\cB_0}]$ and $1_{\cB_0} = 1- o_p(1)$,
		 we get
		\begin{align*}
			B_{ij} = 1_{i=j}\Em[{B_{ii}}1_{\cB_0}] + O_p\left( \frac{\sigma_i \sigma_j}{\theta \rt T}\right).
		\end{align*}
		Combining the above computations, we get
		\begin{align*}
			M_{ii} = \overline M_{ii} - 2 \Em[A_{ii}1_{\cB_0}]\Em[Q_{ii}^{-1}1_{\cB_2}](A_{ii}- \underline\Em[A_{ii}]) 
			+ O_{p}\left(\frac{\sigma_i^2 \norm{\sigmab}_{\ell_1}^2}{\theta T}+ \frac{\sigma_i^2 }{\theta \rt T} +KT^{-1}\right)
		\end{align*}
		and the first claim follows.

		For the off-diagonal elements, write
		\begin{align*}
			(A\tp & Q^{-1} A)_{ij}
			=  \sum_{m,n} A_{mi} A_{nj} Q^{-1}_{mn}
			=A_{ii}A_{jj} Q^{-1}_{ij} + A_{ii} A_{ij} Q_{ii}^{-1} +A_{ji}A_{jj} Q_{jj}^{-1}
			\\&
			\quad
			+ \sum_{{m\ne i,n\ne j, m\ne n}}A_{mi} A_{nj} Q^{-1}_{mn}
			+ A_{ii} \sum_{n\ne i, j}  A_{nj} Q^{-1}_{in}
			+ A_{jj}\sum_{m\ne i,j} A_{mi} Q^{-1}_{mj}.
			\numberthis\label{equation - offdiag of AQA decomp}
		\end{align*}
		Observe that
		by definition of $A_{ii}$, $Q_{ii}$ and the event $\cB_2$ we have
		\begin{align*}
			A_{ii}1_{\cB_2}
			=O\left(\frac{K \sigma_1^2}{ \rt \theta}\right) ,
			\quad Q^{-1}_{ij} 1_{\cB_2} = O(1).
			\numberthis\label{equation - high prob bound of Aii Qii}
		\end{align*}
		Recall
		from \eqref{equation - proof of AQA 1},
		Lemma~\ref{lemma - concentration of Q inverse}
		and Lemma~\ref{lemma - expectation of ABQ}
		that
		\begin{align*}
			A_{ij}1_{\cB_2} = O_{L^2}\left(\frac{\sigma_i \sigma_j}{\theta^{1/2}\rt T}\right),
			\quad
			Q_{ij}^{-1}1_{\cB_2} = O_{L^2}\left(\frac{\sigma_i \sigma_j}{\theta\rt T}\right),
			\quad \forall i\ne j.
			\numberthis\label{equation - L2 bound of Aij Qij}
		\end{align*}
		Substituting \eqref {equation - high prob bound of Aii Qii} and \eqref{equation - L2 bound of Aij Qij} back into the terms in \eqref{equation - offdiag of AQA decomp} we have
		\begin{align*}
			A_{ii}A_{jj} Q^{-1}_{ij}1_{\cB_2} =
			O_{L^2} \left(\frac{K^2 \sigma_1^4 \sigma_i \sigma_j}{\theta^2 \rt T}\right)
			,
			\quad
			A_{ii} A_{ij} Q^{-1}_{ii}1_{\cB_2}
			=
			O_{L^1} \left(\frac{K \sigma_1^2 \sigma_i^2 \sigma_j^2}{\theta^2   T}\right).
		\end{align*}
		With similar computation, the rest of \eqref{equation - offdiag of AQA decomp} are all of smaller order than the above. 
		Substituting the above four estimates back into
		\eqref{equation - offdiag of AQA decomp} we get
		\begin{align*}
			(A\tp Q^{-1} A)_{ij} = O_{L^1}\left(\frac{K^2 \sigma_1^4 \sigma_i \sigma_j}{\theta^2 \rt T}\right),
		\end{align*}
		which is uniform in $i,j$.
		Taking a union bound we obtain
		\begin{align*}
			\Pm\left(\max_{i\ne j} (A\tp Q^{-1} A)_{ij}  >\epsilon\right)
			\le \frac{1}{\epsilon}
			\sum_{i\ne j}  \Em|(A\tp Q^{-1} A)_{ij}| = \frac{1}{\epsilon} O\left(\frac{K^4 \sigma_1^4 \sigma_i \sigma_j}{\theta^2 \rt T}\right),
		\end{align*}
		or in other words we have the bound
		\begin{align*}
			\max_{i\ne j}|(A\tp Q^{-1} A)_{ij}|
			=
			O_{p}\left(\frac{K^4 \sigma_1^4 \sigma_i \sigma_j}{\theta^2 \rt T}\right)
		\end{align*}
		Lastly, since $M = I_K - B - A\tp Q^{-1 } A$, it remains to bound the off-diagonals of $B$. Routine computations similar to \eqref{equation - proof of AQA 1} and the union bound above show that $B_{ij}$ is of high order compared to $(A\tp Q^{-1}A)_{ij}$ and is thus negligible. This completes the proof.
	\end{proof}
\end{proposition}
From Proposition~\ref{proposition - M} we can conclude that the CLT of $M_{ii}$ is given by the CLT of $A_{ii}$, up to centering and scaling. This is what we compute next.

\begin{proposition}\label{proposition - CLT for fii}
	Suppose Assumption \ref{assumptions} and either Assumption \ref{assumptions - tau fixed} or \ref{assumptions - tau div} hold.
	For any $i=1,\ldots,K$ and $\tau\ge 0$, define
	\begin{align*}
		v^2_{i, \tau}:= \frac{1}{T- \tau}\Var\left(\sum_{t=1}^{T- \tau} f_{i,t} f_{i,t+ \tau}\right).
	\end{align*}
	\begin{enumerate}
		\item For any $i$ and $\tau$, the quantity $v_{i, \tau}^2$ satisfies
		      \begin{align*}
			      v^2_{i, \tau}
			       & =
			      \sum_{|k|< T- \tau} \left(1 - \frac{|k|}{T - \tau} \right)u_{i,k}(\tau),
			      \numberthis\label{equation - expression for var in brockwell clt}
		      \end{align*}
		      where $(u_{i,k}(\tau))_k$ is given by
		      \begin{align*}
			      u_{i,k}(\tau) := \gamma_i(k)^2 + \gamma_i(k+\tau)\gamma_i(k- \tau)
			      + (\Em[z_{11}^4]-3)
			      \sum_{l=0}^\infty \phi_{i,l} \phi_{i,l+ \tau} \phi_{i,l+k} \phi_{i,l+k+ \tau}.
		      \end{align*}

		\item As $T\to\infty$, the sequence $(v_{i,\tau}^2)$ tends to a limit
		      \begin{align*}
			      \lim_{T\to\infty}
			      v_{i,\tau}^2
			      = (\Em[z_{11}^4]-3) \gamma_i(\tau)^2+
			      \sum_{k\in\Z} \left(  \gamma_i(k)^2 + \gamma_i(k+\tau) \gamma_i(k - \tau) \right)
		      \end{align*}
		      in the case where $\tau$ is a fixed constant, and
		      \begin{align*}
			      \lim_{T\to\infty}
			      v_{i,\tau}^2
			      = \sum_{k\in\Z} \gamma_i(k)^2
		      \end{align*}
		      in the case where $\tau = \tau_T \to\infty$ as $T\to\infty$.
		      \vspace{1mm}
		\item In both cases where $\tau$ is fixed and where $\tau\to\infty$, we have
		      \begin{align*}
			      \frac{1}{\rt T v_{i, \tau}}
			      \left(\sum_{t=1}^{T- \tau} f_{i,t} f_{i,t+ \tau}- \Em\left[\sum_{t=1}^{T- \tau} f_{i,t} f_{i,t+ \tau}\right]\right)
			      \weakly
			      N(0,1), \quad T\to\infty.
		      \end{align*}
	\end{enumerate}

	\begin{proof}
		In the case where $\tau$ is fixed, the proof can be adapted from the arguments in section 7.3 of \cite{BrockwellDavisFienberg1991} so it remains to consider the case where $\tau\to\infty$. For concreteness, the following proof covers both the case where $\tau$ is finite and fixed and where $\tau$ is diverging as $T\to\infty$.
		For brevity of notation we will drop the subscript $i$ (denoting the $i$-th factor) within the proof and write for instance $f_{t}:= f_{it}$ and $\phi_{l}:= \phi_{il}$.

		With some adaptations to the computations in page 226-227 of \cite{BrockwellDavisFienberg1991}, we may obtain
		\begin{align*}
			v_T^2:= \frac{1}{T- \tau}\Var\left(\sum_{t=1}^{T- \tau} f_{t} f_{t+ \tau}\right)
			 & =
			\frac{1}{T- \tau}\Em   \left[\sum_{t=1}^{T - \tau}\sum_{s=1}^{T - \tau} f_t f_{t+ \tau} f_s  f_{s+ \tau}\right]
			-
			(T-	\tau)\gamma(\tau)^2
			\\
			 & =
			\sum_{|k|< T- \tau} \left(1 - \frac{|k|}{T - \tau} \right)u_k(\tau),
			\numberthis\label{equation - expression for var in brockwell clt}
		\end{align*}
		where $(u_k(\tau))_k$ is given by
		\begin{align*}
			u_k(\tau) := \gamma(k)^2 + \gamma(k+\tau)\gamma(k- \tau)
			+ (\Em[z_{11}^4]-3)
			\sum_{l=0}^\infty \phi_l \phi_{l+ \tau} \phi_{l+k} \phi_{l+k+ \tau}.
		\end{align*}
		Note that the sequence $(\phi_l)_l$ is summable and so is the sequence $(u_k(\tau))_k$.
		Taking the limit of \eqref{equation - expression for var in brockwell clt} and invoking the dominated convergence theorem we conclude
		\begin{align*}
			v^2:= \lim_{T\to\infty} v_T^2
			=
			\sum_{k\in\Z} \lim_{T\to\infty} \left(1 - \frac{|k|}{T - \tau} \right)u_k(\tau).
		\end{align*}
		In the case where $\tau$ is a fixed constant, we have, as in Proposition 7.3.1 of \cite{BrockwellDavisFienberg1991},
		\begin{align*}
			v^2
			= (\Em[z_{11}^4]-3) \gamma(\tau)^2+
			\sum_{k\in\Z} \left(  \gamma(k)^2 + \gamma(k+\tau) \gamma(k - \tau) \right),
			\numberthis\label{equation - limit of variance in clt tau finite}
		\end{align*}
		and in the case where $\tau$ is diverging, i.e. $\gamma(\tau)\to0$ as $T\to\infty$, we easily see that
		\begin{align*}
			v^2 = \lim_{T\to\infty}v_T^2
			=
			\sum_{k\in\Z} \gamma(k)^2.
			\numberthis\label{equation - limit of variance in clt tau inf}
		\end{align*}
		This settles the first two claims of the proposition.

		We first prove a version of the CLT for a truncated version of the factor $(f_{t})_{t\ge 0}$.  The truncation will be justified further below.
		Fix $L>0$ and define $(f^{(L)}_t)_{t=1, \ldots, T}$ by
		\begin{align*}
			f_t^{(L)} :=  \sum_{l=0}^L \phi_{l}z_{t-l},
		\end{align*}
		Consider the stochastic process $(f_t^{(L)} f_{t+\tau}^{(L)})_{t=1, \ldots, T- \tau}$. Clearly $(f_t^{(L)} f_{t+\tau}^{(L)})_t$ is an $(L+\tau)$-dependent process, i.e. $f_t^{(L)} f_{t+\tau}^{(L)}$ is independent from $f_s^{(L)} f_{s+\tau}^{(L)}$ whenever $|s-t|> L+ \tau$.
		The mean is given by $\Em[f_t^{(L)} f_{t+\tau}^{(L)}] = \gamma_L(\tau)$, where $\gamma_L(\cdot)$ is the auto-covariance function of the truncated process $(f^{(L)}_t)$.
		Similar to \eqref{equation - expression for var in brockwell clt}-\eqref{equation - limit of variance in clt tau inf} we may compute
		\begin{align*}
			v^2_{T,(L)}:=\frac{1}{T- \tau}\Var\left(
			\sum_{t=1}^{T- \tau}f_t^{(L)}f_{t+\tau}^{(L)}\right),
		\end{align*}
		which has limits, in the case where $\tau$ is fixed:
		\begin{align*}
			v^2_{(L)}:=\lim_{T\to\infty}v^2_{T,(L)}
			= (\Em[z_{11}^4]-3) \gamma_L(\tau)^2+
			\sum_{k\in\Z} \left(  \gamma_L(k)^2 + \gamma_L(k+\tau) \gamma_L(k - \tau) \right)
		\end{align*}
		and in the case where $\tau\to\infty$:
		\begin{align*}
			v^2_{(L)} = \lim_{T\to\infty}v^2_{T,(L)}
			=
			\sum_{k\in\Z} \gamma_L(k)^2.
		\end{align*}
		Note that in either case $V^{(L)}$ is a non-zero constant.
		It can easily be checked that, under Assumption \ref{assumptions} and either Assumption \ref{assumptions - tau fixed} or \ref{assumptions - tau div},  the process $(f_t^{(L)} f_{t+\tau}^{(L)})_{t=1, \ldots, T- \tau}$ (after centering) satisfies the conditions in \cite{Berk1973}, whose main theorem can be applied here to obtain
		\begin{align*}
			\sqrt{T- \tau}
			\left(\frac{1}{T- \tau}\sum_{t=1}^{T- \tau}f_t^{(L)}f_{t+\tau}^{(L)}
			-\gamma_L(\tau)\right) \weakly N(0,v^2_{(L)}), \quad T\to\infty.
		\end{align*}
		We now justify the  truncation. Since $(\phi_l)_l \in \ell_1$ and $(z_t)_t$ is uniformly bounded in $L^4$, it is easy to conclude that, for each fixed $T$, we have
		\begin{align*}
			\Big\|\sum_{t=1}^{T- \tau} f_t^{(L)} f_{t+ \tau}^{(L)}
			-
			\sum_{t=1}^{T- \tau} f_t f_{t+ \tau}
			\Big\|_{L^2}\to0,
			\quad L\to\infty.
			\numberthis\label{equation - L2 truncation}
		\end{align*}
		Consequently, we may conclude that $\gamma_L(\tau)\to \gamma(\tau)$ and  $v^2_{(L)}\to v^2$ as $L\to\infty$, since they are the first and second moments of the sums in \eqref{equation - L2 truncation}. We may then follow the arguments on page 229 of \cite{BrockwellDavisFienberg1991} and apply Proposition 6.3.9 of \cite{BrockwellDavisFienberg1991} to obtain
		\begin{align*}
			\sqrt{T- \tau}
			\left(\frac{1}{T- \tau}\sum_{t=1}^{T- \tau}f_tf_{t+\tau}
			-\gamma(\tau)\right) \weakly N(0,v^2), \quad T\to\infty.
			\numberthis\label{equation - CLT for untruncated ACVF}
		\end{align*}
		Finally, using \eqref{equation - limit of variance in clt tau finite}, \eqref{equation -  limit of variance in clt tau inf} and $\frac{T- \tau}{T}\to1$,
		by Slutsky's theorem we may conclude that
		\begin{align*}
			\frac{1}{\rt{T} v_T}
			\left(\sum_{t=1}^{T- \tau}f_tf_{t+\tau}
			-(T- \tau)\gamma_L(\tau)\right) \weakly N(0,1).
		\end{align*}
		It remains to observe that  $(T- \tau)\gamma_L(\tau)$ is exactly the expectation of  $\sum f_tf_{t+\tau}$ and the last claim of the proposition follows.
	\end{proof}
\end{proposition}

\begin{proposition}\label{proposition - CLT of Mii}
	Under Assumption \ref{assumptions} and either Assumption \ref{assumptions - tau fixed} or \ref{assumptions - tau div}, we have
	\begin{align*}
		\numberthis\label{equation - CLT of Mii}
		\rt T
		\frac{\theta}{2\sigma_i^4 \gamma_i(\tau	) v_{i, \tau	}}
		\big(M_{ii} - \overline M_{ii} \big)
		 = Z_T + O_p\left(\frac{\norm{\sigmab}_{\ell_1}^2}{\sigma_i^2 \gamma_i(\tau)^2 \rt T}+ \frac{1}{ \sigma_i^2 \gamma_i(\tau) } +\frac{\sigma_n^4\gamma_n(\tau)^2 K }{\sigma_i^4\gamma_i(\tau)^2 \rt T}\right)
	\end{align*}
	where $Z_T\weakly N(0,1)$, the centering $\overline M_{ii}$ is as defined in \eqref{equation - decomposition of Mii} and $v_{i, \tau}$ is defined as in \eqref{equation - definition of v i tau}.

	\begin{proof}
		For simplicity, within the current proof we will denote
		$\xx_{i0}:=\xx_{i,[1:T- \tau]}, \xx_{i \tau}:=\xx_{i,[\tau+1:T]}$
		and similarly $
			\ff_{i0}:=\ff_{i,[1:T- \tau]},
			\ff_{i \tau}:=\ff_{i,[\tau+1:T]},
			\epsilonb_{i0} := \epsilonb_{i,[1:T- \tau]}
			,\epsilonb_{i \tau}:=\epsilonb_{i,[\tau+1:T]}$.

		Observe from \eqref{equation - expression for Mii - Miibar} that the asymptotic distribution of $M$ depends crucially on that of $A$.
		We first give the asymptotic distribution of $A_{ii}$.
		Recall from \eqref{equation - first resolvent identity} that
		$R - I_{T- \tau} = \theta^{-1} E_\tau\tp E_\tau E_0\tp E_0 R$, using which we can write
		\begin{align*}
			A_{ii} & =  \frac{1}{\rt \theta T} \xx_{i0}\tp  \xx_{i\tau}
			+
			\frac{1}{\rt \theta T} \xx_{i0}\tp  (R - I_{T- \tau}) \xx_{i\tau}
			\\
			       & =  \frac{1}{\rt \theta T} \xx_{i0}\tp  \xx_{i\tau} +
			\frac{1}{\theta^{3/2} T} \xx_{i0}\tp E_\tau\tp E_\tau E_0\tp E_0 R \xx_{i\tau}.
			\numberthis\label{equation - Aii - EAii first}
		\end{align*}
		Applying Lemma~\ref{lemma - concentration of xBx} to the last term in \eqref{equation - Aii - EAii first} we have
		\begin{align*}
			\frac{1}{\theta^{3/2} T} \xx_{i0}\tp E_\tau\tp E_\tau E_0\tp E_0 R \xx_{i\tau}1_{\cB_0} - \frac{1}{\theta^{3/2} T}\underline\Em[ \xx_{i0}\tp E_\tau\tp E_\tau E_0\tp E_0 R \xx_{i\tau}1_{\cB_0}]
			=O_{L^2}\left(  \frac{\sigma_i^2}{\theta^{3/2} \rt T} \right).
		\end{align*}
		which gives
		\begin{align*}
			A_{ii}1_{\cB_0} - \underline \Em[A_{ii}1_{\cB_0}]
			\numberthis\label{equation - Aii - EAii into more terms in proof of clt}
			 & = \frac{1}{\rt \theta T} \left(\xx_{i0}\tp  \xx_{i\tau}
			- \Em [ \xx_{i0}\tp  \xx_{i\tau}]\right)1_{\cB_0}
			+ O_{L^2}\left(  \frac{\sigma_i^2}{\theta^{3/2} \rt T} \right).
		\end{align*}
		Next, we recall that $\xx_{i0} = \frac{1}{\rt {\theta T}} (\sigma_i\ff_{i0} + \epsilonb_{i0})$ and $\xx_{i\tau} = \frac{1}{\rt {\theta T}} (\sigma_i\ff_{i\tau} + \epsilonb_{i\tau})$  so that
		\begin{align*}
			\xx_{i0}\tp  \xx_{i\tau} = \sigma_i^2\ff_{i0}\tp \ff_{i\tau} + (\sigma_i\ff_{i0}\tp\epsilonb_{i\tau} + \sigma_i\epsilonb_{i0}\tp\ff_{i\tau} + \epsilonb_{i0}\tp \epsilonb_{i\tau}).
		\end{align*}
		Applying Lemma~\ref{lemma - bai silverstein concentration-} to the three terms in parenthesis on the right hand we get
		\begin{align*}
			\xx_{i0}\tp  \xx_{i\tau} & - \Em[\xx_{i0}\tp  \xx_{i\tau}]
			=  \sigma_i^2\ff_{i0}\tp \ff_{i\tau} - \sigma^2_i\Em[\ff_{i0}\tp \ff_{i\tau}]
			+ O_{L^2}( \sigma_i  \rt T).
			\numberthis\label{equation - xx is close to ff}
		\end{align*}
		Substituting back into \eqref{equation - Aii - EAii into more terms in proof of clt}  we obtain
		\begin{align*}
			A_{ii} 1_{\cB_0}- \underline \Em[A_{ii}1_{\cB_0}]
			 & =
			\frac{\sigma_i^2}{\rt \theta T} \left(\ff_{i0}\tp \ff_{i\tau}
			- \Em[\ff_{i0}\tp \ff_{i\tau}] \right)1_{\cB_0}
			+ O_{L^2}\left(  \frac{\sigma_i}{ {\theta^{1/2}\rt T} }   \right)
			+ O_{L^2}\left(  \frac{\sigma_i^2}{\theta^{3/2} \rt T} \right)
			\\
			 & =
			\numberthis\label{equation - Aii - xx only differnt in mean}
			\frac{\sigma_i^2}{\rt \theta T} \left(\ff_{i0}\tp \ff_{i\tau}
			- \Em[\ff_{i0}\tp \ff_{i\tau}] \right)1_{\cB_0}
			+ 		O_{L^2}\left(  \frac{\sigma_i}{ {\theta^{1/2}\rt T} }   \right).
		\end{align*}
		Rescaling and recalling that $1_{\cB_0} = 1-o_p(T^{-l})$ for any $l\ge 1$, we have
		\begin{align*}
			\rt T\frac{\rt \theta}{\sigma_i^2}\left(A_{ii} - \underline\Em[A_{ii}] \right)
			=
			\frac{1}{\rt T} \left(\ff_{i0}\tp \ff_{i\tau}
			- \Em[\ff_{i0}\tp \ff_{i\tau}] \right) + O_{p}(\sigma_i^{-1}).
		\end{align*}
		From this we observe that we can obtain a CLT for $A_{ii}$ from a CLT for the auto-covariance function of $\ff_i$.
		Indeed, by Proposition~\ref{proposition - CLT for fii} we have
		\begin{align*}
			\numberthis\label{equation - clt of Aii}
			\rt T\frac{\rt \theta}{\sigma_i^2 v_{i, \tau}}\left(A_{ii} - \underline\Em[A_{ii}] \right)
			\weakly
			N(     0,  1  ),
			\quad p,T\to\infty,
		\end{align*}
		where $v_{i, \tau}$ is specified in the statement of Proposition~\ref{proposition - CLT for fii}.

		Finally, we recall from Proposition~\ref{proposition - M} that
		\begin{align*}
			{M_{ii}  - \overline M_{ii}}
			 & =
			- 2\big(A_{ii} - \underline\Em [A_{ii}]\big)
			{\Em[A_{ii}1_{\cB_0}]}{\Em [Q_{ii}^{-1}1_{\cB_2}]}
			+ 
			O_p\left(\frac{\sigma_i^2 \norm{\sigmab}_{\ell_1}^2}{\theta T}+ \frac{\sigma_i^2}{\theta \rt T} +KT^{-1}\right)
			\numberthis\label{equation - M- Mii, show coef has lower bound}.
		\end{align*}
		To apply the CLT in \eqref{equation - clt of Aii} to \eqref{equation - M- Mii, show coef has lower bound}, we need to divide \eqref{equation - M- Mii, show coef has lower bound}
		by the coefficient of $A_{ii} - \underline\Em[A_{ii}]$.
		We recall from \eqref{equation - norm of Q 1cB} that $Q_{ii}^{-1}1_{\cB_2} = 1+ o(1)$. Furthermore, from Lemma~\ref{lemma - expectation of ABQ}  we have
		\begin{align*}
			\numberthis\label{equation - order of Aii}
			\Em[A_{ii}1_{\cB_0}]
			 & =
			\frac{\sigma_i^2 \gamma_i(\tau)}{\rt \theta } + o(1).
		\end{align*}
		Therefore from \eqref{equation - M- Mii, show coef has lower bound} we get
		\begin{align*}
			\rt T\frac{\theta}{- 2\sigma_i^4 \gamma_i(\tau) v_{i,\tau}}
			&\big(M_{ii}  - \overline M_{ii}\big)
			=
			\rt T \frac{\rt \theta}{\sigma_i^2 v_{i,\tau}}\big(A_{ii} - \underline\Em [A_{ii}]\big)
			(1+o(1))
			\\
			&+ O_p\left(\frac{\norm{\sigmab}_{\ell_1}^2}{\sigma_i^2 \gamma_i(\tau)^2 \rt T}+ \frac{1}{ \sigma_i^2 \gamma_i(\tau) } +\frac{\sigma_n^4\gamma_n(\tau)^2 K }{\sigma_i^4\gamma_i(\tau)^2 \rt T}\right),
		\end{align*}
		and the claim follows then from the CLT in
		\eqref{equation - clt of Aii}.
	\end{proof}
\end{proposition}

The asymptotic distribution of $M_{ii}$ proved in Proposition~\ref{proposition - CLT of Mii} is the last piece of ingredient we need to prove the main result of the paper, which we present below.

\begin{proof}[Proof of Theorem~\ref{theorem - CLT}]
	Recall from Section~\ref{section - main clt} that up to now we have dealt with, without loss of generality, the $n$-th largest eigenvalue $\lambda:=\lambda_n$ and the corresponding $\theta:= \theta_n$ and $\delta:= \delta_n:= \lambda_n/\theta_n-1$. Recall from Proposition~\ref{proposition - detequation } that $\delta_n$ satisfies
	\begin{align*}
		\numberthis\label{equation - deteq in the proof of clt}
		\det\left( M + \frac{\delta_n}{\theta_n} X_\tau X_0\tp X_0 X_\tau\tp + \delta o_{p, \norm{\cdot}}(1)\right)=0
	\end{align*}
	We first consider the asymptotic properties of the elements of the matrix $M$.
	From Proposition~\ref{proposition - solution for theta} and the definition of $\overline M_{ii}$ in \eqref{equation - decomposition of Mii}, clearly we see
	$\overline M_{nn} =0$.
	From Theorem~\ref{theorem - 2.1} and Proposition~\ref{proposition - solution for theta} we also recall that  ${\theta_n}/({\sigma_n^4 \gamma_n(\tau)^2})=\theta_n / \mu_{n, \tau}^2=1+o(1)$.
	By Proposition~\ref{proposition - CLT of Mii} we have
	\begin{align*}
		\rt T  \frac{\gamma_n(\tau)}{2 v_{n, \tau}}  M_{nn}
		 = N(0,1) 
		 + O_p\left(\frac{\norm{\sigmab}_{\ell_1}^2}{\sigma_n^2 \gamma_n(\tau)^2 \rt T}+ \frac{1}{ \sigma_n^2 \gamma_n(\tau) } +\frac{ K }{\rt T}\right),
	\end{align*}
	where the last term is $o_p(1)$ by the Assumptions.

	For $i\ne n$, we recall from \eqref{equation - decomposition of Mii} and Lemma~\ref{lemma - expectation of ABQ} that
	\begin{align*}
		\overline M_{ii} = 1 - \frac{\sigma_i^4 \gamma_i (\tau)^2}{\theta_n} + o(1)
		=
		1 - \frac{\mu_{i,\tau}^2}{\mu_{n, \tau}^2} + o(1) \gtrsim 1,
	\end{align*}
	where the last inequality is due to Assumption \ref{assumptions}.
	Using Proposition~\ref{proposition - M} we have
	\begin{align*}
		M_{ii}\gtrsim 1 + o_p(1), \quad \forall i\ne n,
		\quad \max_{i\ne j}|M_{ij}| = O_p\left(\frac{K^4 \sigma_1^6}{\theta_n^2 \rt T}\right).
	\end{align*}
	Next, recall $\delta = o_p(1)$ from Theorem~\ref{theorem - 2.1}.
	Recall that
	\begin{align*}
		(X_\tau  X_0\tp X_0 X_\tau\tp)_{ij}
		& =  O_{L^1}\left(\frac{\sigma_i \sigma_j(\sigma_i^2 \gamma_i(\tau)+ \sigma_j^2 \gamma_j(\tau))}{\rt T}\right),
		\\
		(X_\tau X_0\tp  X_0 X_\tau\tp)_{ii}& =   \mu_{i, \tau} + O_{L_1} \left(\frac{\sigma_i^4 \gamma_i(\tau)}{\rt T}
		+ \frac{\sigma_i^4 + \sigma_i^2 \norm{\sigmab}_{\ell_2}^2}{T}
		\right).
	\end{align*}
	This in particular implies
	\begin{align*}
		\frac{\delta_n}{\theta_n}(X_\tau X_0\tp X_0 X_\tau\tp)_{ii} =
		\delta_n \frac{\sigma_i^4 \gamma_i(\tau)^2}{\theta_n}
		+o_p(\delta_n), 
		\quad
		\frac{\delta_n}{\theta_n}(X_\tau X_0\tp X_0 X_\tau\tp)_{ij} =
		o_p(\delta),\quad \forall i\ne j.
	\end{align*}
	Combining the above, equation \eqref{equation - deteq in the proof of clt}
	becomes $\det(Q)=0$, where $Q$ is a matrix satisfying
	\begin{align*}
		Q_{nn} = M_{nn} + \delta_n \frac{\sigma_n^4 \gamma_n(\tau)^2}{\theta_n} + \delta_n o_p(1),
	\end{align*}
	for its $n$-th diagonal element, $Q_{ii} \gtrsim 1, \quad \forall i\ne n$ and
	\begin{align*}
		\sup_{ij} Q_{ij}  = O_p\left(\frac{K^4 \sigma_1^6}{\theta_n^2 \rt T}\right)
		+ \delta o_p(1).
		\numberthis\label{equation - estimates on Q in final proof}
	\end{align*}
	Using Leibniz's formula to compute  $\det(Q)$, we have
	\begin{align*}
		\numberthis\label{equation - Leibniz formula determinant}
		0 = \det (Q) = \sum_{\pi\in  S_K} \mathrm{sgn}(\pi) \prod_{i=1}^K Q_{i, \pi(i)},
	\end{align*}
	where $\mathrm{sgn}(\pi)$ is the sign of a permutation $\pi$ in the symmetry group $S_K$.
	Next we show that $\prod_i Q_{ii}$ is the leading term in the sum in \eqref{equation - Leibniz formula determinant}.
	Write $S_{K,k}$ for the subgroup of permutations that has exactly $K-k$ fixed points, i.e.
	\begin{align*}
		S_{K,k} = \{\pi\in S_K, i=\pi(i) \text{ for exactly $K-k$ such $i$'s}\}.
	\end{align*}
	Using this notation we can rewrite \eqref{equation - Leibniz formula determinant} into
	\begin{align*}
		0 = \det(Q) & =  \sum_{k=0}^K \sum_{\pi\in S_{K,k}}\mathrm{sgn}(\pi) \prod_{i=1}^K Q_{i, \pi(i)}.\numberthis\label{equation - Leibniz formula determinant 1.5}
	\end{align*}
	We recall that the order of $S_{K,k}$ is given by the rencontres numbers (see \cite{Riordan2012})
	\begin{align*}
		|S_{K,k}| = D_{K,K-k}:= \frac{K!}{(K-k)!}\sum_{i=0}^{k}\frac{(-1)^i}{i!}.
	\end{align*}
	Observe that $|S_{K,0}|=1$ since $S_{K,0}$ contains only the identity permutation  and $|S_{K,1}|=0$ since for any non-identity permutation $\pi$, there exists at least two indices $i,j\in \{1,\ldots,K\}$, $i\ne j$ such that $i\ne \pi(i)$ and $j\ne \pi(j)$.
	Therefore \eqref{equation - Leibniz formula determinant 1.5} becomes
	\begin{align*}
		0 = \det(Q) & =
		\prod_{i=1}^K Q_{ii} +\sum_{k=2}^K
		\sum_{\pi\in S_{K,k}}\mathrm{sgn}(\pi) \prod_{i=1}^K Q_{i, \pi(i)}.
		\numberthis\label{equation - Leibniz formula determinant 2}
	\end{align*}
	Note that for any $k\ge 2$ and any permutation $\pi\in S_{K,k}$, the product $\prod_{i=1}^K Q_{i, \pi(i)}$ contains exactly $k$ off-diagonal elements of $Q$. By \eqref{equation - estimates on Q in final proof} we have the estimate
	\begin{align*}
		\prod_{i=1}^K Q_{i, \pi(i)} = \left(O_p\left(\frac{K^4 \sigma_1^6}{\theta_n^2 \rt T}\right)
		+ \delta o_p(1)\right)^k.
	\end{align*}
	Finally, after substituting back into \eqref{equation - Leibniz formula determinant 2} and a rather tedious computation we have
	\begin{align*}
		0 = \det (Q)
		 & = \prod_{i=1}^K Q_{ii} + \delta o_p(1) +  o_p(T^{-1/2}),
		\numberthis\label{equation - diagonal is dominant in leibniz formula}
	\end{align*}
	which shows that the product $\prod_{i=1}^K Q_{ii}$ is the leading term of $\det (Q)$.

	Next,
	using \eqref{equation - estimates on Q in final proof} again we  see that $\prod_{i=1}^K Q_{ii}$ can be written into
	\begin{align*}
		\prod_{i=1}^K Q_{ii} & =\left(M_{nn} + \delta_n \frac{\sigma_n^4 \gamma_n(\tau)^2}{\theta_n} + \delta_n o_p(1)\right)\left(1+o_p(1)\right)
		=
		(M_{nn} + \delta_n) \left(1+o_p(1)\right),
	\end{align*}
	which can be
	substituted back into \eqref{equation - diagonal is dominant in leibniz formula} to obtain
	\begin{align*}
		0=M_{nn} \left(1+o_p(1)\right) + \delta_n \left(1+o_p(1)\right) + \delta_n o_p(1)+ o_p(T^{-1/2}).
	\end{align*}
	Rearranging (and recalling $\theta_n\asymp \sigma_{n}^4 \gamma_n(\tau)^2$ by Proposition~\ref{proposition - solution for theta}), we finally get
	\begin{align*}
		-\rt T\frac{\gamma_n(\tau)}{2  v_{n, \tau}}\delta_n \left(1+o_p(1)\right)
		=\rt T\frac{\theta_n}{2\sigma_n^4 \gamma_n(\tau) v_{n, \tau}}M_{nn}\left(1+o_p(1)\right)
		+ o_p(1).
	\end{align*}
	Applying Proposition~\ref{proposition - CLT of Mii} we immediately have
	\begin{align*}
		\rt T\frac{\gamma_n(\tau)}{2  v_{n, \tau}}\delta_n \weakly N(0,1),\quad T\to\infty
	\end{align*}
	and the proof is complete.
\end{proof} 
\section{Estimates on bilinear forms}\label{section - clt lemmas}
The following linear algebraic result will be useful throughout.
\begin{lemma}[Sherman-Morrison formula]\label{Lemma - shermanMorrison}
		Suppose $A$ and $B$ are invertible matrices of the same dimension, such that $A-B$ is of rank one. Then
		\begin{align*}
			\numberthis\label{equation - sherman 1}
			A^{-1} - B^{-1}  = - \frac{B^{-1} (A-B) B^{-1}}{1+ \tr(B^{-1}(A-B))}.
		\end{align*}
		Further more, if $A - B = \uu \vvv\tp$, then
		\begin{align*}
			A^{-1} \uu  = \frac{B^{-1} \uu}{ 1+ \vvv\tp B^{-1}\uu},
			\quad 
			\vvv\tp A^{-1}  = \frac{\vvv\tp B^{-1}}{ 1+ \vvv\tp B^{-1}\uu}.
			\numberthis\label{equation - sherman 2}
		\end{align*}
\end{lemma}

We first establish some concentration inequalities for quadratic forms of the random vector $\xx$. To do so we will need to introduce some notations.  We recall from \eqref{equation - fit = timeseries} and \eqref{equation - def of xit} that
\begin{align*}
	x_{it} = \sigma_if_{it} +\epsilon_{it}=  \sigma_i  \sum_{l=0}^{\infty} \phi_{il} z_{i, t-l} + \epsilon_{it}, \quad i = 1,\ldots, K, \quad t = 1, \ldots, T.
\end{align*}
We truncate the series and define an approximation
\begin{subequations}
\begin{align*}
	\numberthis\label{equation - definition of xit truncated}
	x_{it}^{(L)} :=  \sigma_if_{it}^{(L)} + \epsilon_{it}:= \sigma_i  \sum_{l=0}^{L} \phi_{il} z_{i, t-l} + \epsilon_{it}, \quad L\ge 1,
\end{align*}
and write $\xx_{i,[1:T]}^{(L)}$, $\ff_{i,[1:T]}^{(L)}$ for $(x_{it}^{(L)})_ {t=1,\ldots, T}$ and $(f_{it}^{(L)})_ {t=1,\ldots, T}$. 
For each $L$,  We write 
\begin{align*}
	\numberthis\label{equation - definition of phi truncated}
	\underline\phib_{i}\tp := \Big(\phi_{iL}, \ldots, \phi_{i0}, \0_{T-1}\tp\Big)  \in\Rm^{T+L}.
\end{align*}
Let $S$ be the right-shift operator on $\Rm^{T+L}$, i.e. $S \ee_i = \ee_{i+1}$. Define
\begin{align*}
	\numberthis\label{equation - definition of Phi matrix truncated}
	\Phi_i := \Big( \underline\phib_{i}, S  \underline\phib_{i}, \ldots, S^{T-1} \underline\phib_{i}\Big) \in \Rm^{(T+L)\times T},
\end{align*}
then clearly we can write the approximation $\ff_i^{(L)}$ as
\begin{align*}
	\numberthis\label{equation - definition of xxL = zz PHIL}
	\ff_{i,[1:T]}^{(L)} =   \zz_{i,[1-L, T]}\tp    \Phi_i. 
\end{align*}
\end{subequations}

We note that the spaces of $n\times n$ matrices equipped with the Frobenius norm is isometrically isomorphic to $\Rm^{n\times n}$ with the Euclidean norm. 
For each $1\le i,j\le K$, we define linear operators $\Psi_n^{ij}$, $n=0,1,2$,
\begin{align*}
	\Psi_n^{ij} : \Rm^{(T- \tau)\times (T- \tau)}\to \Rm^{(2T+L)\times (2T+L)}
\end{align*}
by sending a $(T- \tau)\times (T- \tau)$ matrix $B$ to the $(2T+L)\times (2T+L)$ matrices
\begin{align*}
	&\Psi_0^{ij} B := \begin{pmatrix}
			\sigma_i \Phi_i \\ I_{T}
		\end{pmatrix}
		\begin{pmatrix}
			I_{T - \tau}
			\\
			\0_{\tau\times (T- \tau)}
		\end{pmatrix}
		B
		\begin{pmatrix}
			I_{T - \tau}
			\\
			\0_{\tau\times (T- \tau)}
		\end{pmatrix}\tp
		\begin{pmatrix}
			\sigma_j \Phi_j \\ I_{T}
		\end{pmatrix}\tp,
	\\
	&\Psi^{ij}_1 B := \begin{pmatrix}\numberthis\label{equation - definition of Psi}
			\sigma_i \Phi_i \\ I_{T}
		\end{pmatrix}
		\begin{pmatrix}
			I_{T - \tau}
			\\
			\0_{\tau\times (T- \tau)}
		\end{pmatrix}
		B
		\begin{pmatrix}
			\0_{\tau\times (T- \tau)}
			\\
			I_{T - \tau}
		\end{pmatrix}\tp
		\begin{pmatrix}
			\sigma_j \Phi_j \\ I_{T}
		\end{pmatrix}\tp,
		\\
	&\Psi^{ij}_2 B := \begin{pmatrix}
			\sigma_i \Phi_i \\ I_{T}
		\end{pmatrix}
		\begin{pmatrix}
			\0_{\tau\times (T- \tau)}
			\\
			I_{T - \tau}
		\end{pmatrix}
		B
		\begin{pmatrix}
			\0_{\tau\times (T- \tau)}
			\\
			I_{T - \tau}
		\end{pmatrix}\tp
		\begin{pmatrix}
			\sigma_j \Phi_j \\ I_{T}
		\end{pmatrix}\tp,{}
\end{align*}
where $\Phi_i := (\underline\phib_{i}, S \underline\phib_{i} \ldots, S^{T-1}\underline\phib_{i}) \in\Rm^{(T+L)\times T}$ is as defined in \eqref{equation - definition of Phi matrix truncated}. We first give some estimates on the operators $\Psi^{ij}_n$. 

\begin{lemma}\label{lemma - properties of Psi}
	The following estimates hold uniformly in $L\in \N$. 
	\begin{enumerate}
		\item \label{lemma - properties of Psi - norm of PHIPHI} 
		The matrix  $\Phi_i\tp \Phi_i$ is symmetric and (banded) Toeplitz with
		\begin{align*}
			\sup_{i} \norm{\Phi_i\tp \Phi_i} \le 1+ \sup_i \norm{\phib_{i}}_{\ell_1}^2 = O(1).
		\end{align*}

		\item \label{lemma - properties of Psi - norm of Psi^ij} 
		For $n=0,1,2$, the operator norms  of $\Psi_n^{ij}$be bounded by 
		\begin{align*}
			\norm{\Psi_n^{ij}}^2\le \Big(1+ \sigma_i^2 \norm{\phib_i}_{\ell_1}^2\Big)\Big(1+ \sigma_j^2 \norm{\phib_j}_{\ell_1}^2\Big) = O(\sigma_i^2 \sigma_j^2).
		\end{align*}
		\item \label{lemma - properties of Psi - trace of Psiij} 
		For any $B \in \Rm^{T\times T}$, the trace of $\Psi_n^{ii} B$ can be bounded by
		\begin{align*}
			\big|\tr(\Psi_n^{ii}B)\big|\le (T- \tau)(1+ \sigma_i^2 \norm{\phib_i}_{\ell_1}^2) \norm{B} = O(\sigma_i^2 (T- \tau) \norm{B}) .
		\end{align*}
	\end{enumerate}

	\begin{proof}\pf{\ref{lemma - properties of Psi - norm of PHIPHI}}
		From the definitions \eqref{equation - definition of phi truncated} and \eqref{equation - definition of Phi matrix truncated} we immediately have
		\begin{align*}
			(\Phi_i\tp \Phi_i)_{s,t} = 1_{|s-t|\le L} \underline\phib_{i}\tp S^{|s-t|} \underline\phib_{i}
			=
			1_{|s-t|=k\le L}\sum_{l=0}^{L-k} \phi_{i,l+k} \phi_{i,l}.
		\end{align*}
		It is clear that $\Phi_i\tp \Phi_i$ is a banded, symmetric Toeplitz matrix. The operator norm of $\Phi_i\tp \Phi_i$ is controlled by the supremum of its symbol over $\C$ (see \cite{BottcherSilbermann1999}) and we have
			\begin{align*}
				\norm{\Phi_i\tp \Phi_i} &\le  \sup_{\lambda\in \C} 
				\left|\sum_{|k|=0}^{L}  \underline\phib_{i}\tp S^{|k|} \underline\phib_{i} e^{\rt{-1}k \lambda}\right|
				\le \norm{\underline\phib_{i}}_{\ell_2}^2
				+ \sum_{k=1}^{L}  \sum_{l=0}^{L-k} |\phi_{i,l+k} \phi_{i,l}|
				\le 1
				+ \norm{\phib_{i}}_{\ell_1}^2,
			\end{align*}
			which is bounded 
			uniformly in $i=1,	\ldots, K$, due to Assumption \eqref{assumptions}. 
		%

		\pfspace
		\pf{\ref{lemma - properties of Psi - norm of Psi^ij}}
		By the cyclic property of the trace  and Cauchy-Schwarz inequality we get
		\begin{align*}
			&\norm{\Psi^{ij}_1 B }_F^2  =  \tr \left(   (\Psi^{ij}_1 B)  (\Psi^{ij}_1 B )\tp    \right)
			   \\
			& \quad = \tr \left(
				(I_{T} + \sigma_i^2 \Phi_i\tp  \Phi_i )
					(I_{T- \tau}, \0)\tp
					B
					(\0, I_{T- \tau})
				(I_{T} + \sigma_j^2 \Phi_j\tp  \Phi_j )
					(\0, I_{T- \tau})\tp
					B\tp(I_{T- \tau}, \0)\right).
			   \\
			& \quad \le 
				\normm{(I_{T} + \sigma_i^2 \Phi_i\tp  \Phi_i )
					(I_{T- \tau}, \0)\tp
					B
					(I_{T- \tau}, \0)}_{\rm F}
				\normm{
				(I_{T} + \sigma_j^2 \Phi_j\tp  \Phi_j )
				(\0, I_{T- \tau})\tp
					B\tp(I_{T- \tau}, \0)}_{\rm F}.
		\end{align*}
		Since $\norm{AB}_F \le \norm{A} \norm{B}_F$, we have
		\begin{align*}
			\norm{\Psi^{ij}_1 B }^2_{\rm F} 
			\le \norm{I_{T} + \sigma_i^2 \Phi_i\tp \Phi_i }
			\norm{I_{T} + \sigma_j^2 \Phi_j\tp \Phi_j }
			\norm{ B}_F^2,
		\end{align*}
		where $\norm{I_{T} +\sigma_i^2 \Phi_i\tp \Phi_i }
			\le
			1+ \sigma_i^2 \norm{\phib_i}_{\ell_1}^2$ by the first claim of the Lemma. By identifying $\Psi_{1}^{ij}$ as an operator between spaces of matrices equipped with the Frobenius norm, this translates to a bound on its spectral norm. The case of $\Psi_0$ and $\Psi_2$ hold analogously.

			\pfspace
		\pf{\ref{lemma - properties of Psi - trace of Psiij} }
		For the last bound, similar computations give
		\begin{align*}
			|\tr ( \Psi_0^{ii} B )| & =  
			\left|\tr \left(
					(I_{T} +\sigma_i^2 \Phi_i\tp \Phi_i )
					(I_{T- \tau}, \0)\tp
					B
					(I_{T- \tau}, \0)
				\right)\right|
			\\
			&\le 
			\norm{I_{T} +\sigma_i^2 \Phi_i\tp \Phi_i} \norm{B}_F
			\le (T- \tau)(1+ \sigma_i^2 \norm{\phib_i}^2_{\ell_1}) \norm{B}
		\end{align*}
		The rest of the claims hold similarly.
	\end{proof}
\end{lemma}

Next, we state an easy extension to Lemma 2.7 of \cite{BaiSilverstein1998} suited to our needs. 
\begin{lemma}\label{lemma - bai silverstein concentration-}
	Let $\zz = (\zz_1\tp, \zz_2\tp)\tp$, where $\zz_1=(z_{1}, \ldots, z_{m})$ and $\zz_2 = (\tilde z_{1}, \ldots, \tilde  z_{n})$ are independent random vectors each with i.i.d. entries satisfying $\Em[z_{1}] = \Em[\tilde z_{1}] =0$, $\Em[z_{1}^2] = \Em[\tilde z_{1}^2] =1$, $\nu_{q}:= \Em|z_{1}|^q <\infty$ and  $\tilde \nu_{q}:= \Em|\tilde z_{1}|^q <\infty$
	for some $q\in [1,\infty)$.
	\begin{enumerate}
		\item \label{lemma - bai silverstein concentration- independent}
		Let $C$ be a deterministic $m\times n$ matrix, then
		\begin{align*}
			\zz_1\tp C\zz_2 = O_{L^q} (\norm{C}_F),
		\end{align*}
		where the constant in the estimate depends only on $q$ and $\nu_q, \td \nu_q$.
		\pfspace

		\item \label{lemma - bai silverstein concentration- main term}
		Let $M$ be a deterministic $(m+n)\times (m+n)$ matrix, then
		\begin{align*}
			\zz\tp M\zz - \tr M = O_{L^q} (\norm{M}_F),
		\end{align*}
		where the constant in the estimate depends only on $q$ and $\nu_k, \td \nu_k$ for $k\le 2q$. 
	\end{enumerate}

	\begin{proof}\pf{\ref{lemma - bai silverstein concentration- independent}}
		By Lemma 2.2 and Lemma 2.3 of \cite{BaiSilverstein1998} we have
		\begin{align*}
			\Em\big| \zz_1\tp C\zz_2  \big|^q
			& = 
			\Em\big| \sum_{i,j} z_{i}\td z_{j} C_{ij}  \big|^q
			\lesssim
			\Em\big| \sum_{i,j} z_{i}^2\td z_{j}^2 C_{ij}^2  \big|^{q/2}
			   \\
			& \lesssim
			\big(\sum_{i,j} \Em [z_{i}^2\td z_{j}^2 C_{ij}^2] \big)^{q/2}
			+ 
			\sum_{i,j} \Em [|z_{i}|^q|\td z_{j}|^q |C_{ij}|^q] 
			   \\
			& = \big(\sum_{i,j} M_{ij}^2\big)^{q/2} + \nu_{q}\td \nu_q \sum_{i,j} |C_{ij}|^q \le (1+ \nu_q \td \nu_q) \norm{C}_F^q,
		\end{align*}
		where the last inequality holds since $\sum |C_{ij}|^q \le (\sum |C_{ij}|^2)^{q/2}$ for $q\ge 2$. 

		\pf{\ref{lemma - bai silverstein concentration- main term}}
		Write $M = \begin{pmatrix}
			A& B \\C &D
		\end{pmatrix}$ 
		where $A,B,C,D$ are of dimensions such  that
		\begin{align*}
			\zz\tp
			M
			\zz
			=
			\zz_1\tp A \zz_1 + \zz_1\tp B \zz_2 + \zz_2\tp C\zz_1+ \zz_2\tp D \zz_2.
		\end{align*}
		By Lemma 2.7 of \citet{BaiSilverstein1998} we have
		\begin{align*}
			&\Em\big| \zz_1\tp A\zz_1  - \tr A \big|^q
			\lesssim
			(\nu_{4}^{q/2} + \nu_{2q}) \tr(AA\tp)^{q/2} 
			\le 
			(\nu_{4}^{q/2} + \nu_{2q}) \norm{M}_F^q,
			\\
			&\Em\big| \zz_2\tp D\zz_2  - \tr D\big|^q
			\lesssim
			(\tilde \nu_{4}^{q/2} + \tilde \nu_{2q}) \tr(DD\tp)^{q/2}
			\le(\nu_{4}^{q/2} + \nu_{2q}) \norm{M}_F^q.
		\end{align*}
		Then we can write
		\begin{align*}
			\Em|\zz\tp M \zz - \tr M|^{q} &\lesssim
			\Em\big| \zz_1\tp A\zz_1  - \tr A \big|^q
			+
			\Em\big| \zz_2\tp B\zz_2  - \tr B \big|^q
			\\&\quad+\Em\big| \zz_1\tp C\zz_2  \big|^q
			+\Em\big| \zz_2\tp D\zz_1  \big|^q.
		\end{align*}
		and the claim follows from (\ref{lemma - bai silverstein concentration- independent}) of the lemma.
	\end{proof}
\end{lemma}

Using Lemma ~\ref{lemma - properties of Psi} and Lemma ~\ref{lemma - bai silverstein concentration-} we can derive the following concentration inequalities for quadratic forms involving certain high probability events.

\begin{lemma}\label{lemma - concentration of xBx}
	Let $\xx,\ff,\epsilonb$ be defined as in \eqref{equation - definition of xx} and $q\le 2$. Under Assumptions \ref{assumptions} and either Assumptions \ref{assumptions - tau fixed} or \ref{assumptions - tau div} we have
	\begin{enumerate}
		\item For any (deterministic) square matrix $B$ of size $T- \tau$, we have
		\label{lemma- concentration of xBx (XBX) }
		\begin{align*}
			&\xx_{i,[1: T - \tau]}\tp B \xx_{j, [\tau+1:T]} - \Em[\xx_{i,[1: T - \tau]}\tp B \xx_{j, [\tau+1:T]}]
			=
			O_{L^q}
			\left({\sigma_i \sigma_j }{\rt {T}} \norm{B} \right),
		\end{align*}
		where the expectation is satisfies
		\begin{align*}
			&\Em[\xx_{i,[1: T - \tau]}\tp B \xx_{j, [\tau+1:T]}] 
			= 1_{i=j}\tr \big(\Psi_1^{ii}(B) \big)
			= 1_{i=j}O(\sigma_i^2 T\norm{B}) .
		\end{align*}

		\item \label{lemma- concentration of xBx (FBE) }
		For all $i,j$ we have 
		$\Em[\ff_{i, [1:T- \tau]}\tp B \epsilonb_{j,[\tau+1: T]}] =0$ 
		and
		\begin{align*}
			&\ff_{i, [1:T- \tau]}\tp B \epsilonb_{j,[\tau+1: T]}= O_{L^{2q}}(\sigma_i \rt T \norm{B} ).
		\end{align*}
		
		\item \label{lemma- concentration of xBx (XRX) }
		Suppose $n\in \{1, \ldots, K\}$ and $c_1,c_2$ are positive constants with $c_1<c_2$. Pick any
		\begin{align*}
			a \in \left[c_1   , c_2  \right]\mu_{n,\tau}^2.
		\end{align*} 
		Recall from \eqref{equation - definition of R} the resolvent $R(a) := (I_{T- \tau} - a^{-1} E_\tau\tp E_\tau E\tp E)^{-1}$, then
		\begin{align*}
			\xx_{i, [1, T- \tau]}\tp R(a)^k  \xx_{j, [\tau+1 : T]}1_{\cB_0}
			- \underline\Em[\xx_{i, [1, T- \tau]}\tp R(a)^k \xx_{j, [\tau+1 : T]}1_{\cB_0}]=
			O_{L^q}(\sigma_i \sigma_j \rt T),
		\end{align*}
		for all $k\in\N$, 
		where $\underline\Em[\ \cdot \ ] := \Em[\ \cdot\ |\cF_p]$ is defined in \eqref{equation - Emcond}. In particular, 
		\begin{align*}
			\xx_{i, [1, T- \tau]}\tp R(a)^k  \xx_{j, [\tau+1 : T]}
			- \underline\Em[\xx_{i, [1, T- \tau]}\tp R(a)^k \xx_{j, [\tau+1 : T]}1_{\cB_0}]=
			O_{p}(\sigma_i \sigma_j \rt T).
		\end{align*}

		\item \label{lemma- concentration of xBx part d }
		Parts	(\ref{lemma- concentration of xBx (XBX) })-(\ref{lemma- concentration of xBx (XRX) }) of the lemma remain true if the vector  $\xx_{j,[\tau+1:T]}$ is replaced by $\xx_{j,[1:T- \tau]}$ and the operator $\Psi^{ii}_1$ is replaced by $\Psi^{ii}_0$. 
	\end{enumerate}
	
	\begin{proof}
		\pf{\ref{lemma- concentration of xBx (XBX) }}
		We apply the truncation procedure as described in \eqref{equation - definition of xit truncated}. 
		Recalling  \eqref{equation - definition of xit truncated}, \eqref{equation - definition of xxL = zz PHIL} and \eqref{equation - definition of Psi} we may write
		\begin{align*}
			\xx_{i,[1:T- \tau]}^{(L)\top} B \xx_{j, [\tau+1 :T]}^{(L)}
			& =  
			(\sigma_i\ff_{i, [1:T]}^{(L)} + \epsilonb_{i, [1:T]})\tp
				\begin{pmatrix}
					I_{T - \tau}
					\\
					\0
				\end{pmatrix}
				B
				\begin{pmatrix}
					\0
					\\
					I_{T - \tau}
				\end{pmatrix}\tp
				(\sigma_j \ff_{j, [1:T]}^{(L)} + \epsilonb_{j, [1:T]} )
				   \\
			&= 
			(\zz_{i,[1-L: T]}\tp , \epsilonb_{i,[1:T]}\tp)
			(\Psi_1^{ij}(B))
			(\zz_{i,[1-L: T]}\tp , \epsilonb_{i,[1:T]}\tp)\tp.
		\end{align*}
		Applying (\ref{lemma - bai silverstein concentration- main term}) of Lemma ~\ref{lemma - bai silverstein concentration-} to the above quadratic form gives
		\begin{align*}
			\numberthis\label{equation - DCT in proof xBx concentration}
			\Em\left|\xx_{i,[1:T- \tau]}^{(L)\top} B \xx_{j, [\tau+1 :T]}^{(L)}
			- \Em\left[\xx_{i,[1:T- \tau]}^{(L)\top} B \xx_{j, [\tau+1 :T]}^{(L)}\right]\right|^q
			&\lesssim \norm{\Psi^{ij}_1(B)}_F^q,
		\end{align*}
		where $\Em[\xx_{i,[1:T]}^{(L)\top} B \xx_{j, [1:T]}^{(L)}] 
			= 1_{i=j}\tr (\Psi_1^{ii}(B) ).$
		Using 
		Lemma  \ref{lemma - properties of Psi} we see that
		\begin{align*}
			\norm{\Psi^{ij}_0(B)}_F^q = O\left({\sigma_i^q \sigma_j^q}\right)\norm{B}_F^q
			=(T- \tau)^{q/2} O\left({\sigma_i^q \sigma_j^q}\right)\norm{B}^{q}
			=T^{q/2} O\left({\sigma_i^q \sigma_j^q}\right)\norm{B}^{q},
		\end{align*}
		and
			$$\Em[\xx_{i,[1:T]}^{(L)\top} B \xx_{j, [1:T]}^{(L)}] = 1_{i=j} O(\sigma_i^2 (T- \tau)) \norm{B}
			= 1_{i=j}  O(\sigma_i^2 T) \norm{B},$$
		both of which are uniform in $L$.

		Since $(\phi_{il})_l$ is summable and $(z_{it})$ have uniformly bounded $4$-th moments, it is clear that $\xx_{i}^{(L)}/\sigma_i $ converges to $\xx_{i}/\sigma_i$ in $L^4$ as $L\to\infty$, for each fixed $T$. By the dominated convergence theorem with \eqref{equation - DCT in proof xBx concentration} as an upper-bound, we can take the limit as $L\to\infty$ inside the expectation in \eqref{equation - DCT in proof xBx concentration} and the claim follows. 
	
		\pfspace
		\pf{\ref{lemma- concentration of xBx (FBE) }} follows from similar computations as in (a) and is omitted.

		\pfspace
		\pf{\ref{lemma- concentration of xBx (XRX) }}
		Note that $E_\tau\tp E_\tau E\tp E$ has bounded operator norm under the event $\cB_0$ defined in \eqref{equation - high prob event}. 
		Since $a \asymp \sigma_n^4 \gamma_n(\tau)^2$ diverges as $T\to\infty$, the resolvent $R(a)$ is well-defined under $\cB_0$ and $\norm{R(a)^k1_{\cB_0}} = O(1)$. 
		After conditioning on the $\sigma$-algebra $\cF$ defined in \eqref{equation - definition of sigalg CFp},  we can then apply (\ref{lemma- concentration of xBx (XBX) }) of the Lemma and get
		\begin{align*}
			&\underline\Em\left|\xx_{i,[1:T- \tau]}\tp R(a)^k \xx_{j, [\tau+1 :T]} 1_{\cB_0}
			- \underline\Em\left[\xx_{i,[1:T- \tau]}\tp R(a)^k \xx_{j, [\tau+1 :T]}1_{\cB_0}\right]  \right|^q
			\lesssim T^{q/2}O(\sigma_i^q \sigma_j^q).
		\end{align*}
		Taking expectations again to remove the conditioning, we obtain
		\begin{align*}
			&\Em\left|\xx_{i,[1:T- \tau]}\tp R(a)^k \xx_{j, [\tau+1 :T]} 1_{\cB_0}
			- \underline\Em\left[\xx_{i,[1:T- \tau]}\tp R(a)^k \xx_{j, [\tau+1 :T]}1_{\cB_0}\right]  \right|^q
			\lesssim T^{q/2}O(\sigma_i^q \sigma_j^q).
		\end{align*}
		Note that $\underline\Em[\xx_{i,[1:T- \tau]}\tp R(a)^k \xx_{j, [\tau+1 :T]}1_{\cB_0}]=0$ for all $i\ne j$ by (\ref{lemma- concentration of xBx (XBX) }) of the Lemma. So
		\begin{align*}
			\xx_{i, [1, T- \tau]}\tp R(a)^k \xx_{j, [\tau+1 : T]}1_{\cB_0}
			=
			1_{i=j}\underline\Em[\xx_{i, [1, T- \tau]}\tp R(a)^k \xx_{i, [\tau+1 : T]}1_{\cB_0}] + O_{L^q}(\sigma_i \sigma_j \rt T ).
		\end{align*}
		By Lemma ~\ref{lemma - high prob event} we have $1_{\cB_0} = 1 - o_p(1)$, from which the last claim follows. 

		\pfspace
		\pf{\ref{lemma- concentration of xBx part d }} follows from similar computations to the above and is omitted.
	\end{proof}
\end{lemma}
Note that the expectations appearing in the previous lemma are conditional on the noise series $\epsilonb$. The following lemma gives a preliminary computation on the unconditional moments of certain quadratic forms. 
Recall matrices $B(a)$, $A(a)$ and $Q(a)$:
	\begin{align*}
		A(a):= \frac{1}{\rt a}&{ X_0 R(a) X_\tau\tp},
		\quad
		B(a):= \frac{1}{a}{ X_\tau  E_0\tp E_0 R(a) X_\tau\tp},
		\\
		&Q(a) := I_K - a^{-1}X_0 R_{a} E_\tau \tp E_\tau  X_0\tp.
	\end{align*}
\begin{lemma}\label{lemma - expectation of ABQ}
	Under the same setting as (\ref{lemma- concentration of xBx (XRX) }) of Lemma  \ref{lemma - concentration of xBx}, we have
	\begin{align*}
		\Em[&A(a)_{ij}1_{\cB_0}]  =  1_{i=j}\left(\frac{\sigma_i^2 \gamma_i(\tau)}{a^{1/2}} + o(1)\right)\\
		&\Var(A(a)_{ij} 1_{\cB_0})  =   O\left(\frac{\sigma_{i}^2 \sigma_j^2}{a T}\right) ,
		\\
		\Em[B(a)_{ij}&1_{\cB_0}]  =  1_{i=j}o(1),
		\quad\Em[Q(a)^{-1}_{ij}1_{\cB_2}]  =  1_{i=j}+ o(1).
	\end{align*}

	\begin{proof}
		Since $\xx_i = \sigma_i\ff_i + \epsilonb_i$, we first observe that
		\begin{align*}
			\numberthis\label{equation - expectation of xoxtau 2}
			\frac{1}{\rt a T}\Em [\xx_{i, [1, T- \tau]}\tp \xx_{j, [\tau+1 : T]}]
			=
			\frac{1}{\rt a T}\Em [\sigma_i^2 \ff_{i, [1, T- \tau]}\tp \ff_{j, [\tau+1 : T]}] = 1_{i=j}\frac{\sigma_i^2\gamma_i(\tau) }{\rt a}.
		\end{align*}
		By definition, the event $\cB_0$ is independent from the vector $\xx$. Therefore 
		\begin{align*}
			\Em[A(a)_{ij}1_{\cB_0}] & =  
			\frac{1}{\rt a T}\Em [\xx_{i, [1, T- \tau]}\tp \xx_{j, [\tau+1 : T]}1_{\cB_0}]
			+
			\frac{1}{\rt a T}\Em [\xx_{i, [1, T- \tau]}\tp (R(a) - I)1_{\cB_0} \xx_{j, [\tau+1 : T]}]
			\\
			& = 
			1_{i=j}\frac{\sigma_i^2\gamma_i(\tau) }{\rt a} \Pm(\cB_0)
			+
			\frac{1}{\rt a T}\Em [\xx_{i, [1, T- \tau]}\tp (R(a) - I)1_{\cB_0} \xx_{j, [\tau+1 : T]}]
			\\
			& = 
			1_{i=j}\left(\frac{\sigma_i^2\gamma_i(\tau) }{\rt a} + o(1)\right)
			+
			\frac{1}{\rt a T}\Em [\xx_{i, [1, T- \tau]}\tp (R(a) - I)1_{\cB_0} \xx_{j, [\tau+1 : T]}],
		\end{align*}
		where the last equality follows since $\Pm(\cB_0) = 1+ o(1)$ by Lemma ~\ref{lemma - high prob event}. 
		It remains to compute the last expectation above. 
		Recall from \eqref{equation - first resolvent identity} that the resolvent $R(a)$ satisfies $R(a) - I = a^{-1} E_\tau\tp E_\tau E\tp E R(a)$. By definition of $\cB_0$ we have $\norm{E_\tau\tp E_\tau E\tp E 1_{\cB_0}} = O(1)$ and $\norm{R(a) 1_{\cB_0}}=O(1)$. Therefore
		\begin{align*}
		 	(R(a) - I)1_{\cB_0} = O_{\norm{\cdot}}(a^{-1}).
		 	\numberthis\label{euqation - (Ra-I) 1cF is o(1)}
		\end{align*} 
		Using \eqref{euqation - (Ra-I) 1cF is o(1)} and (\ref{lemma- concentration of xBx (XBX) }) of Lemma ~\ref{lemma - concentration of xBx} and taking iterated expectations we obtain
		\begin{align*}
			\frac{1}{\rt a T}\Em [\xx_{i, [1, T- \tau]}\tp (R(a)& - I)1_{\cB_0} \xx_{j, [\tau+1 : T]}]
			=
			\frac{1}{\rt a T}\Em\big[ \underline\Em[\xx_{i, [1, T- \tau]}\tp (R(a) - I)1_{\cB_0} \xx_{j, [\tau+1 : T]}]\big]
			\\
			& =  
			1_{i=j}\frac{1}{\rt a T} O({\sigma_i^2}T)\Em [\norm{R(a) - I }1_{\cB_0}]
			=
			1_{i=j} o(1).
		\end{align*}

		For the second moment, 
		using $(a-b)^2  = (a-c)^2 + (c-b)^2 +2(a-c)(c-b)$, we write
		\begin{align*}
			(A(a)_{ij}& 1_{\cB_0} - \Em[A(a)_{ij} 1_{\cB_0}])^2
			\numberthis\label{equation - proof of 2nd moment of Aij decomp}
			\\
			& =  
			(A(a)_{ij} 1_{\cB_0} - \underline\Em[A(a)_{ij} 1_{\cB_0}])^2
			+ (\underline\Em[A(a)_{ij} 1_{\cB_0}] -\Em[A(a)_{ij} 1_{\cB_0}])^2
			\\
			&
			\quad+ 2 (A(a)_{ij} 1_{\cB_0} - \underline\Em[A(a)_{ij} 1_{\cB_0}])(\underline\Em[A(a)_{ij} 1_{\cB_0}] -\Em[A_{ij} 1_{\cB_0}]).
		\end{align*}
		where by \eqref{lemma- concentration of xBx (XRX) } of
		Lemma ~\ref{lemma - concentration of xBx} we have
		\begin{align*}
			\Em \left[(A(a)_{ij} 1_{\cB_0} - \underline\Em[A(a)_{ij} 1_{\cB_0}])^2\right] = 
			\frac{1}{a T^2}O \left({\sigma_i^2 \sigma_j^2T}\right) = O\left(\frac{\sigma_i^2 \sigma_j^2}{a T}\right),
		\end{align*}		and from Lemma ~\ref{lemma - conditional expectations} (whose proof does not depend on the current lemma) we recall 
		\begin{align*}
			\Em \left[(\underline\Em[A(a)_{ij} 1_{\cB_0}] -\Em[A(a)_{ij} 1_{\cB_0}])^2\right] = O\left(\frac{1}{a T}\right).
		\end{align*}
		Taking expectation of \eqref{equation - proof of 2nd moment of Aij decomp} and using the Cauchy Schwarz inequality we have
		\begin{align*}
			\Em[(A(a)_{ij} 1_{\cB_0} - &\Em[A(a)_{ij} 1_{\cB_0}])^2] = O\left(\frac{\sigma_i^2 \sigma_j^2}{a T}\right).
		\end{align*}
		The expectation of $B(a)$ can be computed based on the same ideas and is omitted. 

		Lastly, under the event ${\cB_2}$, the matrix $Q(a)$ is invertible with $\norm{Q(a)1_{{\cB_2}}} = O(1)$. We recall from \eqref{equation - first resolvent identity for Q} that the inverse of $Q(a)$ satisfies
		\begin{align*}
			Q(a)^{-1} = I_K + \frac{1}{a}Q(a)^{-1} X_0 R(a) E_\tau\tp E_\tau X_0\tp.
		\end{align*}
		By definition of ${\cB_2}$ we know $1_{\cB_2} Q(a)^{-1} X_0 R(a) E_\tau\tp E_\tau X_0\tp = O_{\norm{\cdot}}(\sigma_1^2)$ and therefore 
		\begin{align*}
			Q(a)^{-1} 1_{\cB_2} = 1_{\cB_2} I_K  + o_{\norm{\cdot}}(1)
		\end{align*}
		and the last claim follows after taking expectations.
	\end{proof}
\end{lemma}

\section{Estimates on resolvents}\label{section - resolvents}
Define the following families of $\sigma$-algebras $(\cF_i)_{i=1}^p$ and $(\underline\cF_i)_{i=1}^K$ by
\begin{align*}
	\cF_i:= \sigma\big( \epsilonb_{[K+1, K+i], [1:T]} \big),
	\quad 
	\underline \cF_i := \sigma\big(  \xx_{[1:i],[1:T]}, \epsilonb_{[K+1,K+p], [1:T]} \big),
\end{align*}
i.e. $\cF_i$ is the $\sigma$-algebra generated by first $i$ coordinates of the noise series $\epsilonb$ and $\underline\cF_i$ is generated by all $p$ coordinates of $\epsilonb$ plus the first $i$ coordinates of the series $\xx$. 

Throughout the appendix we will write
\begin{align*}
	\Em_i[\ \cdot \ ]:=\Em_i[\ \cdot \ |\cF_i],\quad \underline\Em_i[\ \cdot \ ]:=\Em_i[\ \cdot \ |\underline\cF_i].
	\numberthis\label{notations - conditional expectations}
\end{align*}
Note that by definition $\Em_0[\ \cdot \ ] = \Em [\ \cdot \ ]$ and
	$\Em_p[\ \cdot \ ] = \underline \Em_0[\ \cdot \ ] =\underline\Em[\ \cdot \ ]$. 

We first develope a concentration inequality for normalized traces of the  resolvent $R$.
\begin{lemma}\label{lemma - concentration of R}
	For any matrix $B$ with $T - \tau$ columns, we have
	{\begin{align}
			\tag{a}\label{equation  - lemma concentration of R}
			\frac{1}{T}\tr( B (R1_{\cB_0} - \Em[R1_{\cB_0}])) & =  O_{L^2}\left(\frac{\norm{B}}{\theta\rt{ T}} \right),
			\\
			\tag{b}\label{equation  - lemma concentration of EER}
			\frac{1}{T}\tr( B (E_0\tp E_0 R1_{\cB_0} - \Em[E_0\tp E_0& R1_{\cB_0}]))  =  O_{L^2}\left(\frac{\norm{B}}{\rt{ T}} \right).
		\end{align}}
	\vspace{-5mm}
	\begin{proof}\pf{\ref{equation  - lemma concentration of R}}
	Similar to Lemma ~\ref{lemma - concentration of Q inverse} the proof is based on a martingale difference decomposition of $R  1_{\cB_0}- \Em[R 1_{\cB_0}]$. We first setup the necessary notations and carry out some preliminary computations. 

		Recall that the $k$-th row of $E_0$ is equal to $T^{-1/2}\epsilonb_{K+k,[1:T - \tau]}\tp$. For brevity of notation we will adopt the following notation
		\begin{align*}
			\numberthis\label{equation - definition of underline epsilon k}
			\underline \epsilonb_{k0}:=\epsilonb_{K+k,[1:T - \tau]},
			\quad 
			\underline \epsilonb_{k \tau}:=\epsilonb_{K+k,[\tau+1 :T]}.
		\end{align*}
		Let $E_{k0}$ and $E_{k\tau}$ be the matrices $E_0$ and $E_\tau$ with the $k$-th row replaced by zeros, i.e. $E_{k0}:= E_0 - T^{-1/2} \ee_k \underline \epsilonb_{k 0}$ and
		$E_{k\tau}:= E_0 - T^{-1/2} \ee_k \underline \epsilonb_{k \tau}$. Define
		\begin{align*}
			\underline R_{k}:= \Big(I_T - \frac{1}{\theta} E_{k\tau}\tp E_{k\tau} E_0\tp E_0\Big)^{-1},
			\quad
			R_{k}:= \Big(I_T - \frac{1}{\theta} E_{k\tau}\tp E_{k\tau} E_{k0}\tp E_{k0}\Big)^{-1},
		\end{align*}
		where $R_k$ is not to be confused with $R_\theta$ and $R_\lambda$ defined previously.
		Then
		\begin{align*}
			E_{0}\tp E_0 -  E_{k0}\tp E_{k0}
			=\frac{1}{T} \underline \epsilonb_{k0} \underline \epsilonb_{k0}\tp,
			\quad
			E_{\tau}\tp E_\tau -  E_{k\tau}\tp E_{k\tau}
			=\frac{1}{T} \underline \epsilonb_{k \tau} \underline \epsilonb_{k \tau}\tp,
		\end{align*}
		from which we can compute
		\begin{align*}
			&R^{-1} - \underline R_{k}^{-1}
			=- \frac{1}{\theta} (E_\tau\tp E_\tau - E_{k\tau}\tp E_{k\tau}) E_0\tp E_0
			= - \frac{1}{\theta T} \underline \epsilonb_{k \tau} \underline \epsilonb_{k \tau}\tp E_0\tp E_0
			\\
			&\underline R_{k}^{-1} - R_{k}^{-1}
			=- \frac{1}{\theta}  E_{k\tau}\tp E_{k\tau} (E_0\tp E_0 - E_{k0}\tp E_{k0})
			= - \frac{1}{\theta T}  E_{k\tau}\tp E_{k\tau}
			\underline \epsilonb_{k0} \underline \epsilonb_{k0}\tp.
		\end{align*}
		We furthermore define scalars
		\begin{align*}
			 & \underline\beta_k = \frac{1}{1+ \tr (\underline R_{k}(R^{-1} - \underline R_{k}^{-1}))}
			=
			\frac{1}{1- \frac{1}{\theta T}\underline \epsilonb_{k \tau}\tp E_0\tp E_0 \underline R_{k} \underline \epsilonb_{k \tau} },
			\\
			 & \beta_{k} = \frac{1}{1+ \tr (R_{k}(\underline R_{k}^{-1} - R_{k}^{-1}))}
			=
			\frac{1}{1- \frac{1}{\theta T}\underline \epsilonb_{k0}\tp R_{k} E_{k\tau}\tp E_{k\tau}
				\underline \epsilonb_{k0}  },
		\end{align*}
		both of which are clearly of order $1+ o(1)$ under the event $\cB_0$. 
		Using \eqref{equation - sherman 1} we get
		\begin{subequations}
			\begin{align*}
				\numberthis\label{equation - R- Rbark}
				 & R - \underline R_{k} =
				-\underline\beta_k \underline R_{k}(R ^{-1} - \underline R_{k}^{-1})\underline R_{k}
				= \frac{\underline\beta_k}{\theta T} \underline R_{k}\underline \epsilonb_{k \tau} \underline \epsilonb_{k \tau}\tp E_0\tp E_0\underline R_{k},
				\\
				\numberthis\label{equation - Rbark- Rk}
				 & \underline R_{k}- R_{k}
				=-\beta_k  R_{k}(\underline R_k ^{-1} -  R_{k}^{-1}) R_{k}
				= \frac{\beta_{k}}{\theta T} R_{k}E_{k\tau}\tp E_{k\tau}
				\underline \epsilonb_{k0} \underline \epsilonb_{k0}\tp R_{k}.
			\end{align*}
		\end{subequations}
		Substituting \eqref{equation - Rbark- Rk} back into \eqref{equation - R- Rbark} we get
		\begin{align*}
			R - \underline R_{k}
			 & =  \frac{\underline\beta_k}{\theta T}
			\Big(R_{k}
			+ \frac{\beta_k}{\theta T} R_{k}E_{k\tau}\tp E_{k\tau}
			\underline \epsilonb_{k0} \underline \epsilonb_{k0}\tp R_{k}\Big)
			\underline \epsilonb_{k \tau} \underline \epsilonb_{k \tau}\tp E_0\tp E_0
			\Big(R_{k}
			+ \frac{\beta_k}{\theta T} R_{k}E_{k\tau}\tp E_{k\tau}
			\underline \epsilonb_{k0} \underline \epsilonb_{k0}\tp R_{k}\Big),
		\end{align*}
		and so we have
		\begin{align*}
			\numberthis\label{equation - long version of R -}
			R -  R_{k} =  (\underline R_{k}- R_{k}) + (R- \underline R_{k})
			& = : U_1 + U_2 + U_3 + U_4 + U_5,
		\end{align*}
		where we have defined
		\begin{align*}\numberthis\label{equation - defn of Un}
			&U_1  :=  \frac{\beta_{k}}{\theta T} R_{k}E_{k\tau}\tp E_{k\tau}
			\underline \epsilonb_{k0} \underline \epsilonb_{k0}\tp R_{k},
			\quad
			U_2  := \frac{\underline\beta_k}{\theta T}  R_{k}\underline \epsilonb_{k \tau} \underline \epsilonb_{k \tau}\tp E_0\tp E_0R_{k},
			\\
			&U_3  := \frac{\underline\beta_k \beta_k}{\theta^2 T^2}
			R_k\underline \epsilonb_{k \tau} \underline \epsilonb_{k \tau}\tp E_0\tp E_0
			R_{k}E_{k\tau}\tp E_{k\tau}
			\underline \epsilonb_{k0} \underline \epsilonb_{k0}\tp R_{k},
			\\
			&U_4  :=
			\frac{\underline\beta_k \beta_k}{\theta^2 T^2}
			R_{k}E_{k\tau}\tp E_{k\tau}
			\underline \epsilonb_{k0} \underline \epsilonb_{k0}\tp R_{k}
			\underline \epsilonb_{k \tau} \underline \epsilonb_{k \tau}\tp E_0\tp E_0
			R_{k},
			\\
			&U_5   :=  \frac{\underline\beta_k \beta_k^2}{\theta^3 T^3}
			R_{k}E_{k\tau}\tp E_{k\tau}
			\underline \epsilonb_{k0} \underline \epsilonb_{k0}\tp R_{k}
			\underline \epsilonb_{k \tau} \underline \epsilonb_{k \tau}\tp E_0\tp E_0
			R_{k}E_{k\tau}\tp E_{k\tau}
			\underline \epsilonb_{k0} \underline \epsilonb_{k0}\tp R_{k}.
		\end{align*}
		Recall the event $\cB_0$ from \eqref{lemma - high prob event}. Define
		\begin{align*}
			\cB_0^k := \left\{\norm{E_{k0} \tp E_{k0}} + \norm{	E_{k \tau}\tp E_{k\tau}} \le 4\left(1+ \frac{	p}{	T}	\right) \right\},
			\quad k=1, \ldots, p.\numberthis\label{equation - B0k}
		\end{align*}
		Clearly $\norm{E_{k0} \tp E_{k0}}\le \norm{E_{0} \tp E_{0}}$ which implies $\cB_0\subseteq \cB_0^k$ and so $1_{\cB_0}\le 1_{\cB_0^k}$.
		Recall the family of conditional expectations $\Em_i[\ \cdot\ ]$ defined in \eqref{notations - conditional expectations}. Then
		\begin{align*}
			\frac{1}{T}\tr (B &(R 1_{\cB_0} - \Em[R 1_{\cB_0}])) 
			  =   \frac{1}{T}\sum_{k=1}^p (\Em_{k} - \Em_{k-1}) \tr (BR 1_{\cB_0})
				\\
			& =  \frac{1}{T}\sum_{k=1}^p (\Em_{k} - \Em_{k-1}) \left(\tr (BR 1_{\cB_0}) - \tr (BR_k 1_{\cB_0^k}) \right)
			\\
			& =  \frac{1}{T}\sum_{k=1}^p (\Em_{k} - \Em_{k-1}) \tr (B(R- R_k) 1_{\cB_0})  
			-
			\frac{1}{T}\sum_{k=1}^p (\Em_{k} - \Em_{k-1}) \tr (BR_k (1_{\cB_0^k}- 1_{\cB_0 } ))
			\\
			& =: I_1 + I_2, 
			\numberthis\label{equation - conc Rk 1}
		\end{align*}
		where the second equality holds
		since $\Em_{k} [\tr (BR_k 1_{\cB_0^k})] =  \Em_{k-1}[\tr (BR_k 1_{\cB_0^k})]$ and the third equality is purely algebraic computations. We first deal with the second term in \eqref{equation - conc Rk 1}. Using $\tr(B R_k)\le p \norm{B R_k} $ and $\norm{B R_k 1_{\cB_0^k}} = O(\norm{B})$ we have
		\begin{align*}
			\Em|I_2|^2
			&=
			\frac{1}{T^2}\sum_{k=1}^p \Em\left|(\Em_k - \Em_{k-1}) \tr(B R_k( 1_{\cB_0^k}-1_{\cB_0 })) \right|^2
			\le 
			\frac{4p^2}{T^2}\sum_{k=1}^p \Em\left| \norm{B R_k } ( 1_{\cB_0^k}-1_{\cB_0 }) \right|^2
			\\
			& =
			O\left(\frac{p^2 }{T^2}\norm{B}^2\right)\sum_{k=1}^p\Em\left| 1_{\cB_0^k}-1_{\cB_0 }\right|^2
			=
			O\left(\frac{p^2}{T^2} \norm{B}^2\right) \sum_{k=1}^p \Pm(\cB_0^c) = o(T^{-l} \norm{B}^2),
		\end{align*}
		for any $l\in\N$ by Lemma ~\ref{lemma - high prob event}.
		For the first term in \eqref{equation - conc Rk 1}, since $I_1$ is a sum of  a martingale difference sequence, using \eqref{equation - long version of R -} and $\cB_0\subseteq \cB_0^k$ we have
		\begin{align*}
			\Em|I_1|^2
			& \le \frac{1}{T^2}  \sum_{k=1}^p \Em \big|(\Em_{k} - \Em_{k-1}) \tr(B(R - R_k)1_{\cB_0^k})\big|^2
			\\
			         & \le 
			        \frac{4}{T^2}\sum_{k=1}^p \Em \big|\tr(B(R - R_k)1_{\cB_0^k})\big|^2
			\le \frac{20}{T^2}\sum_{k=1}^p\sum_{n=1}^5 \Em \big|\tr(BU_n1_{\cB_0^k})\big|^2,
		\end{align*}
		and it remains to bound the second moment of each $\tr(B U_n 1_{\cB_0^k})$. 
		Since  $\{\epsilon_{it}\}$ are assumed to be i.i.d. standard Gaussian, we have the following moment estimate
		\begin{align*}
			\numberthis\label{equation - moment estimate of gaussians}
			\Em \left[\norm{\underline \epsilonb_{k0}}^n\right]
			=
			\Em \Bigg[ \Big(\sum_{t=1}^{T- \tau} \epsilon_{kt}^2 \Big)^{n/2} \Bigg]
			\lesssim (T- \tau)^{n/2-1} \sum_{t=1}^{T- \tau} \Em|\epsilon_{kt}|^n = O(T^{n/2}).
		\end{align*}
		Using $\beta_k 1_{\cB_0^k}=1+o(1)$ and the trivial inequality $x\tp A x \le \norm{x}^2 \norm{A}$ we obtain
		\begin{subequations}
		\begin{align*}
			\numberthis\label{equation - bound on trBU1}
			\Em & \big|\tr(BU_11_{\cB_0^k})\big|^2
			 \lesssim
			\frac{1}{\theta^2 T^2} \Em \left [(\underline \epsilonb_{k0}\tp R_{k}B R_{k}E_{k\tau}\tp E_{k\tau}
			\underline \epsilonb_{k0} )^21_{\cB_0^k}\right]
			\\
			 & \le
			\frac{1}{\theta^2 T^2} \Em \left[
				\norm{\underline \epsilonb_{k0}}^4 \norm{R_{k}}^4\norm{E_{k\tau}\tp E_{k\tau}} ^2
				1_{\cB_0^k}
				\right] \norm{B}^2
				\lesssim
			\frac{1}{\theta^2}  \norm{B}^2.
		\end{align*}
		The second term $U_2$ can be dealt in exactly the same way to obtain
		\begin{align*}
			\numberthis\label{equation - bound on trBU2}
			\Em \big|\tr(BU_21_{\cB_0^k})\big|^2 \lesssim
			\frac{1}{\theta^2}  \norm{B}^2,
		\end{align*}
		and we omit the details. For $U_3$, similar computations gives
		\begin{align*}
			\Em \big|\tr(BU_31_{\cB_0^k})\big|^2
			 & \lesssim
			\frac{1}{\theta^4 T^4}
			\Em \left[
			(\underline \epsilonb_{k0}\tp R_{k}B R_k\underline \epsilonb_{k \tau} \underline \epsilonb_{k \tau}\tp E_0\tp E_0
			R_{k}E_{k\tau}\tp E_{k\tau}
			\underline \epsilonb_{k0} )^21_{\cB_0^k}
			\right]
			\\
			 & \le
			\frac{1}{\theta^4 T^4}
			\Em \left[
			\norm{\underline \epsilonb_{k0}}^4 \norm{\underline \epsilonb_{k \tau}}^4
			\norm{R_{k}}^4   \norm{E_0\tp E_0 R_{k}E_{k\tau}\tp E_{k\tau}}^21_{\cB_0^k}
			\right]\norm{B}^2,
		\end{align*}
		since $x\tp Ay \le \norm{x} \norm{y} \norm{A}$. Therefore
		\begin{align*}
			\Em \big|\tr(BU_3)\big|^2 \lesssim
			\frac{1}{\theta^4 T^4}
			\Em \left[
				\norm{\underline \epsilonb_{k0}}^8 \right]^{1/2}
			\Em \left[ \norm{\underline \epsilonb_{k \tau}}^8
				\right]^{1/2} 
			\lesssim \frac{1}{\theta^4 } \norm{B}^2.
			\numberthis\label{equation - bound on trBU3}
		\end{align*}
		Once again $U_4$ can be bounded in the same way to obtain
		\begin{align*}
			\numberthis\label{equation - bound on trBU4}
			\Em \big|\tr(BU_4)1_{\cB_0^k}\big|^2
			\le \frac{1}{\theta^4 } \norm{B}^2.
		\end{align*}
		With the same approach but more laborious computations we can obtain
		\begin{align*}
			\numberthis\label{equation - bound on trBU5}
			\Em \big|\tr(BU_5)1_{\cB_0^k}\big|^2
			\lesssim \frac{1}{\theta^6 } \norm{B}^2.
		\end{align*}
		\end{subequations}
		Note that the estimates \eqref{equation - bound on trBU1}-\eqref{equation - bound on trBU5} are uniform in $k=1,\ldots, p$.  We then conclude
		\begin{align*}
			\Em\Big|\frac{1}{T} & \tr (B (R1_{\cB_0^k} - \Em[R1_{\cB_0^k}]))\Big|^2 = O\left(\frac{p}{T^2\theta^2}\right)\norm{B}^2 = O\left(\frac{1}{T \theta^2}\right)\norm{B}^2,
		\end{align*}
		and the conclusion follows.

		\pfspace
		\pf{\ref{equation  - lemma concentration of EER}}
		Similar to \eqref{equation  - lemma concentration of R}, via a martingale difference decomposition we obtain
		\begin{align*}
			\Em\Big|\frac{1}{T}\tr &(B (E_0\tp E_0R1_{\cB_0^k} - \Em[E_0\tp E_0R1_{\cB_0^k}])) \Big|^2
			\lesssim    
			\frac{1}{T^2} \sum_{k=1}^p \Em\big|\tr(B(E_0\tp E_0 R - E_{k0}\tp E_{k0} R_k))1_{\cB_0^k}\big|^2,
		\end{align*}
		where, recalling the $U_n$'s defined  in the proof of \eqref{equation  - lemma concentration of R}, we have
		\begin{align*}
		\numberthis\label{equation - mart decomp for EER}
			E_0\tp E_0 R - E_{k0}\tp E_{k0} R_k 
			 & =   
			\frac{1}{T}\underline\epsilonb_{k0}\underline\epsilonb_{k0}\tp R_k
			+ \frac{1}{T}\underline\epsilonb_{k0}\underline\epsilonb_{k0}\tp (R - R_k)
			+ E_{k0}\tp E_{k0}(R - R_k)
			   \\
			& = \frac{1}{T}\underline\epsilonb_{k0}\underline\epsilonb_{k0}\tp R_k
			+ \frac{1}{T}\sum_{n=1}^5 \underline\epsilonb_{k0}\underline\epsilonb_{k0}\tp U_n
			+ \sum_{n=1}^5 E_{k0}\tp E_{k0}U_n.
		\end{align*}
		We deal with the first two term in \eqref{equation - mart decomp for EER} to illustrate the ideas of the proof, the other terms can be dealth with similarly. Using \eqref{equation - moment estimate of gaussians} and $p\asymp T$, clearly we have
		\begin{align*}
			\numberthis\label{equation - leading term in mart decomp for EER}
			\frac{1}{T^2}  \sum_{k=1}^p  \Em \Big| \frac{1}{T}	\tr ( B \underline\epsilonb_{k0}\underline\epsilonb_{k0}\tp R_k)1_{\cB_0^k}\Big|^2
			\lesssim 
			\frac{1}{T^4}  \sum_{k=1}^p  T^2 \Em[\norm{B R_k1_{\cB_0^k}}^2] = O\left(\frac{1}{T}\right)\norm{B}^2.
		\end{align*}
		Similar to the computations in \eqref{equation - bound on trBU1}, we can get
		\begin{align*}
			\Em \Big|  \frac{1}{T}\tr(B ( &     \underline\epsilonb_{k0}\underline\epsilonb_{k0}\tp U_1     ))1_{\cB_0^k}
			   \Big|^2
			\lesssim 
			\frac{1}{\theta^2 T^2} \frac{1}{T^2}
			\Em [(\underline \epsilonb_{k0}\tp R_{k}B \underline\epsilonb_{k0}\underline\epsilonb_{k0} \tp 
						R_{k} E_{k\tau}\tp E_{k\tau}
						\underline \epsilonb_{k0} )^2 1_{\cB_0^k}]
						\\
			 & \le
			\frac{1}{\theta^2 T^4} \Em \left[
				\norm{\underline \epsilonb_{k0}}^8 \norm{R_{k}}^4\norm{E_{k\tau}\tp E_{k\tau}} ^2
				1_{\cB_0^k}
				\right] \norm{B}^2
				\lesssim
			\frac{1}{\theta^2}  \norm{B}^2,
		\end{align*}
		which immediately gives
		\begin{align*}
			\frac{1}{T^2} \sum_{k=1}^p \Em \Big|  \frac{1}{T}\tr(B ( &     \underline\epsilonb_{k0}\underline\epsilonb_{k0}\tp U_1     ))1_{\cB_0^k}
			   \Big|^2
			   = O\left(\frac{p}{\theta ^2 T^2}\right) \norm{B}^2
			   =O\left(\frac{1}{\theta ^2 T}\right) \norm{B}^2.
		\end{align*}
		Note that this term is negligible in comparison to \eqref{equation - leading term in mart decomp for EER}. 
		Using the same ideas, 
		it is routine to check that the other 9 terms in \eqref{equation - mart decomp for EER} are negligible as well, and we omit the details. The bound therefore follows from \eqref{equation - leading term in mart decomp for EER}. 
	\end{proof}
\end{lemma}

Next recall that 
		$
			Q = I_K - \frac{1}{\theta} X_0  R E_\tau \tp E_\tau X_0\tp.
		$
We now state a concentration inequality for entries of the matrix $Q^{-1}$, under the event $\cB_2$.
\begin{lemma}
	\label{lemma - concentration of Q inverse}
	Write $Q_{ij}^{-1}:= (Q^{-1})_{ij}$. Then
	\begin{enumerate}
		\item \label{lemma - concentration of Q inverse ---- diag}
		For all $k=1,\ldots, K$,  we have
		\begin{align*}
		  Q_{kk}^{-1}1_{\cB_2}  - \underline \Em  [Q_{kk}^{-1}1_{\cB_2}] = 
		  O_{L^1}\left(\frac{\sigma_k^2}{\theta \rt T}\right).
		\end{align*}

		\item \label{lemma - concentration of Q inverse ---- offdiag}
		The off-diagonal elements of $Q^{-1}$ satisfies
		\begin{align*}
			Q_{ij}^{-1}1_{\cB_{2}} =  O_{L^2}\left(\frac{\sigma_i \sigma_j}{	\theta \rt T}\right)
		\end{align*}
		uniformly in $i,j = 1,\ldots, K$, $i\ne j$.
	\end{enumerate}

	\begin{proof}
		\pf{\ref{lemma - concentration of Q inverse ---- diag}}
		Recalling the event $\cB_2$, we note that the matrix $Q$ is invertible with probability tending to 1. 
		The proof relies on expressing $Q_{kk}^{-1}1_{\cB_2}  - \underline \Em [Q_{kk}^{-1}1_{\cB_2}] $ as a sum of martingale differences. We first setup the notations necessary.

		Let $T^{-1/2}\xx_i :=T^{-1/2} \xx_{i, [1:T- \tau]}$ be the (column vector) of the $i$-th row of $X_0$, i.e. we can write $X_0 = T^{-1/2} \sum_{i=1}^K  \ee_i\xx_i\tp$. Define 
		$X_{i0} := X_0 - \frac{1}{\rt{T}} \ee_{i} \xx_i\tp,$
		and
		\begin{align*}
			Q_{(i)} := I_K - \frac{1}{\theta} X_{i0}  R E_\tau \tp E_\tau X_0\tp,
			\quad
			Q_{(ii)} := I_K - \frac{1}{\theta} X_{i0}  R E_\tau \tp E_\tau X_{i0}\tp,
		\end{align*}
		from which we can immediately compute
		\begin{align*}
			Q - Q_{(i)} 
			= - \frac{1}{\theta \rt T} \ee_{i} \xx_i\tp R E_\tau \tp E_\tau X_0\tp,
			\quad 
			Q_{(i)} - Q_{(ii)}
			=
			- \frac{1}{\theta\rt T} X_{i0}  R E_\tau \tp E_\tau \xx_i \ee_i\tp.
		\end{align*}
		Note that all elements on the $i$-th row of $Q_{(i)}$ are equal to zero except for the diagonal which is equal to 1, i.e. $Q_{(i)}$ is equal to the identity when restricted to the $i$-th coordinate. Then the inverse $Q_{(i)}^{-1}$, whenever it exists, must also equal to the identity when restricted to the $i$-th coordinate. A similar observation can be made for the matrix $Q_{(ii)}$ and it is not hard to observe that 
		\begin{align*}
			\numberthis\label{equation - zeroes in Qii}
			\ee_i\tp& (Q_{(ii)})^{-1} \ee_i = 1, 
			\quad \ee_i \tp (Q_{(i)})^{-1}\ee_j=  0, \quad \forall j\ne i,
			\\
			&\ee_i\tp (Q_{(ii)})^{-1} \ee_j =\ee_j\tp (Q_{(ii)})^{-1} \ee_i = 0, \quad \forall j\ne i.
		\end{align*}
		To compute the difference $Q^{-1} - Q_{(i)}^{-1}$, which will turn out to be the  central focus of the proof, we first define the following scalars
		\begin{align*}
			\numberthis\label{equation - bi in the proof of Q}
			b_i :& =  \frac{1}{1+ \tr( Q_{(i)}^{-1}(Q - Q_{(i)}) )} 
			=
			\frac{1}{1 - \frac{1}{\theta \rt T}  \xx_i\tp R E_\tau \tp E_\tau X_0\tp Q_{(i)}^{-1}\ee_{i}}
			   \\
			b_{ii} :& =  \frac{1}{1+ \tr( Q_{(ii)}^{-1}(Q_{(i)} - Q_{(ii)}) )} 
			=
			\frac{1}{1 - \frac{1}{\theta\rt T} \ee_i\tp Q_{(ii)}^{-1}X_{i0}  R E_\tau \tp E_\tau \xx_i }=1,
		\end{align*}
		where the last equality holds by \eqref{equation - zeroes in Qii}. Then using the identity \eqref{equation - sherman 1} we have
		\begin{subequations}
		\begin{align*}
			\numberthis\label{equation - Qinv - Qinvi}
			&Q^{-1} - Q_{(i)}^{-1} =  \frac{b_i}{\theta \rt T} 
			Q_{(i)}^{-1}\ee_{i} \xx_i\tp R E_\tau \tp E_\tau X_0\tp Q_{(i)}^{-1},
			\\
			\numberthis\label{equation - Qinvi - Qinvii}
			&Q_{(i)}^{-1} - Q_{(ii)}^{-1} =  \frac{1}{\theta \rt T} 
			Q_{(ii)}^{-1} X_{i0}  R E_\tau \tp E_\tau \xx_i \ee_i\tp Q_{(ii)}^{-1}.
		\end{align*}
		We observe that the matrices $Q_{(i)}^{-1}$ and $Q_{(ii)}^{-1}$ differ only on off-diagonal elements on the $i$-th column. Indeed, from
		\eqref{equation - zeroes in Qii} and \eqref{equation - Qinvi - Qinvii}, if $n\ne i$ or if $n=m=i$ then
		\end{subequations}
		\begin{align*}
			\numberthis\label{equation - Qinvi - Qinvii nonzeroes}
			\ee_m\tp (Q_{(i)}^{-1} - Q_{(ii)}^{-1})\ee_n =  \frac{1}{\theta \rt T} 
			\ee_m\tp Q_{(ii)}^{-1} X_{i0}  R E_\tau \tp E_\tau \xx_i \ee_i\tp Q_{(ii)}^{-1}
			\ee_n=0.
		\end{align*}
		Then, substituting \eqref{equation - Qinvi - Qinvii} back into \eqref{equation - Qinv - Qinvi} we obtain
		\begin{align*}
			\ee_k\tp(Q^{-1} -& Q_{(i)}^{-1})\ee_k
			 = 
			\frac{b_i}{\theta \rt T}\ee_k\tp   Q_{(i)}^{-1}\ee_{i} \xx_i\tp R E_\tau \tp E_\tau X_0\tp Q_{(i)}^{-1}  \ee_k
			\\
			& = 
			\frac{b_i}{\theta \rt T}\ee_k\tp   Q_{(i)}^{-1}\ee_{i} \xx_i\tp R E_\tau \tp E_\tau X_0\tp Q_{(ii)}^{-1}  \ee_k
			   \\
			& \quad\quad\quad\quad + 
			\frac{b_i}{\theta^2 T}\ee_k\tp   Q_{(i)}^{-1}\ee_{i} \xx_i\tp R E_\tau \tp E_\tau X_0\tp 
			Q_{(ii)}^{-1} X_{i0}  R E_\tau \tp E_\tau \xx_i \ee_i\tp Q_{(ii)}^{-1}  \ee_k
			\\
			& = 
			\frac{b_i}{\theta \rt T}\ee_k\tp   Q_{(ii)}^{-1}\ee_{i} \xx_i\tp R E_\tau \tp E_\tau X_0\tp Q_{(ii)}^{-1}  \ee_k
			   \\
			& \quad \quad\quad\quad+ 
			\frac{b_i}{\theta^2 T}\ee_k\tp   Q_{(i)}^{-1}\ee_{i} \xx_i\tp R E_\tau \tp E_\tau X_0\tp 
			Q_{(ii)}^{-1} X_{i0}  R E_\tau \tp E_\tau \xx_i \ee_i\tp Q_{(ii)}^{-1}  \ee_k
			\\
			& \quad \quad\quad\quad+ 
			\frac{b_i}{\theta^2 T}\ee_k\tp   Q_{(ii)}^{-1} X_{i0}  R E_\tau \tp E_\tau \xx_i \ee_i\tp Q_{(ii)}^{-1}\ee_{i} \xx_i\tp R E_\tau \tp E_\tau X_0\tp Q_{(ii)}^{-1}  \ee_k
			   \\
			& =: I_1 + I_2 + I_3. 
			\numberthis\label{equation - martingale Q into I123}
		\end{align*}
		To simplify this expression further, define the following quadratic forms
		\begin{align*}
			\xi_i: =  \frac{1}{\theta  T}&  \xx_i\tp R E_\tau \tp E_\tau \xx_i,
			\quad
			\eta_i :  =  \frac{1}{\theta^2 T}  \xx_i\tp R E_\tau \tp E_\tau X_{i0}\tp  
			Q_{(ii)}^{-1} X_{i0}  R E_\tau \tp E_\tau \xx_i ,
		\\
			&\zeta_{ik}:= \frac{1}{\theta^2 T}\xx_i\tp R E_\tau \tp E_\tau X_{i0}\tp Q_{(ii)}^{-1}  \ee_k\ee_k\tp   Q_{(ii)}^{-1} X_{i0}  R E_\tau \tp E_\tau \xx_i ,
			\numberthis\label{equation - definition of xi in proof of Q}
		\end{align*}
		then using \eqref{equation - zeroes in Qii}, we can easily write $I_1, I_2$ and $I_3$ into
		\begin{align*}
			\numberthis\label{equation - I123 in proof of Q}
			I_1 & =  1_{i=k}
			\frac{b_k}{\theta \rt T} \xx_k\tp R E_\tau \tp E_\tau X_0\tp Q_{(kk)}^{-1}  \ee_k
			= 1_{i=k} b_k \xi_k,
			   \\
			I_2  
			& =  1_{i=k}\frac{b_k}{\theta^2 T} 
			\xx_k\tp R E_\tau \tp E_\tau X_0\tp 
			Q_{(kk)}^{-1} X_{k0}  R E_\tau \tp E_\tau \xx_k
			= 1_{i=k} b_k \eta_k ,
			\\
			I_{3} & =  1_{i\ne k}
			\frac{b_i}{\theta^2 T}\ee_k\tp   Q_{(ii)}^{-1} X_{i0}  R E_\tau \tp E_\tau \xx_i 
			 \xx_i\tp R E_\tau \tp E_\tau X_0\tp Q_{(ii)}^{-1}  \ee_k
			= 1_{i\ne k} b_i \zeta_{ik}.
		\end{align*}
		We first state some estimates on $\xi$ and $ \eta$ under the appropriate events. Recall from \eqref{equation - high prob event} the event
		$\cB_1 := \left\{ \norm{X_0\tp X_0}  \le 2\sum_{i=1}^K \sigma_i^2  \right\}$. Define the event
		\begin{align*}
			\cB_1^i := \Big\{ \norm{X_{i0}\tp X_{i0}}\le 2\sum_{i=1}^K \sigma_i^2   \Big\},\quad i =1 ,\ldots, K
			\numberthis\label{equation - high prob event i in proof of Q}
		\end{align*}
		and write $\cB_2^i:=\cB_0\cap \cB_1^i$. Then clearly $\cB_2^i\subseteq \cB_2$.
		Define
		\begin{align*}
			\overline \xi_i:= \frac{1}{\theta  T}  \tr\big( \Psi_{0}^{ii}(R E_\tau \tp E_\tau)\big),
			\quad
			\overline \eta_i :  =  \frac{1}{\theta^2 T}  \tr\big( \Psi^{ii}_0(R E_\tau \tp E_\tau X_{i0}\tp Q_{(ii)}^{-1} X_{i0}  R E_\tau \tp E_\tau) \big),
		\end{align*}
		where $\Psi^{ii}_0$ is defined in \eqref{equation - definition of Psi}. Write 
		\begin{align*}
			\underline \xi_i := \xi_i - \overline \xi_i,\quad
			\underline \eta_i := \eta_i - \overline \eta_i.
			\numberthis\label{equation - definition of xii underline}
		\end{align*}
		Using Lemma ~\ref{lemma - concentration of xBx} and taking iterated expectations we have
		\begin{align*}
			&\hspace{3mm}
			\Em \big[\underline \xi_i^2 1_{\cB_2^i}\big]
			=\Em \big[\underline \Em [\underline \xi_i^2 1_{\cB_2^i}] \big]
			= \frac{1}{\theta^2 T^2} O(\sigma_i^4 T) \Em \norm{R E_\tau\tp E 1_{\cB_2^i}}^2
			= O\left(\frac{\sigma_i^4}{\theta^2 T}\right),
			   \\
			&
			\Em \big[\underline \eta_i^21_{\cB_2^i}\big]
			=
			\frac{1}{\theta^4 T^2} O(\sigma_i^4 T) 
			\Em \norm{R E_\tau \tp E_\tau X_{i0}\tp  
			Q_{(ii)}^{-1} X_{i0}  R E_\tau \tp E_\tau 1_{\cB_2^i}}^2 = O\left(\frac{\sigma_i^4 \sum_{j=1}^K  \sigma_j^4}{\theta^4 T}\right).
		\end{align*}
		By Lemma ~\ref{lemma - properties of Psi} we also have
		\begin{align*}
			\numberthis\label{equation - moment of underline xi 2}
			\overline \xi_i 1_{\cB^i_2} =  O(\sigma_i^2\theta^{-1}) &,\quad
			 \overline\eta_i 1_{\cB^i_2} = O\left(\frac{\sigma_i^2\sum_{j=1}^K\sigma_j^2}{\theta^2}\right) .
		\end{align*}
		We then consider the scalar $b_i$ defined in \eqref{equation - bi in the proof of Q}. From  \eqref{equation - zeroes in Qii} and \eqref{equation - Qinvi - Qinvii} we observe
		\begin{align*}
			\frac{1}{\theta \rt T}  &\xx_i\tp R E_\tau \tp E_\tau X_0\tp Q_{(i)}^{-1}\ee_{i}
			= 
			\frac{1}{\theta \rt T}  \xx_i\tp R E_\tau \tp E_\tau X_0\tp Q_{(ii)}^{-1}\ee_{i}
			\\&\quad+ \frac{1}{\theta^2 T}  \xx_i\tp R E_\tau \tp E_\tau X_0\tp  
			Q_{(ii)}^{-1} X_{i0}  R E_\tau \tp E_\tau \xx_i \ee_i\tp Q_{(ii)}^{-1}\ee_{i}
			=
			\xi_i + \eta_i.
		\end{align*}
		Substituting back into \eqref{equation - bi in the proof of Q} 
		we can simplify to obtain
		\begin{align*}
			b_i = (1- \xi_i - \eta_i)^{-1}
			\numberthis\label{equation - bi - new expression}. 
		\end{align*}
		Define 
		$\overline b_i = (1- \overline \xi_i - \overline \eta_i)^{-1}$ so that subtracting the two we get
		\begin{align*}
			b_i = (1- \xi_i - \eta_i)^{-1}
			& =  \overline b_i - b_i \overline b_i (\underline \xi_i + \underline \eta_i).
			\numberthis\label{equation - bi - bibar}
		\end{align*}
		Finally, from the expression \eqref{equation - bi - new expression} and the bounds \eqref{equation - moment of underline xi 2} we clearly have
		\begin{align*}
			b_i 1_{\cB_2} = 1+ o(1), \quad 
			\overline b_i 1_{\cB_2^i}= 1 + o(1).
			\numberthis\label{equation - bi bbar bounded}
		\end{align*}

		We can now carry out the main idea of the proof. Recall notations $\underline\Em[\ \cdot \ ]$ and 
		$\underline\Em_i[\ \cdot\ ]$ from \eqref{notations - conditional expectations}. By definition of $Q_{(ii)}$ and $\cB_2^i$ we have
		\begin{align*}
			\ee_k\tp  \big(Q^{-1}1_{\cB_2} - \underline \Em[&Q^{-1}1_{\cB_2}]\big)\ee_k
			  =  
			\sum_{i=1}^K (\underline\Em_i - \underline\Em_{i-1})  
			\left(\ee_k\tp  Q^{-1}1_{\cB_2} \ee_k 
			- \ee_k\tp  Q_{(ii)}^{-1} 1_{\cB_2^i}\ee_k\right)
			\\
			& 
			=
			\sum_{i=1}^K (\underline\Em_i - \underline\Em_{i-1})  
			\left(\ee_k\tp  Q^{-1}1_{\cB_2} \ee_k 
			- \ee_k\tp  Q_{(i)}^{-1} 1_{\cB_2^i}\ee_k\right) ,
		\end{align*}
		where the last equality follows from \eqref{equation - Qinvi - Qinvii nonzeroes}.
		Similar to how we dealt with the second term in \eqref{equation - conc Rk 1} in the proof of Lemma ~\ref{lemma - concentration of R}, using Lemma ~\ref{lemma - high prob event} we may obtain
		\begin{align*}
			\ee_k\tp  \big(Q^{-1}1_{\cB_2} - \underline \Em &[Q^{-1}1_{\cB_2}]\big)\ee_k
			  = 
			\sum_{i=1}^K (\underline\Em_i - \underline\Em_{i-1})  \ee_k \left(Q^{-1}  - Q_{(ii)}^{-1} \right)1_{\cB_2} \ee_k
			+ 
			O_{L^2}(KT^{-1})
			\\
			& = 
			\sum_{i=1}^K (\underline\Em_i - \underline\Em_{i-1}) (I_1 + I_2 + I_3)1_{\cB_2} 
			+O_{L^2}(KT^{-1}),
			\numberthis\label{equation - conc Qi 1}
		\end{align*}
		where the 
		second equality holds by \eqref{equation - Qinvi - Qinvii nonzeroes}.

		As will be shown, the term involving $I_1$ is the leading term of \eqref{equation - conc Qi 1}, this is what we consider now.
		Using the identity \eqref{equation - I123 in proof of Q} we simply have
		\begin{align*}
			\sum_{i=1}^K (\underline\Em_i - \underline\Em_{i-1})I_11_{\cB_2^i}
			=
			(\underline\Em_k - \underline\Em_{k-1})b_k \xi_k 1_{\cB_2},
		\end{align*}
		which, recalling \eqref{equation - definition of xii underline} and using \eqref{equation - bi - bibar}, can be written into
		\begin{align*}
			(\underline\Em_k - \underline\Em_{k-1})&b_k \xi_k1_{\cB_2}
			   =   
			(\underline\Em_k - \underline\Em_{k-1})
				\left(
				\overline b_k - b_k \overline b_k (\underline \xi_k + \underline \eta_k)\right)
			\left(
				\overline \xi_k +  \underline \xi_k\right)1_{\cB_2}
			   \\
			& = (\underline\Em_k- \underline\Em_{k-1}) \Big[ \overline b_k \overline \xi_k+ 
				\overline b_k \underline \xi_k
				- b_k \overline b_k (\underline \xi_k + \underline \eta_k)(
				\overline \xi_k +  \underline \xi_k)
				\Big]1_{\cB_2}.
				\numberthis\label{equation expression for I1 in Q}
		\end{align*}
		We consider the three terms in the square bracket in \eqref{equation expression for I1 in Q} separately. For the first term, we note that $(\underline\Em_k- \underline\Em_{k-1})\overline b_k \overline \xi_k1_{\cB_2^k} =0$ by definition of $\overline b_k \overline \xi_k$ and $\cB_2^k$. Using this, we have
		\begin{align*}
			(\underline\Em_k- \underline\Em_{k-1}) \overline b_k \overline \xi_k 1_{\cB_2}
			=
			0
			- (\underline\Em_k- \underline\Em_{k-1}) \overline b_k \overline \xi_k (1_{\cB_2^k} - 1_{\cB_2}).
		\end{align*}
		Recalling \eqref{equation - moment of underline xi 2} and \eqref{equation - bi bbar bounded} and using Assumptions \ref{assumptions}  we have
		\begin{align*}
			\Em\big|(\underline\Em_k- \underline\Em_{k-1}) \overline b_k \overline \xi_k  1_{\cB_2}\big|
			&\le 2\Em \big|\overline b_k \overline \xi_k (1_{\cB_2^k} - 1_{\cB_2})\big|
			=	
			O(\sigma_k^2\theta^{-1}) \Em| 1_{\cB_2^k} - 1_{\cB_2}| = 
			O(\sigma_k^2 \theta^{-1 }KT^{-1}),
		\end{align*}
		where the last equality follows from the fact that ${\cB_2}\subseteq \cB_2^k$ and Lemma ~\ref{lemma - high prob event}. For the second term in \eqref{equation expression for I1 in Q}, using ${\cB_2}\subseteq \cB_2^k$, \eqref{equation - bi bbar bounded} and \eqref{equation - moment of underline xi 2} we have
		\begin{align*}
			\Em \big|(\underline\Em_k - \underline\Em_{k-1})& \overline b_k \underline \xi_k 1_{\cB_2}\big|^2
			\lesssim 
			4 \Em \big|\overline b_k \underline \xi_k 1_{\cB_2^{k}}\big|^2 
			= O\left(\frac{\sigma_k^4}{\theta^2 T}\right).
		\end{align*}
		Similarly the third term of \eqref{equation expression for I1 in Q} is bounded by
		\begin{align*}
			\Em\big| (\underline\Em_k- &\underline\Em_{k-1})\big[ b_k \overline b_k (\underline \xi_k + \underline \eta_k)(\overline\xi_k+ \underline \xi_k)
			1_{\cB_2}\big]\big|
			\lesssim 2\Em\big| (\underline \xi_k + \underline \eta_k)(\overline\xi_k+ \underline \xi_k)
			1_{\cB_2^k}\big|.
		\end{align*}
		Expanding, applying the Cauchy-Schwarz inequality and using \eqref{equation - bi bbar bounded} and \eqref{equation - moment of underline xi 2}, we may obtain a bound of order $o_{L^1}(T^{-1/2})$; we omit the repetitive details.
		Substituting the above bounds back into equation  \eqref{equation expression for I1 in Q} we obtain
		\begin{align*}
			\Em\big| \sum_{i=1}^K (\underline\Em_i - \underline\Em_{i-1})I_1
			 \big|
			=
			\Em \Big| (\underline\Em_k - \underline\Em_{k-1})b_k \xi_k \Big|
			= O\left(\frac{\sigma_k^2}{\theta \rt T}\right). 
		\end{align*}
		The cases of $I_2$ and $I_3$ can be dealt with with similar approaches and we omit the details. In fact, from the definitions in \eqref{equation - I123 in proof of Q} it is not difficult to see that $\eta$ and $\zeta$ are higher order terms relative to $\xi$ under the event $\cB_2$. 
		It can therefore be shown that the term involving $I_1$ is the leading term in \eqref{equation - conc Qi 1}
		and the claim follows.

	\pfspace
	\pf{\ref{lemma - concentration of Q inverse ---- offdiag}}
	Define $\overline Q:= I_T - \theta^{-1} X_0\tp X_0 R E_\tau\tp E_\tau$ so that similar to \eqref{equation - first resolvent identity for Q} we have
	\begin{align*}
		Q^{-1}- I_K = \frac{1}{\theta}X_0 R E_\tau\tp E_\tau X_0\tp
		(I_K - \theta^{-1} X_0 R E_\tau\tp E_\tau X_0\tp )^{-1} 
		= \frac{1}{\theta}X_0 R E_\tau\tp E_\tau 
		\overline Q^{-1} X_0\tp.
	\end{align*}
	Recall that we have $X_0\tp X_0 = T^{-1}\sum_{i=1}^K \xx_i \xx_i\tp$, define the matrices
	\begin{align*}
		\overline Q_{(j)}:= I_T - \frac{1}{\theta T} \sum_{k\ne j} \xx_k \xx_k\tp R E_\tau\tp E_\tau,
		\quad
		\overline Q_{(ij)}:= I_T - \frac{1}{\theta T} \sum_{k\ne i,j} \xx_k \xx_k\tp R E_\tau\tp E_\tau,
	\end{align*}
	so that 
	$\overline Q - \overline Q_{(j)} = - \frac{1}{\theta T}\xx_j \xx_j\tp R E_\tau\tp E_\tau$ and
	$\overline Q_{(j)} - \overline Q_{(ij)} = - \frac{1}{\theta T}\xx_i \xx_i\tp R E_\tau\tp E_\tau$. Let
	\begin{align*}
		&a_j:= \frac{1}{1+ \tr  (\overline Q_{(j)}^{-1} (\overline Q - \overline Q_{(j)} ))}
		=
		\frac{1}{1 - \frac{1}{\theta T}
			\xx_j\tp R E_\tau\tp E_\tau \overline Q_{(j)}^{-1} \xx_j },
			\\
		&a_{ij}:= \frac{1}{1+ \tr  (\overline Q_{(ij)}^{-1} (\overline Q_{(j)} - \overline Q_{(ij)} ))}
		=
		\frac{1}{1 - \frac{1}{\theta T}
			\xx_i\tp R E_\tau\tp E_\tau \overline Q_{(ij)}^{-1} \xx_i },
	\end{align*}
	then by \eqref{equation - sherman 2} we have 
	\begin{align*}
		\overline Q^{-1} \xx_j = a_j \overline Q_{(j)}^{-1} \xx_j,
		\quad 
		\xx_i\tp R E_\tau\tp E_\tau\overline Q_{(i)}^{-1} = a_{ij} \xx_i\tp R E_\tau\tp E_\tau \overline Q_{(ij)}^{-1}.
	\end{align*}
	We can therefore write 
	\begin{align*}
		Q^{-1}_{ij} = \frac{1}{\theta T} \xx_i\tp R E_\tau\tp E_\tau \overline Q^{-1} \xx_j
		=
		\frac{a_j a_{ij}}{\theta T} \xx_i\tp R E_\tau\tp E_\tau \overline Q_{(ij)}^{-1} \xx_j.
	\end{align*}
	Now define $X_{ij0}:= X_0 - T^{-1/2}(\ee_i \xx_i\tp + \ee_j \xx_j)$ and 
	events
	$\cB_1^{ij}$ and $\cB_2^{ij}$ analogous to
	\eqref{equation - high prob event i in proof of Q} with $X_{i0}$ replaced by $X_{ij0}$. Similar to (a) of the Lemma we have
	$a_j= 1+ o(1)$ and $a_{ij}= 1+ o(1)$ under the event $\cB_2$. Therefore we have
	\begin{align*}
		\Em\big|Q_{ij}^{-1} 1_{\cB_{2}} \big|^2 
		\lesssim \frac{	1}{	\theta^2 T^2}	\Em\big|\xx_i\tp R E_\tau\tp E_\tau \overline Q_{(ij)}^{-1} \xx_j 1_{\cB_{2}} \big|^2
		\le \frac{	1}{	\theta^2 T^2}	\Em\big|\xx_i\tp R E_\tau\tp E_\tau \overline Q_{(ij)}^{-1} \xx_j 1_{\cB_{2}^{ij}} \big|^2.
	\end{align*}
	By Lemma ~\ref{lemma - concentration of xBx} we have
	\begin{align*}
		\Em\big|Q_{ij}^{-1} 1_{\cB_{2}} \big| ^2
		= \frac{1}{\theta^2 T^2}
		O(\sigma_i^2 \sigma_j^2 T),
	\end{align*}
	which completes the whole proof.
	\end{proof}
\end{lemma}

We finally show that the conditional expectations of diagonal elements of $A$, $B$ and $Q^{-1}$, defined in \eqref{equation - definition of A and B}, are sufficiently close to the unconditional expectation. 
\begin{lemma}\label{lemma - conditional expectations}
		For each $i=1,\ldots, K$, we have
		\begin{align*}
			\underline\Em[A_{ii}1_{\cB_0}] - \Em[A_{ii}1_{\cB_0}]
			=  O_{L^2}\left(  \frac{\sigma_i^2}{ {\theta^{3/2} \rt T}}  \right),
			\quad 
			\underline \Em[&B_{ii}1_{\cB_0}]  - \Em[B_{ii}1_{\cB_0}]
			 = O_{L^2}\left(  \frac{\sigma_i^2}{ {\theta\rt  T}}  \right),
			\\
			\underline\Em[Q^{-1}_{ii}1_{\cB_2}] - \Em[Q^{-1}_{ii}1_{\cB_2}]
			&=
			O_{L^2}\left(\frac{K \norm{\sigmab}_{\ell_2}^2}{\theta \rt T} \right).
		\end{align*}
	\begin{proof}
		From (a) of Lemma ~\ref{lemma - concentration of xBx} we recall that
		\begin{align*}
			\underline\Em [A_{ii} 1_{\cB_0}] 
			=
			\frac{1}{\rt \theta T}\underline\Em [\xx_{i,[1:T-\tau]}\tp R1_{\cB_0} \xx_{i,[\tau+1:T]} ] 
			=
			\frac{1}{\rt \theta T}\tr (\Psi^{ii}_1(R))1_{\cB_0}
			\numberthis\label{equation - proof of EtrR 0}
		\end{align*}
		where, using \eqref{equation - definition of Psi} and the cyclic property of the trace, we have
		\begin{align*}
			\tr (\Psi^{ii}_1(R)) 
			= \tr \left(  (\0,I_{T- \tau})(\sigma_i^2 \Phi_i\tp \Phi_i + I_{T})(I_{T- \tau},\0)\tp R  \right)=:\tr(GR).
			\numberthis\label{equation - proof of EtrR 1}
		\end{align*}
		Furthermore, using (a) of Lemma ~\ref{lemma - properties of Psi}, we see that
		\begin{align*}
			G:= (\0,I_{T- \tau})(\sigma_i^2 \Phi_i\tp \Phi_i + I_{T})(I_{T- \tau},\0)\tp = O_{\norm{\cdot}}(\sigma_i^2). 
			\numberthis\label{equation - proof of EtrR 2}
		\end{align*}
		From \eqref{equation - proof of EtrR 0} we have $\Em[A_{ii}1_{\cB_0}] = \Em[\underline \Em[A_{ii}1_{\cB_0}]] = \frac{1}{\rt \theta T}\Em[\tr (\Psi^{ii}_1(R))1_{\cB_0}]$ and so
		\begin{align*}
			\underline \Em[A_{ii}1_{\cB_0}]&  - \Em[A_{ii}1_{\cB_0}]
			 = 
			\frac{1}{\rt \theta T}\left( \tr (\Psi^{ii}_1(R))1_{\cB_0} - \Em[\tr (\Psi^{ii}_1(R))1_{\cB_0}]\right)
			\\
			& = 
			\frac{1}{\rt \theta T}\left( \tr (GR)1_{\cB_0} - \Em[\tr (GR)1_{\cB_0}]\right)
			= \frac{1}{\rt \theta T} \tr \big(
				G (R 1_{\cB_0} - \Em[R1_{\cB_0}] )
			\big),
		\end{align*}
		by linearity of the expectation and the trace.
		By \eqref{equation  - lemma concentration of R} of Lemma ~\ref{lemma - concentration of R} we have
		\begin{align*}
			\underline \Em[A_{ii}1_{\cB_0}]  - \Em[A_{ii}1_{\cB_0}]
			& = \frac{1}{\rt \theta}O_{L^2}   \left(\frac{\norm{G}}{\theta \rt T} \right)
			=O_{L^2}   \left(\frac{\sigma_i^2}{ { \theta^{3/2} \rt T}} \right),
		\end{align*}
		where the last equality follows from \eqref{equation - proof of EtrR 2}. 
		For the case of $B$, similar computations and \eqref{equation  - lemma concentration of EER} of Lemma ~\ref{lemma - concentration of R} give
		\begin{align*}
			\underline \Em[B_{ii}1_{\cB_0}] & - \Em[B_{ii}1_{\cB_0}]
			 = 
			\frac{1}{\theta T}\left( \tr (\Psi^{ii}_1(E_0\tp E_0 R))1_{\cB_0} - \Em[\tr (\Psi^{ii}_1(E_0\tp E_0 R))1_{\cB_0}]\right)
			\\
			& 
			=\frac{1}{\theta T}\tr\left(  G (E_0\tp E_0 R1_{\cB_0} - \Em[E_0\tp E_0 R1_{\cB_0}]\right)
			=
			\frac{1}{\theta }O_{L^2}\left(\frac{ \norm{G}}{\rt T}\right) =
			O_{L^2}\left(  \frac{\sigma_i^{2}}{{\theta \rt T}}  \right).
		\end{align*}
		Finally we consider $\underline\Em [Q_{ii}^{-1}]$. We recall from \eqref{equation - first resolvent identity for Q} that 
		\begin{align*}
			Q := I_K - \frac{1}{\theta} X_{0}  R E_\tau \tp E_\tau X_{0}\tp.
			\numberthis\label{equation - Q = I - ? in proof of last EQ}
		\end{align*}
		The strategy of the proof, similar to that of Lemma ~\ref{lemma - concentration of R} and Lemma ~\ref{lemma - concentration of Q inverse}, is express $\underline\Em[Q^{-1}_{ii}1_{\cB_2}] - \Em[Q^{-1}_{ii}1_{\cB_2}]$ into a sum of martingale difference sequence. We first introduce the necessary notations and carry out some algebraic computations.

		Similar to \eqref{equation - definition of underline epsilon k},
		we will define $\underline\epsilonb_{k0} := \epsilonb_{K+k,[1:T-\tau]}$ and $\underline\epsilonb_{k\tau} := \epsilonb_{K+k,[\tau+1:T]}$. 
		Recall from \eqref{equation - long version of R -} that $R - R_k = \sum_{n=1}^5 U_n$, where
		the $U_n$'s  are defined in \eqref{equation - defn of Un}. 
		Similar to the computations in \eqref{equation - mart decomp for EER}, we may obtain
		\begin{align*}
			R E_\tau\tp E_\tau &-  R_k E_{k\tau}\tp E_{k\tau}
			=  \frac{1}{T}R_k\underline \epsilonb_{k\tau}\underline \epsilonb_{k\tau}\tp
			+ \sum_{n=1}^5 U_n E_{\tau}\tp E_{\tau} =: V + W,
		\end{align*}
		where we defined 
		\begin{align*}
			&V  : =  \frac{1}{T}R_k\underline \epsilonb_{k\tau}\underline \epsilonb_{k\tau}\tp 
			+ (U_2 + U_3) E_\tau\tp E_\tau, 
			   \\
			&W : =  (U_1 + U_4 + U_5)E_\tau\tp E_\tau.
		\end{align*}
		Define matrices $V_{1}, V_{2}, V_3, W_{1}, W_{2}, W_{3}$ by,
		\begin{align*}
			V_1 := I_T,
			\quad
			&V_2:=\frac{\underline\beta_k}{\theta} E_0\tp E_0R_{k} E_\tau\tp E_\tau,
			 \quad
			V_3 :=  \frac{\underline\beta_k \beta_k}{\theta^2 T}
			E_0\tp E_0R_{k}E_{k\tau}\tp E_{k\tau}\epsilonb_{k0} \epsilonb_{k0}\tp R_{k} E_\tau\tp E_\tau
				\\
			&W_1  := \beta_{k}R_{k}E_\tau\tp E_\tau,
			\quad 
			W_2 :=\frac{\underline\beta_k \beta_{k}}{\theta T}
			R_{k}
			\underline \epsilonb_{k\tau} \underline \epsilonb_{k\tau}\tp E_0\tp E_0R_{k}E_\tau\tp E_\tau,
				\\
			&W_3:= \frac{\underline\beta_k \beta_k^2}{\theta^2 T^2}
			R_{k}
			\underline \epsilonb_{k\tau} \underline \epsilonb_{k\tau}\tp E_0\tp E_0
			R_{k}E_{k\tau}\tp E_{k\tau}
			\underline\epsilonb_{k0} \underline\epsilonb_{k0}\tp R_{k}E_\tau\tp E_\tau,
		\end{align*}
		so that using \eqref{equation - defn of Un} we can decompose $V$ and $W$ into
		\begin{subequations}
			\begin{align*}
			\numberthis\label{equation - V = V123}
			V & =  \frac{1}{T}R_k\underline \epsilonb_{k\tau}\underline \epsilonb_{k\tau}\tp  (V_1 + V_2 + V_3)
			   \\
			   \numberthis\label{equation - W = W123}
			W& = \frac{1}{\theta T} R_{k}E_{k\tau}\tp E_{k\tau}
			\underline\epsilonb_{k0} \underline\epsilonb_{k0}\tp (W_1 + W_2 + W_3).
		\end{align*}
		\end{subequations}
		It is clear that $V$ and $W$ are matrices of rank one. 
		We define
		\begin{align*}
			\underline Q_{(k)} :& =  I_K -\frac{1}{\theta} X_0 (R E_\tau\tp E_\tau - V) X_0\tp,
			\\
			\underline Q_{(kk)} :&= I_K -\frac{1}{\theta} X_0 (R E_\tau\tp E_\tau - V- W) X_0\tp,
		\end{align*}
		then from \eqref{equation - Q = I - ? in proof of last EQ}  we can write $Q - \underline Q_{(k)} = - \theta^{-1} X_0VX_0\tp$ and 
		$Q_{(k)} - \underline Q_{(kk)} = - \theta^{-1} X_0WX_0\tp$.
		Define the following scalars quantities
		\begin{align*}
			\alpha_k := \frac{1}{1 - \theta^{-1}\tr (Q_{(k)}^{-1} X_0  V X_0\tp)},
			\quad
			\alpha_{kk} := \frac{1}{1 - \theta^{-1}\tr (Q_{(kk)}^{-1} X_0  W X_0\tp)},
		\end{align*}
		then using \eqref{equation - sherman 1} we obtain
		\begin{align*}
			Q^{-1} = \underline Q_{(k)}^{-1}
			+ 
			\frac{1}{\theta}\underline Q_{(k)}^{-1} X_0 V X_0\tp \underline Q_{(k)}^{-1},
			\quad 
			\underline Q_{(k)}^{-1} = \underline Q_{(kk)}^{-1}
			+ 
			\frac{1}{\theta}\underline Q_{(kk)}^{-1} X_0 W X_0\tp \underline Q_{(kk)}^{-1}.
		\end{align*}
		Substituting the second identity into the first gives
		\begin{align*}
			&Q^{-1} -  \underline Q_{(kk)}^{-1}
			=  
			\frac{1}{\theta}
			\underline Q_{(kk)}^{-1} X_0 W X_0\tp \underline Q_{(kk)}^{-1}
			\\
			&\quad+ 
			\frac{1}{\theta} 
			\Big(\underline Q_{(kk)}^{-1}
				+ 
				\frac{1}{\theta}\underline Q_{(kk)}^{-1} X_0 W X_0\tp \underline Q_{(kk)}^{-1}\Big) 
				X_0 V X_0\tp 
			\Big(\underline Q_{(kk)}^{-1}
				+ 
				\frac{1}{\theta}\underline Q_{(kk)}^{-1} X_0 W X_0\tp \underline Q_{(kk)}^{-1}\Big),
		\end{align*}
		which after simplifying becomes 
		\begin{align*}
			\numberthis\label{equation - terms in conditional mean Q - Qkk}
			&Q^{-1} -  \underline Q_{(kk)}^{-1}
			 = 
			\frac{1}{\theta} \underline Q_{(kk)}^{-1}X_0 V X_0\tp  \underline Q_{(kk)}^{-1} 
			+
			\frac{1}{\theta}
			\underline Q_{(kk)}^{-1} X_0 W X_0\tp \underline Q_{(kk)}^{-1}
			   \\
			& +
			\frac{1}{\theta^2} \underline Q_{(kk)}^{-1}X_0 V X_0\tp\underline Q_{(kk)}^{-1} X_0 W X_0\tp \underline Q_{(kk)}^{-1}
			+
			\frac{1}{\theta^2}\underline Q_{(kk)}^{-1} X_0 W X_0\tp \underline Q_{(kk)}^{-1}X_0 V X_0\tp \underline Q_{(kk)}^{-1}.
		\end{align*}

		Before we proceed with the proof we first prove some moment estimates for the terms  in \eqref{equation - terms in conditional mean Q - Qkk}.
		We start with some informal observations. By comparing \eqref{equation - V = V123} and \eqref{equation - V = V123}, we see that the matrix $W$ is smaller in magnitude in comparison to $V$ by a factor of $\theta	^{-1}$. This suggests that the first term in \eqref{equation - terms in conditional mean Q - Qkk} is the leading term while the rest are high order terms in comparison and we will therefore only deal with first term in detail below.  The same arguments can be applied to the rest of \eqref{equation - terms in conditional mean Q - Qkk} to make the above argument rigorous, but we omit the repetitive details.

		Recall the family of event $\{\cB_0^k, k=1,\ldots, p\}$ from \eqref{equation - B0k} and define $\cB_2^k:= \cB_0^k\cap \cB_1$. 
			From definition we note that $\cB_2\subseteq\cB_2^k$.
		Furthermore, from Lemma ~\ref{lemma - high prob event} we have
		\begin{align*}
		 	1_{\cB_2^k} - 1_{\cB_2} \le 1- 1_{\cB_2}  = o_p(T^{-l}),\quad\forall l\in\N.
		 	\numberthis\label{equation - two indicators proof oo EQ}
		\end{align*}
		In the computations below, we will often substitute $1_{\cB_2^k}$ with $1_{\cB_2}$ and vice versa in expectations. Whenever we do so, we may use \eqref{equation - two indicators proof oo EQ} and a similar argument to how we dealt with \eqref{equation - conc Rk 1} to show that the error term of such a substitution is negligible for the purpose of the proof.
		Hence from now on we will use the two indicators $1_{\cB_2^k}$ and $1_{\cB_2}$ interchangeably below without further justifications. 

		Since we can write $X_0\tp = \frac{1}{\rt T} \sum_{l=1}^K\xx_l \ee_l\tp$, the first term in \eqref{equation - terms in conditional mean Q - Qkk} can be expressed as
		\begin{align*}
			\frac{1}{\theta} \ee_i\tp&\underline Q_{(kk)}^{-1}X_0 V X_0\tp  \underline Q_{(kk)}^{-1} \ee_i
			 =  	
			\frac{1}{\theta T} \sum_{l=1}^K\sum_{m=1}^K \ee_i\tp\underline Q_{(kk)}^{-1} \ee_l  (\xx_l\tp V  \xx_m )\ee_m\tp \underline Q_{(kk)}^{-1} \ee_i,
			\numberthis\label{equation - proof of EQ first term}
		\end{align*}
		where, recalling \eqref{equation - V = V123},  we have
		\begin{align*}
			\xx_l\tp V  \xx_m & =  
			\frac{1}{T}\sum_{n=1}^3  \xx_l\tp R_k\underline \epsilonb_{k\tau}\underline \epsilonb_{k\tau}\tp V_n  \xx_m.
			\numberthis\label{equation - proof of EQ first term 2}
		\end{align*}
		Using \eqref{equation - proof of EQ first term}-\eqref{equation - proof of EQ first term 2} and the inequality
		 $(\sum_{i=1}^n x_i)^p \lesssim n^{p-1} \sum_{i=1}^n x_i^p$ we have
		\begin{align*}
			\Em\Big|
			\frac{1}{\theta} \ee_i\tp \underline Q_{(kk)}^{-1}X_0 &V X_0\tp  \underline Q_{(kk)}^{-1} \ee_i  				1_{\cB_2}	\Big|^2
			\lesssim 
			\frac{K^2}{\theta^2 T^2}
			\sum_{l=1}^K\sum_{m=1}^K 
			\Em 
			\Big| 
				\ee_i\tp\underline Q_{(kk)}^{-1} \ee_l  (\xx_l\tp V  \xx_m )\ee_m\tp \underline Q_{(kk)}^{-1} \ee_i		1_{\cB_2}		\Big|^2
			\\
			&\lesssim 
			\frac{3K^2}{\theta^2 T^4}
			\sum_{l=1}^K\sum_{m=1}^K \sum_{n=1}^3
			\Em 
			\Big| 
				\xx_l\tp R_k\underline \epsilonb_{k\tau}\underline \epsilonb_{k\tau}\tp V_n  \xx_m \norm{\underline Q_{(kk)}^{-1}}^2 1_{\cB_2}\Big|^2
				.
				\numberthis\label{equation - QXVXQ}
		\end{align*}
		Note that under the event $\cB_2$, we can easily see that $\norm{Q^{-1}_{(kk)}1_{\cB_2}} = O(1)$. Therefore by the Cauchy Schwarz inequality we can obtain
		\begin{align*}
			\mathrm{\eqref{equation - QXVXQ}}
			\lesssim 
			\frac{K^2}{\theta^2 T^4}
			\sum_{l=1}^K\sum_{m=1}^K \sum_{n=1}^3
			\Em \big[(\xx_l\tp R_k\underline \epsilonb_{k\tau})^4 1_{\cB_2}\big]^{1/2}
			\Em \big[(\underline \epsilonb_{k\tau}\tp V_n  \xx_m)^4 1_{\cB_2}\big]^{1/2}.
			\numberthis\label{equation - QXVXQ 1}
		\end{align*}
		Note that $\cB_2 = \cB_1\cap \cB_0 \subseteq \cB_0\subseteq\cB_0^k$ 
		 so that $1_{\cB_2}\le 1_{\cB_0^k}$. We can then condition on $R_k$ and 
		apply (\ref{lemma - bai silverstein concentration- independent})  of Lemma ~\ref{lemma - bai silverstein concentration-} to
		the first quadratic form in \eqref{equation - QXVXQ}  to get
		\begin{align*}
			\Em |\xx_l\tp R_k \underline\epsilonb_{k\tau}1_{\cB_2}|^4 
			& \le  
			\Em \left|
			\begin{pmatrix}
				\zz_{l,[1:T]}\\ \epsilonb_{l,[1:T]}
			\end{pmatrix}\tp
			\begin{pmatrix}
			 	\sigma_l \Phi_l \\I_{T}
			\end{pmatrix}
			\begin{pmatrix}
			 	I_{T- \tau} \\ \0_{\tau\times(T- \tau)}
			\end{pmatrix}
			 R_k \underline\epsilonb_{k\tau}1_{\cB_0^k}\right|^4
			 \\
			&\lesssim 
			\Em 
			 \big[\tr(   R_k^2  (\sigma_l^2 \Phi_l\tp \Phi_l + I_T)  )^2 1_{\cB_0^k}\big]
			= O( \sigma_l^4 T^2),
		\end{align*}
		where the last equality follows from using $tr(R)\le T \norm{R}$ and applying Lemma ~\ref{lemma - properties of Psi}. Similarly, for the quadratic involving $V_1$ in \eqref{equation - QXVXQ 1}, we have
		\begin{align*}
			\Em |\underline\epsilonb_{k\tau}\tp V_1 \xx_m|^4  = \Em |\underline\epsilonb_{k\tau}\tp \xx_m|^4
			\lesssim \tr(\sigma_m^2 \Phi_m\tp \Phi_m + I_T) ^2
			= O(\sigma_m^4 T^2).
		\end{align*}
		
		We observe here that that the matrices $V_2$ and $V_3$ are smaller in magnitude in comparison to $V_1$ by a factor of $\theta^{-1}$ under the event $\cB_2$. Hence it is to be expected that the quadratic forms involving $V_2$ and $V_3$ in \eqref{equation - QXVXQ 1} should be negligible in comparison to the one involving $V_1$. To be more concrete, we sketch here how bound the quadratic form involving $V_2$; the case of $V_3$ can be dealt with in a similar manner.
		Recall that the matrix $E_0\tp E_0$ can be written as $E_0\tp E_0 = E_{k0} \tp E_{k0} + \frac{1}{T}\underline\epsilonb_{k0} \underline\epsilonb_{k0}\tp$. 
		Then we can write
		\begin{align*}
			\Em |\underline\epsilonb_{k\tau}\tp &V_2 \xx_m 1_{\cB_2}|^4 
			 = 
			\frac{1}{\theta^4} \Em \big|\underline \beta_k \underline\epsilonb_{k\tau}\tp E_0\tp E_0R_{k} E_\tau\tp E_\tau\xx_m1_{\cB_2}\big|^4
			   \\
			& \lesssim 
			\frac{1}{\theta^4} \Em \big| \underline\epsilonb_{k\tau}\tp E_{k0}\tp E_{k0}R_{k} E_{k\tau}\tp E_{k\tau}\xx_m1_{\cB_2}\big|^4
			+ 
			\frac{1}{\theta^4 } \Em \big| \frac{1}{T}\underline\epsilonb_{k\tau}\tp \underline\epsilonb_{k0}\frac{1}{T}\underline\epsilonb_{k0}\tp R_{k} \underline\epsilonb_{k0}\underline\epsilonb_{k0}\tp\xx_m1_{\cB_2}\big|^4
			\\&\quad
			+ \frac{1}{\theta^4} \Em \big| \frac{1}{T}\underline\epsilonb_{k\tau}\tp E_{k0}\tp E_{k0}R_{k} \underline\epsilonb_{k\tau}\underline\epsilonb_{k\tau}\tp\xx_m1_{\cB_2}\big|^4
			+ \frac{1}{\theta^4} \Em \big| \frac{1}{T}\underline\epsilonb_{k\tau}\tp \underline\epsilonb_{k0}\underline\epsilonb_{k0}\tp R_{k} E_{k0}\tp E_{k0}\xx_m1_{\cB_2}\big|^4.
		\end{align*}
		At this point we recognize that the four terms above has a similar structure as the case of $V_1$. Namely they all involve quadratic forms where the matrix in the middle is independent from the vectors on each side. Using the same approach as we did in the case of $V_1$ we can indeed show that this is a negligible term in comparison. The case of $V_3$ is similar albeit more tedious, and we omit the details.
		We may then conclude that
		\begin{align*}
			\Em\big|
			\frac{1}{\theta} \ee_i\tp &\underline Q_{(kk)}^{-1}X_0 V X_0\tp  \underline Q_{(kk)}^{-1} \ee_i1_{\cB_2}\big|^2
			\lesssim 
			\frac{K^2}{\theta^2 T^4}
			\sum_{l=1}^K\sum_{m=1}^K 
			\sigma_l^2 \sigma_m^2 T^2  
			\numberthis\label{equation - 2nd moment fo QXVXQ}.
		\end{align*}
		The same strategy described above can then be repeated for each of the remaining three terms in \eqref{equation - terms in conditional mean Q - Qkk} to show that they are higher order terms compared to the term in \eqref{equation - 2nd moment fo QXVXQ} (see the remark below \eqref{equation - terms in conditional mean Q - Qkk}).
		We may therefore conclude that
		\begin{align*}
			\Em|\ee_i\tp( Q^{-1} - \underline Q_{(kk)}^{-1})\ee_i1_{\cB_2}|^2
			\lesssim 
			\frac{K^2 }{\theta^2 T^2} \left(\sum_{m=1}^K \sigma_m^2\right)^2.
			\numberthis\label{equation - last?}
		\end{align*}
		Finally, we can decompose  $\underline\Em[\ee_i\tp Q^{-1}\ee_i 1_{\cB_2}] - \Em[\ee_i\tp Q^{-1}\ee_i 1_{\cB_2}]$ into
		\begin{align*}
			\underline\Em[\ee_i\tp Q^{-1}\ee_i1_{\cB_2}] - \Em[\ee_i\tp Q^{-1}\ee_i1_{\cB_2}]
			& = 
			\sum_{k=1}^p (\Em_i - \Em_{i-1})\ee_i\tp Q^{-1}\ee_i1_{\cB_2}
			\\
			& = \sum_{k=1}^p (\Em_i - \Em_{i-1})\ee_i\tp (Q^{-1}- Q^{-1}_{(kk)})\ee_i1_{\cB_2},
		\end{align*}
		where for the last equality we refer to \eqref{equation - two indicators proof oo EQ} and the remark immediately below it. Using the bound \eqref{equation - last?} we immediately have
		\begin{align*}
		\Em\big|\underline\Em[&\ee_i\tp Q^{-1}\ee_i1_{\cB_2}] - \Em[\ee_i\tp Q^{-1}\ee_i1_{\cB_2}]\big|^2
			\\
		&\le 4\sum_{k=1}^p \Em |\ee_i\tp( Q^{-1} - \underline Q_{(kk)}^{-1}1_{\cB_2})\ee_i|^2
		\lesssim
		\frac{K^2 }{\theta^2 T} \left(\sum_{m=1}^K \sigma_m^2\right)^2,
		\end{align*}
		from which the claim follows.
	\end{proof}
\end{lemma}

\section{Proof of Theorem~\ref{3:th:2}}\label{3:sec:Appendix_A}
\begin{proof}[Proof of Theorem~\ref{3:th:2}]
	Without loss of generality, we only consider the case for $Z_{i,\tau} > 0$ since the case for $Z_{i,\tau} < 0$ can be considered in precisely the same way. For a constant significant level $\alpha$, to see $Pr(Z_{i,\tau} > z_{\alpha} | H_1) \rightarrow 1$ as $T,p \to \infty$, it is sufficient to show that $Z_{i,\tau} \rightarrow \infty$ as $T,p \to \infty$.
	
	To start, we firstly notice that for any $i \in \{1,2,...,K\}$ and a finite time lag $\tau$, $\frac{\gamma_{i,\tau}}{2\sqrt{2}v_{i,\tau}}$ does not divergent with $T$ and $p$, since both $\gamma_{i,\tau}$ and $v_{i,\tau}$ are some constants when $T,p \to \infty$. It then suffices to show $\sqrt{T} \frac{\lambda_{i,\tau}^{(1)}-\lambda_{i,\tau}^{(2)}}{\theta_{i,\tau}} \to \infty$ when $T,p \to \infty$. Note that by the definition of $\theta_{i,\tau}$ in (\ref{3:para}), we can show that
	\begin{align}\label{3:ap1}
		\frac{\lambda_{i,\tau}^{(1)}-\lambda_{i,\tau}^{(2)}}{\theta_{i,\tau}}
		= \frac{\lambda_{i,\tau}^{(1)}}{\theta_{i,\tau}^{(1)}} \frac{\theta_{i,\tau}^{(1)}}{\theta_{i,\tau}}
		- \frac{\lambda_{i,\tau}^{(2)}}{\theta_{i,\tau}^{(2)}} \frac{\theta_{i,\tau}^{(2)}}{\theta_{i,\tau}}
		= \frac{\lambda_{i,\tau}^{(1)}}{\theta_{i,\tau}^{(1)}} \frac{2+2c}{2+c} -
		\frac{\lambda_{i,\tau}^{(2)}}{\theta_{i,\tau}^{(2)}} \frac{2}{2+c},
	\end{align}
	where the second equation follows from the fact that $\theta_{i,\tau}^{(1)} = (1+c) \theta_{i,\tau}^{(2)}$ and $\theta_{i,\tau} = \frac{\theta_{i,\tau}^{(1)}+ \theta_{i,\tau}^{(2)}}{2} = \frac{2+c}{2}\theta_{i,\tau}^{(2)}$. Moreover, under Assumptions \ref{assumptions} and \ref{assumptions - tau fixed}, we know from Theorem~\ref{theorem - CLT} that for $m = 1$ and $2$,
	\begin{align*}
		\sqrt{T} \frac{\gamma_{i,\tau}^{(m)}}{2v_{i,\tau}^{(m)}} \frac{\lambda_{i,\tau}^{(m)}-\theta_{i,\tau}^{(m)}}{\theta_{i,\tau}^{(m)}} \Rightarrow \mathcal{N}(0,1),
	\end{align*}
	as $T,p \to \infty$ where $\theta_{i,\tau}^{(m)}$ is the asymptotic centering of $\lambda_{i,\tau}^{(m)}$. As a result, 
	\begin{align*}
		\frac{\lambda_{i,\tau}^{(m)}}{\theta_{i,\tau}^{(m)}} = 1+ o_P\left(\frac{1}{\sqrt{T}}\right),
	\end{align*}
	as $T,p \to \infty$, where we stress the fact that $\gamma_{i,\tau}^{(m)}$ and $v_{i,\tau}^{(m)}$ are constant when $T,p \to \infty$.
	Therefore, (\ref{3:ap1}) reduces to
	\begin{align*}
		\frac{\lambda_{i,\tau}^{(1)}-\lambda_{i,\tau}^{(2)}}{\theta_{i,\tau}}
		= \frac{2+2c}{2+c} \left(  1+ o_P\left(\frac{1}{\sqrt{T}}\right) \right) - \frac{2}{2+c} \left( 1+ o_P\left(\frac{1}{\sqrt{T}}\right)  \right) 
		= \frac{2c}{2+c} + o_P\left(\frac{1}{\sqrt{T}} \frac{2c}{2+c}\right),
	\end{align*}
	for $T,p \to \infty$, and we conclude that
	\begin{align*}
		\sqrt{T}\frac{\lambda_{i,\tau}^{(1)}-\lambda_{i,\tau}^{(2)}}{\theta_{i,\tau}}
		= \sqrt{T} \frac{2c}{2+c} + o_P\left(\frac{2c}{2+c}\right),
	\end{align*}
	when $T,p \to \infty$.
	
	Consequently, when $T,p \to \infty$, $Z_{i,\tau} \rightarrow \infty$  as long as $\sqrt{T} \frac{2c}{2+c} \to \infty$ and $\lambda_{i,\tau}^{(1)} \ne \lambda_{i,\tau}^{(2)}$. And it is sufficient to show the assertion in this theorem.
\end{proof}

\section{Additional Simulations on the Power of Autocovariance Test}
\subsection{The Impact of Autocorrelation on the Power of Autocovariance Test}
In this section, we study the impact of the autocorrelations on the power of autocovariance test. To be specific, we consider exactly the same setup as in Section~\ref{3:sec:4} of the main paper, except that we keep the variance of factors the same across two different factor models but
set the AR coefficients for the second population to be 
$\phi_{1}^{(2)} = 0.9\phi_{1}^{(1)}, 0.8\phi_{1}^{(1)}, 0.7\phi_{1}^{(1)},0.6\phi_{1}^{(1)}, 0.5\phi_{1}^{(1)},$ respectively.
By doing that, we can investigate how the empirical powers of the autocovariance test are affected by the difference between autocorrelations of factors in two factor models.  

Similar to Section~\ref{3:sec:4}, for each combination of $T,N$ and $\delta$, two high-dimensional time series observations are generated 
first. We then follow the same estimation and testing procedures in Section~\ref{3:sec:3} and compute the test statistic $\widetilde{Z}_{i,\tau}$ by (\ref{3:teststat}), where again $B=500$ bootstrap samples are generated to find $\widetilde{\theta}_{i,\tau}^{(m)\ast}$, $\widetilde{v}_{i,\tau}^{(m)\ast}$, and $\widetilde{\gamma}_{i,\tau}^{(m)\ast}$ for both samples with the number of factors assumed to be known (i.e., $\widetilde{K}_{m} = 1$). Lastly, based on $M=500$ Monte Carlo simulations, the empirical powers of a one-sided autocovariance test for $i=1,$ $\tau=1$, and $\alpha=0.1$ can be estimated by 
the empirical probability that $\widetilde{Z}_{1,1}$ is greater than $z_{1-\alpha}$, i.e.,
\begin{align*}
	\frac{1}{M} \sum_{m=1}^M \mathbf{1}_{\left\{\widetilde{Z}_{1,1}(m) > z_{1-\alpha}\right\}},
\end{align*}
where we have assumed 
$\mu_{1,1}^{(1)} > \mu_{1,1}^{(2)}$ for various choices of $\phi_{1}^{(2)}$.

\begin{figure}[!htbp]
	\centering
	\includegraphics[width=0.75\textwidth]{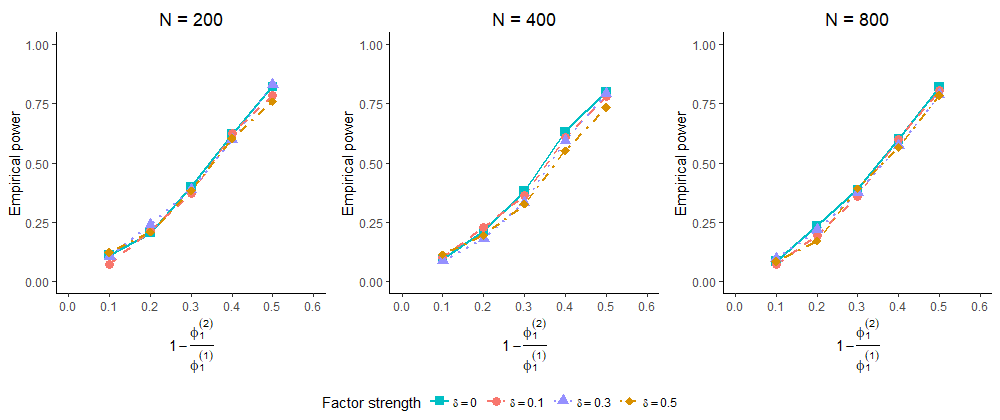}
	\caption{Empirical powers of the autocovariance test with $T=400$, $N=200,400,800,$ and $\delta = 0, 0.1, 0.3, 0.5.$}\label{3:f4}
\end{figure}

\begin{figure}[!htbp]
	\centering
	\includegraphics[width=0.75\textwidth]{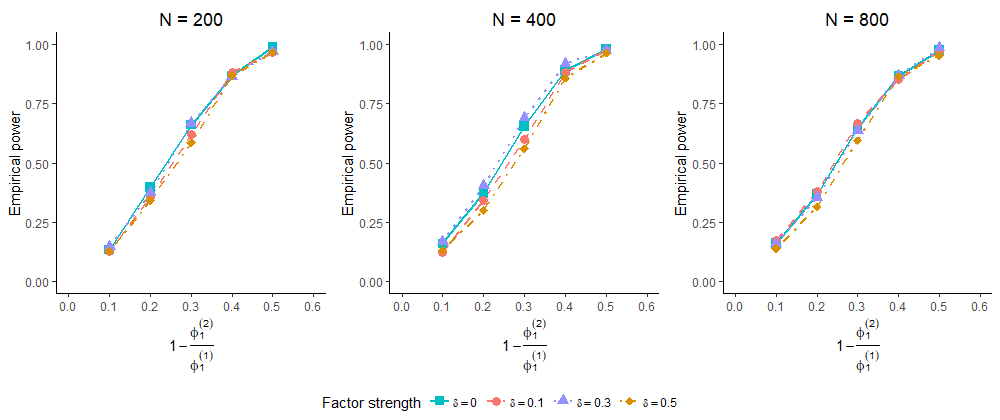}
	\caption{Empirical powers of the autocovariance test with $T=800$, $N=200,400,800,$ and $\delta = 0, 0.1, 0.3, 0.5.$}\label{3:f5}
\end{figure}

As presented in Figure~\ref{3:f4}~and~\ref{3:f5}, for all ratios of $N$ and $T$, empirical powers increase towards $1$, while $\phi_1^{(2)}$ drops from $0.9\phi_1^{(1)}$ to $0.5\phi_1^{(1)}$. As a consequence, Figure~\ref{3:f4}~and~\ref{3:f5} suggest that the autocovariance test can correctly reject the null hypothesis when two high-dimensional time series have different temporal autocorrelations $\phi_1^{(2)} \ne \phi_1^{(1)}$. However, unlike the case for the impact of variance, empirical powers of the one-sided autocovariance test for relatively weak factor models with large $\delta$, especially $\delta = 0.5$, are slightly lower than those of relatively strong factor models with small $\delta$. In other words, compared with strong factor models (i.e. high factor strength), the autocovariance test for weak factor models is slightly less potent in detecting the same proportional changes in autocorrelations of factors for two different factor models.

\subsection{The Impact of Temporal-dependent Noises on the Power of Autocovariance Test}
In this section, we study the impact of temporal-dependent noises on the power of autocovariance test. To be specific, we consider exactly the same setup as in Section~\ref{3:sec:4} of the main paper, except that we change the DGP of $\{\epsilon_{j,t}\}$ from i.i.d.\ $\mathcal{N}(0,1)$ to the same AR(1) processes as $\{f_{1,t}^{(m)}\}$ but with various variances. In that sense, the error components are temporal-dependent and we are studying the case of ``$1$ large spikes $+(N-1)$ small/medium eigenvalues" in the symmetrized autocovariance matrix. To explore the impact of temporal-dependent noise, we set the standard deviation of $\{\epsilon_{j,t}\}$ to be $1,2$ and $5$, respectively. 

\begin{figure}[!htbp]
	\centering
	\includegraphics[width=0.75\textwidth]{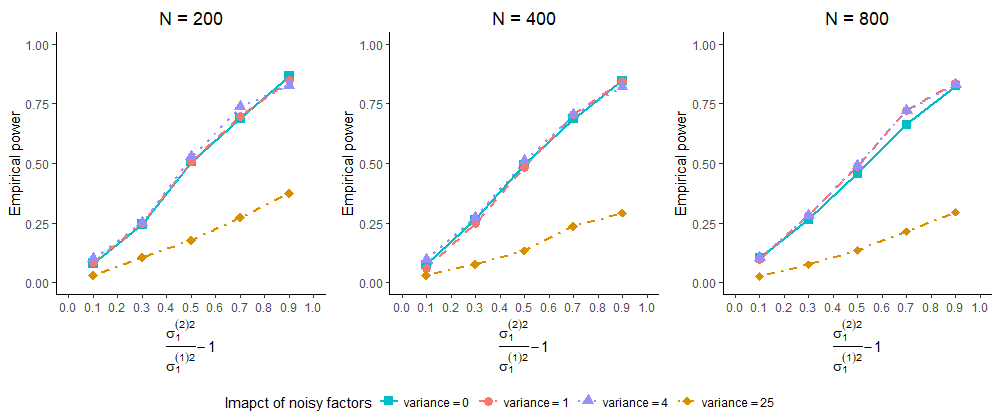}
	\caption{Empirical powers of the autocovariance test with $T=400$, $N=200,400,800,$ and $\sigma_\epsilon = 0, 1, 2, 5.$}\label{pic_error}
\end{figure}

Empirical powers of the autocovariance test with $T=400$ and various choices of $N$ are presented in Figure~\ref{pic_error}. When the autocovariances of noise components are relatively small ($\sigma_\epsilon = 1,2$), i.e. in the case of ``$1$ large spikes $+(N-1)$ small eigenvalues" in the symmetrized autocovariance matrix, the power of autocovariance test is not much different from the case with i.i.d.\ errors. However, when the autocovariance of noise increases ($\sigma_\epsilon = 5$), i.e. in the case of ``$1$ large spikes $+(N-1)$ medium eigenvalues", the power drops since the estimation of sample autocovariance matrices starts to suffer from the ``curse of dimensionality''.

\subsection{The Impact of Time Lags on the Power of Autocovariance Test}
In this section, we study the impact of the choice of time lags on the power of autocovariance test. To be specific, we consider exactly the same setup as in Section~\ref{3:sec:4} of the main paper, except that we perform our test on various choices of time lag $\tau$. To explore the impact of time lag $\tau$, we perform the test for $\tau = 1,2,3$ and $5$, respectively. 

\begin{figure}[!htbp]
	\centering
	\includegraphics[width=0.75\textwidth]{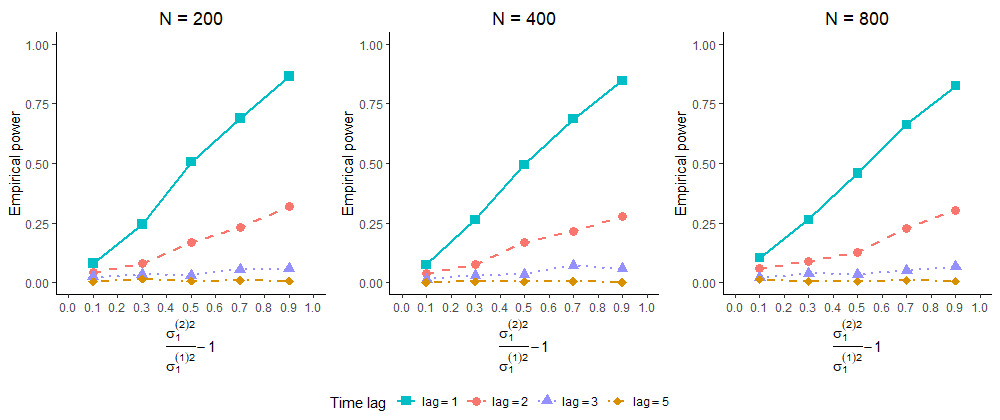}
	\caption{Empirical powers of the autocovariance test with $T=400$, $N=200,400,800,$ and $\tau = 1, 2, 3, 5.$}\label{pic_lag}
\end{figure}

Empirical powers of the autocovariance test with $T=400$ and various choices of $N$ are presented in Figure~\ref{pic_lag}. When the time lag $\tau$ increases, the power of the autocovariance test drops. This is mainly because when $\tau$ is relatively large, the autocorrelation of the factors, which are generated with AR($1$) and $\phi_1=0.5$, becomes weak so that the spikiness of the first eigenvalue drops accordingly.

\section{Additional Study on Hierarchical Clustering for Multi-country Mortality Data }\label{Appendix_app}

\subsection{Comparing Projection Matrices for Multi-country Mortality Data with One Factor}
As discussed in Remark~\ref{3:re:1}, the test of eigenvectors and eigenspace is of interest to many applications as discussed by a few works such as \citet{FanFanHanLv2022} and \citet{silin2020hypothesis}, and we notice that \citet{tang_clustering_2020} proposed a clustering method for multi-country mortality data using eigen functions under a functional time series setup. 

Therefore, in this section, we include a comparison of the mortality rate projection matrices for countries with one spiked factor to study the similarity of the mortality data.

For each country, the projection matrix is computed as the sum of eigenvectors corresponding to spiked eigenvalues multiplied by their transpose, i.e. $\sum_i^K v_iv_i^\top$, where $v_i$ is the eigenvector corresponding to the $i$-th largest eigenvalue of the symmetrized autocovariance matrix. We measure the dissimilarity between each pair of countries by the $\ell_2$ norm of the difference between two projection matrices. The $\ell_2$ norm takes values between $0$ (for exactly the same projection matrix) and $1$ (for projections corresponding to orthonormal eigenvectors). The results are presented in Figure~\ref{fprojection}.

\begin{figure}[!htbp]
	\centering
	\includegraphics[width=0.7\textwidth]{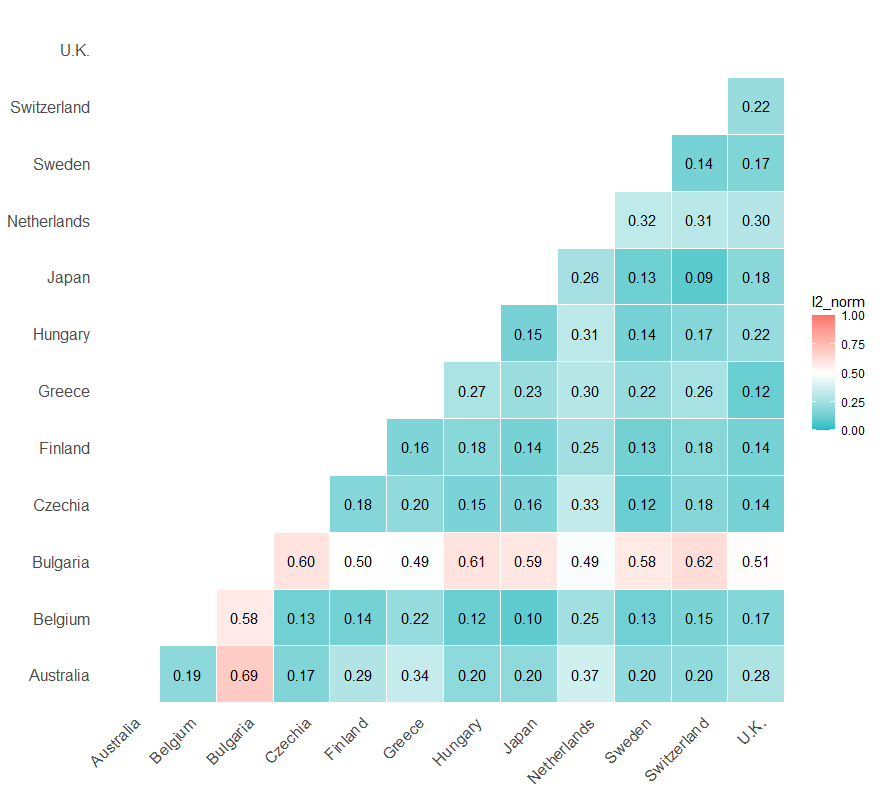}
	\caption{$\ell_2$ norm of the difference between the projection matrices for countries with one factor}\label{fprojection}
\end{figure}

Indeed, for most countries (except for Bulgaria) with one factor, the dissimilarity is relatively small. The similarity in projection matrices reveals that common human characteristics can be extracted from different countries. Meanwhile, this similarity does not guarantee that common human characteristics are of the same importance in leading the trend of human mortality rates in different countries. This motivates our proposed test on the equivalence of the largest eigenvalues since it can tell whether common human characteristics (compared to the country-specific features such as socioeconomic conditions) are of the same strength in affecting the mortality rates in different countries. Similar discussions on the roles of both human characteristics and country-specific features in mortality forecasting are also seen in \citet{LiLee2005}.

\subsection{Hierarchical Clustering for Multi-country Mortality Data with More than One Factors}
For countries with more than one factor, to test on the equivalence of autocovariances through factor models, test statistics between each pair of countries are computed for all three factors as $\widetilde{Z}_{1,1}, \widetilde{Z}_{2,1}$ and $\widetilde{Z}_{3,1}$. As depicted in Figure~\ref{3:f8}, the $p$-values for $\widetilde{Z}_{1,1}$ and $\widetilde{Z}_{2,1}$ between all pairs of countries are relatively large, which suggests that the differences of the first two factors between each pair of countries are not significant (at $\alpha = 0.1$). Nonetheless, $p$-values for $\widetilde{Z}_{1,3}$ are relatively small between Canada and France, Canada and Italy, and Italy and Portugal. 
As a result, despite that $p$-value is $0.09$ for $\widetilde{Z}_{1,3}$ between Italy and Portugal, one may suggest considering France, Italy, and Portugal have similar spiked eigenvalues of their autocovariance matrices in a three-factor model and include them in a combined statistical analysis while leaving Canada for an independent analysis.
To measure the dissimilarity of mortality data between two countries with more than one factors, the overall distance between two countries can be defined as a weighted average of the distances for all factors. In specific, for each pair of countries, we can define the distance for the $i$-th factor as $dist_i = 1 - p_i$, where $p_i$ is the $p$-value of the autocovariance test computed using the $i$-th factor of both countries. It is then straightforward to compute the overall distance between this pair of countries as 
$$dist = \sum_{i=1}^{K} w_i\cdot dist_i,$$ 
where $w_i$ is a weight on $dist_i$. Practically, we suggest that the weight $w_i$ is related to the magnitude of each singular value of the autocovariance matrix (or equivalently the squared root of the eigenvalues of symmetrized autocovariance matrix), since the singular values are related to the autocovariance explained by each factor. Based on this idea, we compute $w_i$ as 
$$w_i = (w_{i}^{(1)} + w_{i}^{(2)})/2,$$ 
where $w_{i}^{(1)} = \sqrt{\tilde \lambda_{i,\tau}^{(1)}} / \left(\sum_{i=1}^{K} \sqrt{\tilde \lambda_{i,\tau}^{(1)}}\right)$ and $w_{i}^{(2)} = \sqrt{\tilde \lambda_{i,\tau}^{(2)}} / \left(\sum_{i=1}^{K} \sqrt{\tilde \lambda_{i,\tau}^{(2)}}\right)$. The result of hierarchical clustering using average linkage for all countries with three factors is presented in Figure~\ref{3:f11}.   

\begin{figure}[!htbp]
	\centering
	\includegraphics[width=0.65\textwidth]{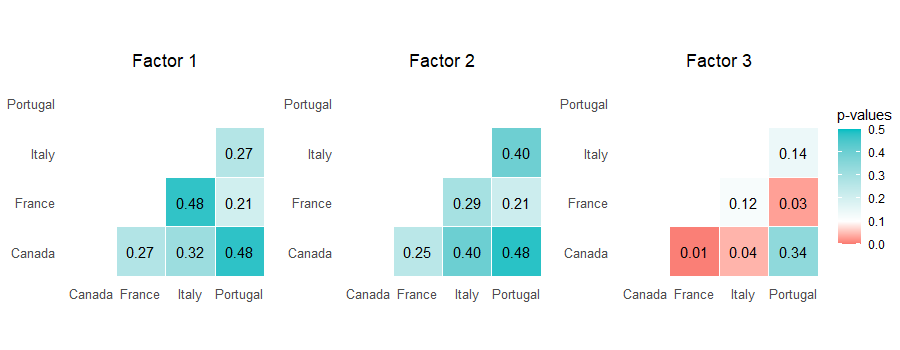}
	\caption{$p$-values of the autocovariance test for each pair of countries that have three factors in the estimated factor model}\label{3:f8}
\end{figure}

\begin{figure}[!htbp]
	\centering
	\includegraphics[width=0.55\textwidth]{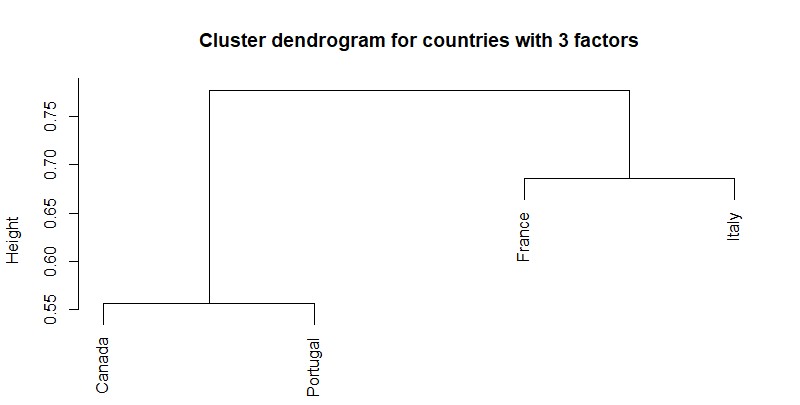}
	\caption{Cluster dendrogram for countries that have three factors in the estimated factor model}\label{3:f11}
\end{figure}

However, as discussed in Remark \ref{3:re:1}, even the number of factors across some countries are different, it may still be of interest to perform the test for the first factor only. The idea behind is also very straightforward, that is we can study whether the mortality rates across all countries have the same low-dimensional representations in the eigenspace spanned by the first eigenvector shared by all countries. In this sense, we perform autocovariance tests and the hierarchical clustering analysis on the first factor for all countries regardless of the estimated total number of factors. The result of $p$-values computed for the first factors are illustrated in Figure~\ref{3:f9} and the result of hierarchical clustering using average linkage for the first factor is presented in Figure~\ref{3:f12}. 
As seen in Figure~\ref{3:f9}, in addition to what has been discussed for those countries with only one factor in their factor models, the first factor of Portugal and France are also different from the first factor of Bulgaria and Finland, respectively.  Consequently, despite the differences between the estimated numbers of factors for  Canada, Denmark, Italy, Poland, and all other countries, the total death rates projected in the eigenspace spanned by the first common eigenvector are not significantly different across these countries.


\begin{figure}[!htbp]
	\centering
	\includegraphics[width=0.7\textwidth]{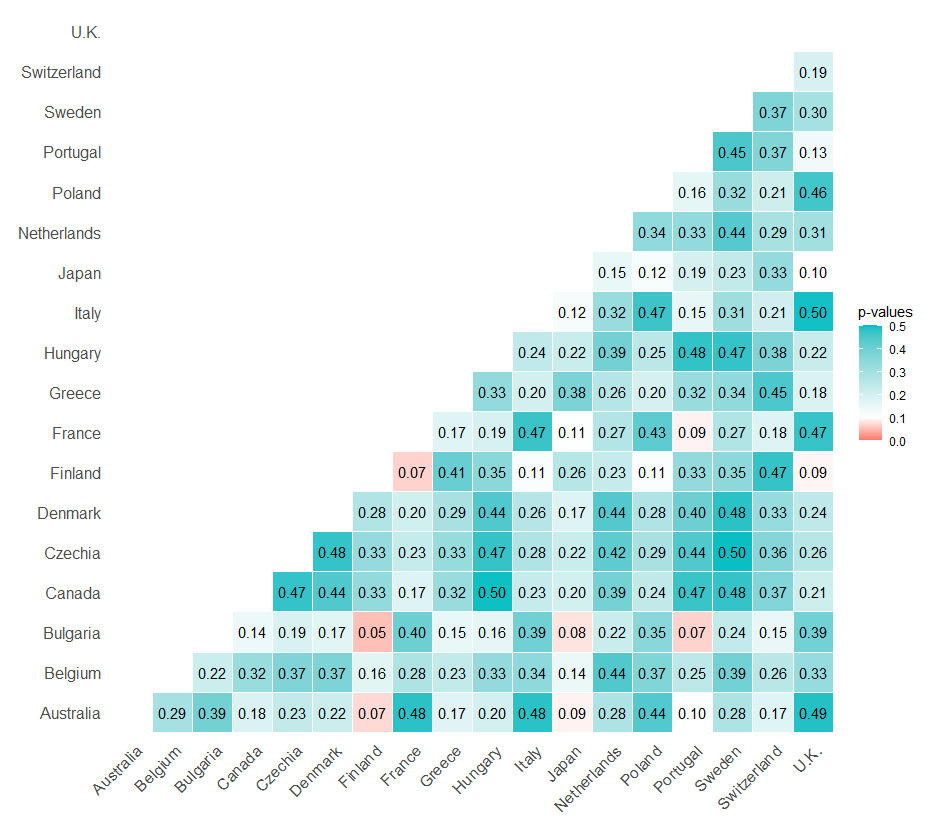}
	\caption{$p$-values of the autocovariance test of the first factor for all countries}\label{3:f9}
\end{figure}

\begin{figure}[!htbp]
	\centering
	\includegraphics[width=0.7\textwidth]{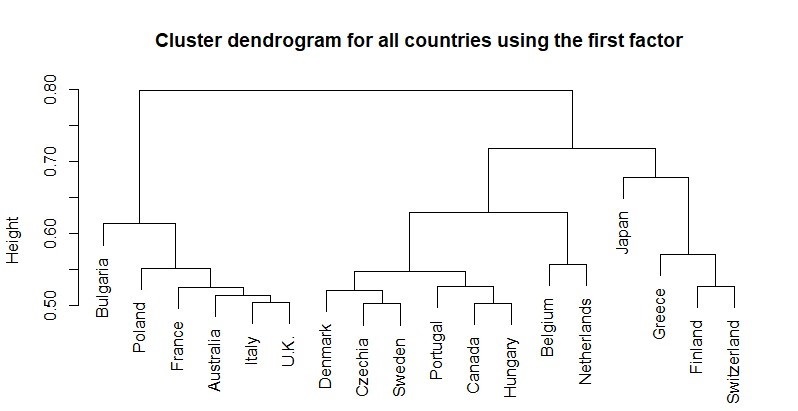}
	\caption{Cluster dendrogram for all countries using the first factor}\label{3:f12}
\end{figure}

\section{Simulation Studies on Gaussian Assumption in Theorem~\ref{theorem - CLT} }\label{Appendix_sim_clt}

In this section, we conduct simulations to show that the results of Theorem~\ref{theorem - CLT} can still be obtained when we replace the Gaussian error $\epsilon_{it}$ by Gamma or Student's t distributed errors. The simulations considered in this section are based on the following factor models
\begin{align*}
	\yy_t
	=
	\begin{pmatrix}
		\mathrm{diag}(\sigma_1, \ldots, \sigma_K) \\ \boldsymbol{0}_{p\times K}
	\end{pmatrix}
	\ff_t + \epsilonb_t
	\numberthis
\end{align*}
where we choose $T=4000,\ p=400,\ K=2$, $\sigma_1=50$, $\sigma_2=10$, and the factors $\ff_t$ are generated by MA($2$) processes
\begin{align*}
	f_{it} = \sum_{l=0}^{2} \phi_{il} z_{i, t-l}, \quad i = 1, 2, \quad t = 1, \ldots, T,
\end{align*}
where $\phib_1 = (5,3,1)$ and $\phib_2 = (1,-2,1)$ with variances of factors normalized to 1. The errors $\epsilonb_t$ are generated i.i.d.\ from three different distributions:(1) Gaussian i.e. $\mathcal{N}(0,1)$; (2) Gamma$(2,2)$ with mean and variance normalized to 0 and 1, respectively; (3) Student's t with d.f. $= 4$ with mean and variance normalized to 0 and 1, respectively. By generating $M=5000$ Monte Carlo simulations, we draw the histograms and density plots of $\lambda_{1,1}$ and $\lambda_{2,1}$ for the above three cases (subject to the same scaling) as shown in Figure~\ref{sim_nonG1}~to~\ref{sim_nonG2}, where the differences among the histograms and density plots are rather small.

\begin{figure}[!htbp]
	\centering
	\includegraphics[width=0.7\textwidth]{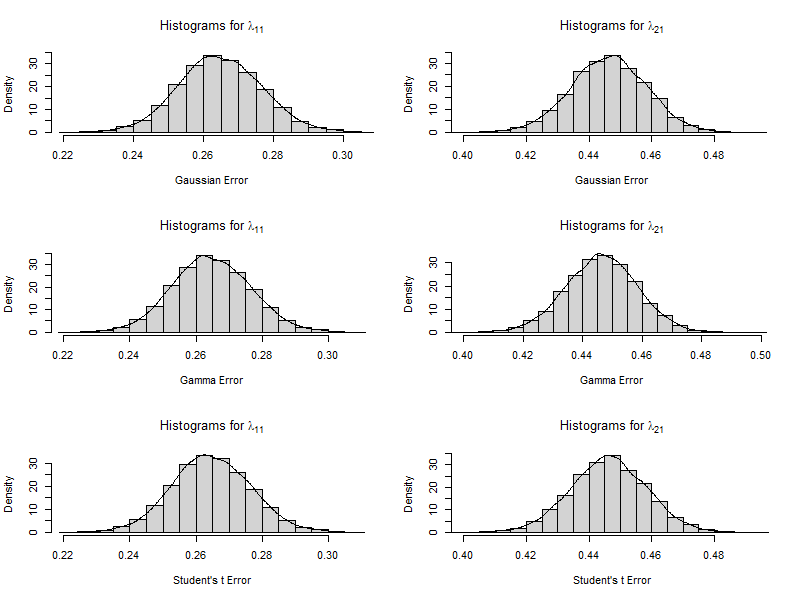}
	\caption{Histograms of $\lambda_{1,1}$ and $\lambda_{2,1}$ (subject to a rescaling) under three different assumptions of $\epsilonb_t$}\label{sim_nonG1}
\end{figure}

\begin{figure}[!htbp]
	\centering
	\includegraphics[width=0.7\textwidth]{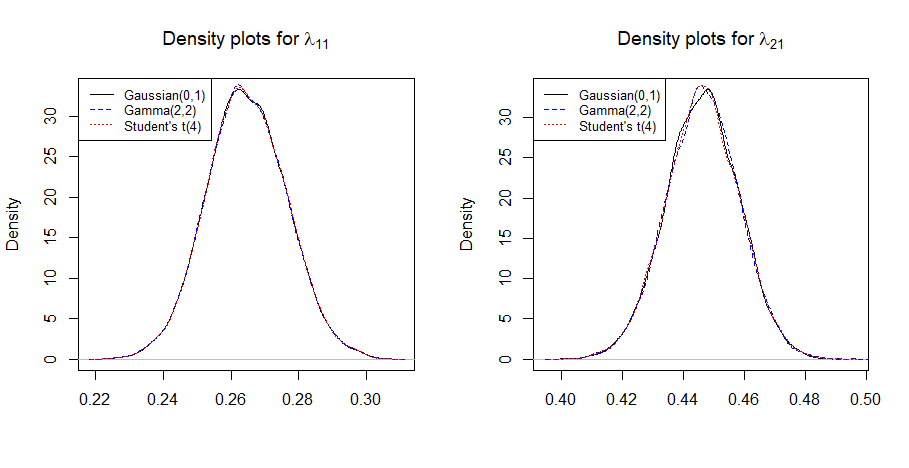}
	\caption{Density plots of $\lambda_{1,1}$ and $\lambda_{2,1}$ (subject to a rescaling) under three different assumptions of $\epsilonb_t$}\label{sim_nonG2}
\end{figure}
				\end{appendices}

	\end{document}